\documentclass[a4paper]{amsart}

\usepackage{amscd,amsthm,amsfonts,amssymb,amsmath}
\usepackage[colorlinks=true]{hyperref}
\usepackage{fullpage}
\usepackage[all]{xy}
\usepackage{mathrsfs}
\usepackage{bm}
\usepackage{tikz}
\usepackage{bbm}
\usepackage{graphics}
\usepackage{enumerate}

\newcommand{\BA}{{\mathbb {A}}}\newcommand{\BC}{{\mathbb {C}}}
\newcommand{\BG}{{\mathbb {G}}}\newcommand{\BH}{{\mathbb {H}}}

\newcommand{\BQ}{{\mathbb {Q}}}\newcommand{\BR}{{\mathbb {R}}}

\newcommand{\BZ}{{\mathbb {Z}}}

\newcommand{\bfA}{{\mathbf {A}}}

\newcommand{\CC}{{\mathcal {C}}}
\newcommand{\CH}{{\mathcal {H}}}
\newcommand{\CI}{{\mathcal {I}}}
\newcommand{\CO}{{\mathcal {O}}}\newcommand{\CP}{{\mathcal {P}}}
\newcommand{\CS}{{\mathcal {S}}}

\newcommand{\msl}{\mathscr{L}}

\newcommand{\fa}{{\mathfrak{a}}} \newcommand{\fc}{{\mathfrak{c}}} 
 \newcommand{\fg}{{\mathfrak{g}}} \newcommand{\fh}{{\mathfrak{h}}}
  \newcommand{\fl}{{\mathfrak{l}}}
\newcommand{\fm}{{\mathfrak{m}}} \newcommand{\fn}{{\mathfrak{n}}}\newcommand{\fo}{{\mathfrak{o}}} 
\newcommand{\fq}{{\mathfrak{q}}} \newcommand{\fr}{{\mathfrak{r}}}\newcommand{\fs}{{\mathfrak{s}}} 
 \newcommand{\fv}{{\mathfrak{v}}}

 \newcommand{\fS}{{\mathfrak{S}}}

                      \newcommand{\Ad}{{\mathrm{Ad}}}              \newcommand{\ad}{{\mathrm{ad}}}

\newcommand{\adeles}{ad\`{e}les~}

                      \newcommand{\bs}{\backslash}

                      		\newcommand{\Cent}{\mathrm{Cent}}
\newcommand{\capa}{{\mathrm{capa}}}

\newcommand{\End}{{\mathrm{End}}}

\newcommand{\Hom}{{\mathrm{Hom}}}

\newcommand{\ideles}{id\`{e}les~}

\newcommand{\ov}{\overline}

                  \newcommand{\ra}{\rightarrow}

                                  \newcommand{\Res}{{\mathrm{Res}}}

 \newcommand{\sk}{\medskip}                      \newcommand{\s}{\sk\noindent}

\newcommand{\Trd}{{\mathrm{Trd}}}

\newcommand{\vol}{{\mathrm{vol}}}                          \newcommand{\val}{{\mathrm{val}}}

\newcommand{\wt}{\widetilde}                        \newcommand{\wh}{\widehat}

\newtheorem{thm}{Theorem}[section]
\newtheorem{coro}[thm]{Corollary}
\newtheorem{lem}[thm]{Lemma}
\newtheorem{prop}[thm]{Proposition}

\newtheorem{defn}[thm]{Definition}

\theoremstyle{definition}

\theoremstyle{remark}
\newtheorem{remark}[thm]{Remark}

\numberwithin{equation}{subsection}
\def\mat(#1,#2,#3,#4){
  \begin{pmatrix}
  #1 & #2 \\ #3 & #4
  \end{pmatrix}
}

\begin{document}
\title{An infinitesimal variant of Guo-Jacquet trace formula II}
\author{Huajie LI}
\date{\today}
\maketitle

\begin{abstract}
We establish an infinitesimal variant of Guo-Jacquet trace formula for the case of a central simple algebra over a number field $F$ containing a quadratic field extension $E/F$. It is an equality between a sum of geometric distributions on the tangent space of some symmetric space and its Fourier transform. To prove this, we need to define an analogue of Arthur's truncation and then use the Poisson summation formula. We describe the terms attached to regular semi-simple orbits as explicit weighted orbital integrals. To compare them to those for another case studied in our previous work, we state and prove the weighted fundamental lemma at the infinitesimal level by using Labesse's work on the base change for $GL_n$. 
\end{abstract}

\tableofcontents


\section{\textbf{Introduction}}

Guo and Jacquet have proposed a conjecture \cite{MR1382478} in order to generalize Waldspurger's famous result \cite{MR783511}, which relates toric periods and central values of automorphic $L$-functions for $GL_2$, to higher ranks. The approach of relative trace formulae makes it possible to reduce the conjectural comparison of periods (related to the spectral side) to the comparison of (weighted) orbital integrals (related to the geometric side) on different symmetric spaces. This approach was first adopted by Jacquet \cite{MR868299} to reprove Waldspurger's theorem. For higher ranks, Feigon-Martin-Whitehouse \cite{MR3805647} obtained some partial results using a simple form of relative trace formulae. For the comparison of local orbital integrals, Guo reduced the fundamental lemma \cite{MR1382478} to that of the base change for $GL_n$ and Zhang proved the smooth transfer \cite{MR3414387} by global methods. 

However, an obstruction in the approach is the divergence of sums of integrals in both sides of relative trace formulae. Such a problem has already existed in the classical Arthur-Selberg trace formula and Arthur introduced a truncation process \cite{MR518111}\cite{MR558260} to tackle it (see also \cite{MR1893921} for its Lie algebra variant). We start working at the infinitesimal level (namely the tangent space of a symmetric space) for a couple of reasons. Firstly, our truncation for the tangent space is expected to be adapted to a truncation for the symmetric space. Secondly, infinitesimal trace formulae should be useful for the proof of results on the transfer (see Zhang's work \cite{MR3414387} on the ordinary orbital integrals). 

Guo-Jacquet trace formulae concern two symmetric pairs. The first one is associated to a pair denoted by $(G',H')$, where $G':=GL_{2n}$ and $H':=GL_n\times GL_n$ are reductive groups over a number field $F$ and $H'$ embeds diagonally in $G'$. Let $\fs'\simeq \fg\fl_n\oplus\fg\fl_n$ be the tangent space at the neutral element of the symmetric space $G'/H'$. We have established an infinitesimal trace formula in \cite{2019arXiv190407102L} for the action of $H'$ on $\fs'$ by conjugation. The second one, associated to a pair denoted by $(G,H)$, is the main object in this paper. Before introducing it, we remark that we shall work in a more general setting than the original one. The reason is that the converse direction of Guo-Jacquet conjecture was originally proposed only for $n$ odd.  In our searching for an analogue for $n$ even, the related local conjecture of Prasad and Takloo-Bighash \cite[Conjecture 1]{MR2806111} suggests that we should consider more inner forms of $G'$. Some recent progress on this local conjecture has been made by Xue \cite{Xue} with the help of a simple form of global relative trace formulae. 

Let $E/F$ be a quadratic extension of number fields. Suppose that $E=F(\alpha)$, where $\alpha\in E$ and $\alpha^2\in F$. Let $\fg$ be a central simple algebra over $F$ containing $E$. Write $\fh$ to be the centralizer of $\alpha$ in $\fg$. Denote by $G$ and $H$ the groups of invertible elements in $\fg$ and $\fh$ respectively. Both of them are viewed as reductive groups over $F$. Let $\fs:=\{X\in\fg: \Ad(\alpha)(X)=-X\}$, where $\Ad$ denotes the adjoint action of $G$ on $\fg$. It can be identified with the tangent space at the neutral element of the symmetric space $G/H$. The main global result in this paper is an infinitesimal trace formula for the action of $H$ on $\fs$ by conjugation. 

Denote by $\BA$ the ring of \adeles of $F$ and by $H(\BA)^1$ the subset of elements in $H(\BA)$ with absolute-value-$1$ reduced norm. We define a relation of equivalence on $\fs(F)$: two elements of $\fs(F)$ are equivalent if and only if they lie in the same fibre over the categorical quotient $\fs//H$. Denote by $\CO$ the set of classes of equivalence. Let $f$ be a Bruhat-Schwartz function on $\fs(\BA)$. For each $\fo\in\CO$ and $x\in H(F)\bs H(\BA)$, define
$$ k_{f,\fo}(x):=\sum_{Y\in\fo} f(\Ad(x^{-1})(Y)). $$
As mentioned, we are facing the problem that
$$ \sum_{\fo\in\CO} \int_{H(F)\bs H(\BA)^1} k_{f,\fo}(x)dx $$
is divergent. We define the truncation $k_{f,\fo}^T(x)$ (see (\ref{deftrunc2.1})) which is an analogue of Arthur's truncation in \cite{MR518111}, where $T$ is a truncation parameter in some cone $T_+ +\fa_{P_0}^+$ of the coroot space of $H$, such that the following theorem holds. 

\begin{thm}[see Theorem \ref{convergence2}]
For all $T\in T_+ +\fa_{P_0}^+$, 
$$ \sum_{\fo\in\CO} \int_{H(F)\bs H(\BA)^1} k_{f,\fo}^T(x) dx $$
is absolutely convergent. 
\end{thm}

We also know the behaviour of each term (viewed as a distribution) with respect to the truncation parameter. It is even simpler than that in the case of $(G',H')$ (cf. \cite[Theorem 1.2]{2019arXiv190407102L}). 

\begin{thm}[see Corollary \ref{polynomial2}]
For all $T\in T_+ +\fa_{P_0}^+$ and $\fo\in\CO$, define 
$$ J_\fo^{T}(f):=\int_{H(F)\bs H(\BA)^1} k_{f,\fo}^T(x) dx. $$
Then the function $T\mapsto J_\fo^{T}(f)$ is the restriction of a polynomial in $T$. 
\end{thm}

Now we can take the constant term of each term to eliminate the truncation parameter $T$. Denote by $J_\fo(f)$ the constant term of $J_\fo^T(f)$. These distributions are not invariant by $H(\BA)^1$ (see Proposition \ref{noninvariance2}), but we can write the regular semi-simple terms as explicit weighted orbital integrals with the same weights as Arthur's in \cite{MR518111}. 

\begin{thm}[see Theorem \ref{woi2}]
Let $\fo\in\CO$ be a regular semi-simple orbit, $P_1$ a standard parabolic subgroup of $H$ and $Y_1\in\fo$ an elliptic element with respect to $P_1$ (see the precise definition in Section \ref{wtorbint}). Denote by $H_{Y_1}$ the centralizer of $Y_1$ in $H$. We have
$$ J_\fo(f)=\vol(A_{P_1}^\infty H_{Y_1}(F)\bs H_{Y_1}(\BA))\cdot \int_{H_{Y_1}(\BA)\bs H(\BA)} f(\Ad(x^{-1})(Y_1)) v_{P_1}(x) dx, $$
where $A_{P_1}^\infty$ is defined in Section \ref{hcmap}, and $v_{P_1}(x)$ is the volume of some convex hull. Here we choose compatible measures on $A_{P_1}^\infty H_{Y_1}(F)\bs H_{Y_1}(\BA)$, $H_{Y_1}(\BA)\bs H(\BA)$ and the convex hull associated to $v_{P_1}(x)$ such that $J_\fo(f)$ is independent of these choices. 
\end{thm}

Thanks to the truncation, we solve the divergence issue in the following infinitesimal trace formula. It is a consequence of  the Poisson summation formula. 

\begin{thm}[see Theorem \ref{infitf2}]
We have the equality
$$ \sum_{\fo\in\CO} J_\fo(f)=\sum_{\fo\in\CO} J_\fo(\hat{f}), $$
where $\hat{f}$ (see (\ref{fouriertransform2})) is the Fourier transform of $f$. 
\end{thm}

Notice that the symmetric pairs $(G,H)$ and $(G',H')$ are the same after the base change to an algebraic closure of $F$ containing $E$. In fact, the truncation and most proofs of the global results above are simpler than those in \cite{2019arXiv190407102L} in some sense. The simplicity results from the equality $H(\BA)^1=H(\BA)\cap G(\BA)^1$ here, where $G(\BA)^1$ denotes the subset of elements in $G(\BA)$ with absolute-value-$1$ reduced norm. Moreover, there is a bijection between the set of standard parabolic subgroups in $H$ and the set of semi-standard parabolic subgroups in $G$ whose intersection with $H$ is a standard parabolic subgroup in $H$. One may consult Section \ref{explicitdescription} for more details. However, there are still some rationality issues. We shall give sufficient details and self-contained proofs here for completeness. 

At the end of this paper, we hope to provide some new evidence of noninvariant comparison of Guo-Jacquet trace formulae. We shall turn to the local setting with $F$ denoting a local field. In the comparison of geometric sides, an important case is the so-called fundamental lemma. It roughly says that some basic functions for two symmetric pairs should have associated local orbital integrals on matching orbits at almost all unramified places. Guo \cite{MR1382478} proved it for the units of spherical Hecke algebras for Guo-Jacquet trace formulae with the help of the base change fundamental lemma for the full spherical Hecke algebras for $GL_n$ known by Kottwitz \cite[Lemma 8.8]{MR564478} and Arthur-Clozel \cite[Theorem 4.5 in Chapter 1]{MR1007299}. An infinitesimal version \cite[Lemma 5.18]{MR3414387} was used by Zhang to prove the smooth transfer for Guo-Jacquet trace formulae following the same philosophy of Waldspurger's work \cite{MR1440722} on the endoscopic transfer. We would like to generalize \cite[Lemma 5.18]{MR3414387} in the weighted context. 

For almost all unramified places, $(G, H)$ is isomorphic to $(GL_{2n}, \Res_{E/F} GL_{n,E})$ and $\fs(F)\simeq \fg\fl_n(E)$. Denote by $\CO_F$ (resp. $\CO_E$) the ring of integers in $F$ (resp. $E$). For $f$ and $f'$ a pair of locally constant and compactly supported complex functions on $\fs(F)$ and $\fs'(F)$ respectively, we define the notion of being ``strongly associated'' (see the precise definition in Definition \ref{defstrass}) inspired by \cite[Definition III.3.2]{MR1339717}. Roughly speaking, $f$ and $f'$ are said to be strongly associated if their local weighted orbital integrals are equal at matching orbits. Let $f_0$ and $f'_0$ be the characteristic functions of $\fs(\CO_F)\simeq \fg\fl_n(\CO_E)$ and $\fs'(\CO_F)\simeq (\fg\fl_n\oplus\fg\fl_n)(\CO_F)$ respectively. Because the weighted orbital integrals that we got share the same weights with those in twisted trace formulae (see Remark \ref{rmkwoi2} and \cite[Remark 9.3]{2019arXiv190407102L}), we are able to show the following result by using Labesse's work on the base change weighted fundamental lemma for the full spherical Hecke algebras for $GL_n$. 

\begin{thm} [see Theorem \ref{wfl}]
For almost all unramified places, $f_0$ and $f'_0$ are strongly associated. 
\end{thm}

\s{\textbf{Acknowledgement. }}I would like to thank my PhD advisor Pierre-Henri Chaudouard for suggesting considering a more general case than Guo-Jacquet's original one. I have benefited a lot from numerous discussions with him on this work. I would also like to thank Jean-Loup Waldspurger and the anonymous referee for valuable comments on my draft allowing me to correct several errors and improve some arguments. This work was supported by grants from R\'{e}gion Ile-de-France. 


\section{\textbf{Notation}}\label{notation}

We shall use $F$ to denote a number field in this article except for the last section where $F$ denotes a non-archimedean local field of characteristic $0$. 

\subsection{\textbf{Roots and weights}}

Let $F$ be a number field or a non-archimedean local field of characteristic $0$. Suppose that $H$ is a reductive group defined over $F$. Fix a minimal Levi $F$-subgroup $M_0$ of $H$. All the following groups are assumed to be defined over $F$ without further mention. We call a parabolic subgroup or a Levi subgroup of $H$ semi-standard if it contains $M_0$. Fix a minimal semi-standard parabolic subgroup $P_0$ of $H$. We call a parabolic subgroup $P$ of $H$ standard if $P_0\subseteq P$. For any semi-standard parabolic subgroup $P$ of $H$, we usually write $M_P$ for the Levi factor containing $M_0$ and $N_P$ the unipotent radical. Denote by $A_P$ the maximal $F$-split torus in the centre of $M_P$. Let $X(M_P)_F$ be the group of characters of $M_P$ defined over $F$. Then define
$$ \fa_P:=\Hom_\BZ(X(M_P)_F, \BR) $$
and its dual space
$$ \fa_P^*:=X(M_P)_F\otimes_\BZ \BR, $$
which are both $\BR$-linear spaces of dimension $\dim(A_P)$. Notice that the restriction $X(M_P)_F\hookrightarrow X(A_P)_F$ induces an isomorphism
$$ \fa_P^*\simeq X(A_P)_F\otimes_\BZ \BR. $$

Suppose that  $P_1\subseteq P_2$ is a pair of standard parabolic subgroups of $H$. The restriction $X(M_{P_2})_F\hookrightarrow X(M_{P_1})_F$ induces $\fa_{P_2}^*\hookrightarrow\fa_{P_1}^*$  and its dual map $\fa_{P_1}\twoheadrightarrow\fa_{P_2}$. Denote by $\fa_{P_1}^{P_2}$ the kernel of the latter map $\fa_{P_1}\twoheadrightarrow\fa_{P_2}$. The restriction $X(A_{P_1})_F\twoheadrightarrow X(A_{P_2})_F$ induces $\fa_{P_1}^*\twoheadrightarrow\fa_{P_2}^*$ and its dual map $\fa_{P_2}\hookrightarrow\fa_{P_1}$. The latter map $\fa_{P_2}\hookrightarrow\fa_{P_1}$  provides a section of the previous map $\fa_{P_1}\twoheadrightarrow\fa_{P_2}$. Thus we have decompositions
$$ \fa_{P_1}=\fa_{P_2}\oplus\fa_{P_1}^{P_2} $$
and
$$ \fa_{P_1}^*=\fa_{P_2}^*\oplus(\fa_{P_1}^{P_2})^*. $$
When $P_1=P_0$, we write $\fa_{P_1}$, $A_{P_1}$ and $\fa_{P_1}^{P_2}$ as $\fa_0$, $A_0$ and $\fa_0^{P_2}$ respectively. 

For a pair of standard parabolic subgroups $P_1\subseteq P_2$ of $H$, write $\Delta_{P_1}^{P_2}$ for the set of simple roots for the action of $A_{P_1}$ on $N_{P_1}^{P_2}:=N_{P_1}\cap M_{P_2}$. Notice that $\Delta_{P_1}^{P_2}$ is a basis of $(\fa_{P_1}^{P_2})^*$. Let
$$ (\wh{\Delta}_{P_1}^{P_2})^\vee:=\{\varpi_\alpha^\vee:\alpha\in\Delta_{P_1}^{P_2}\} $$
be the basis of $\fa_{P_1}^{P_2}$ dual to $\Delta_{P_1}^{P_2}$. 
One has the coroot $\beta^\vee$ associated to any $\beta\in\Delta_{P_0}^{P_2}$. For every $\alpha\in\Delta_{P_1}^{P_2}$, let $\alpha^\vee$ be the projection of $\beta^\vee$ to $\fa_{P_1}^{P_2}$, where $\beta\in\Delta_{P_0}^{P_2}$ whose restriction to $\fa_{P_1}^{P_2}$ is $\alpha$. 
Define
$$ (\Delta_{P_1}^{P_2})^\vee:=\{\alpha^\vee:\alpha\in\Delta_{P_1}^{P_2}\}, $$
which is a basis of $\fa_{P_1}^{P_2}$. Denote by
$$ \wh{\Delta}_{P_1}^{P_2}:=\{\varpi_\alpha:\alpha\in\Delta_{P_1}^{P_2}\} $$
the basis of $(\fa_{P_1}^{P_2})^*$ dual to $ (\Delta_{P_1}^{P_2})^\vee$. 

For a standard parabolic subgroup $P$ of $H$, set
$$ \fa_{P}^+:=\{T\in\fa_P: \forall \alpha\in\Delta_P^H, \alpha(T)>0\}. $$
For $P_1\subseteq P_2$ as above, define $\tau_{P_1}^{P_2}$ and $\wh{\tau}_{P_1}^{P_2}$ as the characteristic functions of
$$ \{T\in\fa_0: \forall \alpha\in\Delta_{P_1}^{P_2}, \alpha(T)>0\} $$
and
$$ \{T\in\fa_0: \forall \varpi\in\wh{\Delta}_{P_1}^{P_2}, \varpi(T)>0\} $$
respectively. 

\subsection{\textbf{The functions $H_P$ and $F^P$}}\label{hcmap}

Let $F$ be a number field. Let $\BA$ be the ring of \adeles of $F$ and let $|\cdot|_\BA$ be the product of normalized local absolute values on the group of \ideles $\BA^*$. Fix a maximal compact subgroup $K$ of $H(\BA)$ that is admissible relative to $M_0$ in the sense of \cite[p. 9]{MR625344}. In this paper, we choose such a $K$ as follows when $H(F)=GL_n(D)$, where $D$ is a central division algebra over a finite field extension $E$ of $F$. 
For every place $v$ of $E$, fix an isomorphism $D\otimes_{E}E_v\simeq \fg\fl_{r_v}(D_v)$, where $D_v$ is a central division algebra over $E_v$. Under this isomorphism, the completion at $v$ of $H(F)$ is $H_v\simeq GL_{n_v}(D_v)$, where $n_v=nr_v$. At every non-archimedean place $v$ of $E$ (see \cite[p. 191]{MR1344916}), let $K_v$ be $GL_{n_v}(\CO_{D_v})$, where $\CO_{D_v}$ is the ring of integers of $D_v$; at every archimedean place $v$ of $E$ (see \cite[p. 199]{MR1344916}), let $K_v$ be the orthogonal group, unitary group and hyperunitary group for $H_v\simeq GL_{n_v}(\BR)$, $GL_{n_v}(\BC)$ and $GL_{n_v}(\BH)$ respectively. Define $K:=\prod_{v} K_v$, which is a maximal compact subgroup of $H(\BA)$. 

Suppose that $P$ is a standard parabolic subgroup of $H$. Let $H_P$ be the homomorphism $M_P(\BA)\ra \fa_P$ given by
$$ \forall m\in M_P(\BA), \forall \chi\in X(M_P)_F, \langle H_P(m), \chi\rangle=\log(|\chi(m)|_\BA). $$
Write $M_P(\BA)^1$ for the kernel of $H_P$ and $A_P^\infty$ for the neutral component for the topology of $\BR$-manifolds of the group of $\BR$-points of the maximal $\BQ$-split torus in $\Res_{F/\BQ}A_P$. Then any element $x\in H(\BA)$ can be written as $x=nmak$, where $n\in N_P(\BA)$, $m\in M_P(\BA)^1$, $a\in A_P^\infty$ and $k\in K$. We can define a continuous map $H_P: H(\BA)\ra\fa_P$ by setting $H_P(x):=H_P(a)$ with respect to this decomposition. Notice that $H_P$ induces an isomorphism from $A_P^\infty$ to $\fa_P$. If $P\subseteq Q$ are a pair of semi-standard parabolic subgroups, write
$$ A_P^{Q, \infty}:=A_P^\infty\cap M_Q(\BA)^1. $$
Then $H_P$ also induces an isomorphism from $A_P^{Q, \infty}$ to $\fa_P^Q$. 

Denote by $\Omega^H$ the Weyl group of $(H, A_0)$. In the cases to be considered in this paper, for every $s\in\Omega^H$, we can always choose one representative $\omega_s\in H(F)\cap K$. In fact, we are dealing with the restriction of scalars of inner forms of $GL_n$, thus we can choose $\Omega^H$ to be permutation matrices. 

From the reduction theory (see \cite[p. 941]{MR518111}), we know that there exists a real number $t_0<0$ and a compact subset $\omega_{P_0}\subseteq N_{P_0}(\BA)M_0(\BA)^1$ such that for any standard parabolic subgroup $P$ of $H$, we have
$$ H(\BA)=P(F)\fS_{P_0}^P(\omega_{P_0}, t_0). $$
Here the Siegel set $\fS_{P_0}^P(\omega_{P_0}, t_0)$ is defined by
$$ \fS_{P_0}^P(\omega_{P_0}, t_0):=\omega_{P_0} A_{P_0}^\infty(P, t_0) K, $$
where 
$$ A_{P_0}^\infty(P, t_0):=\{a\in A_{P_0}^\infty: \forall \alpha\in\Delta_{P_0}^P, \alpha(H_{P_0}(a))>t_0\}. $$
We shall fix such $t_0$ and $\omega_{P_0}$. 
Moreover, we require that $(M_P(\BA)\cap\omega_{P_0}, M_P(\BA)\cap K, P_0\cap M_P, t_0)$ will satisfy the same properties of $(\omega_{P_0}, K, P_0, t_0)$ for any standard parabolic subgroup $P$ of $H$. 

Let $t_0$ be as above. For $T\in\fa_0$, define the truncated Siegel set
$$ \fS_{P_0}^P(\omega_{P_0}, t_0, T):=\omega_{P_0} A_{P_0}^\infty(P, t_0, T) K, $$
where 
$$ A_{P_0}^\infty(P, t_0, T):=\{a\in A_{P_0}^\infty(P, t_0): \forall \varpi\in\wh{\Delta}_{P_0}^P, \varpi(H_{P_0}(a)-T)\leq 0\}. $$
Denote by $F_{P_0}^P(\cdot, T)$ the characteristic function of the projection of $\fS_{P_0}^P(\omega_{P_0}, t_0, T)$ to $P(F)\bs H(\BA)$. 

\subsection{\textbf{Bruhat-Schwartz functions and Haar measures}}\label{BSandHaar2}

Let $F$ be a number field. Write $\fh$ for the Lie algebra of $H$. For an $F$-linear subspace $\fs$ of $\fh$, denote by $\CS(\fs(\BA))$ the Bruhat-Schwartz space of $\fs(\BA)$, namely the $\BC$-linear space of functions on $\fs(\BA)$ generated by $f_\infty\otimes\chi^\infty$, where $f_\infty$ is a Schwartz function on $\fs(F\otimes_\BQ \BR)$ and $\chi^\infty$ is the characteristic function of an open compact subset of $\fs(\BA^\infty)$, where $\BA^\infty$ denotes the ring of finite \adeles of $F$. 

Let $P$ be a standard parabolic subgroup of $H$. For every algebraic subgroup $V$ of $N_P$ (resp. every subspace $\fv$ of $\fh$), choose the unique Haar measure on $V(\BA)$ (resp. on $\fv(\BA)$) such that $\vol(V(F)\bs V(\BA))=1$ (resp. $\vol(\fv(F)\bs\fv(\BA))=1$). We also take the Haar measure on $K$ such that $\vol(K)=1$. 

Fix a Euclidean norm $\|\cdot\|$ on $\fa_0$ invariant by the group $\Omega^H$ and Haar measures on subspaces of $\fa_0$ compatible with this norm. If $P\subseteq Q$ are a pair of standard parabolic subgroups, we obtain Haar measures on $A_P^\infty$ and $A_P^{Q, \infty}$ via the isomorphism $H_P$. 

Denote by $\rho_P\in (\fa_P^H)^*$ the half of the sum of weights  (with multiplicities) for the action of $A_P$ on $\fn_P$. We choose compatible Haar measures on $H(\BA)$ and its Levi subgroups by requiring that for any $f\in L^1(H(\BA))$, 
  \[\begin{split}
   \int_{H(\BA)}f(x)dx&=\int_{N_P(\BA)} \int_{M_P(\BA)} \int_K f(nmk) e^{-2\rho_P(H_P(m))} dndmdk \\
                              &=\int_{N_P(\BA)} \int_{M_P(\BA)^1} \int_{A_P^\infty} \int_K f(nmak) e^{-2\rho_P(H_P(a))} dndmdadk.
  \end{split}\]


\section{\textbf{The symmetric pair}}\label{symmetricpair}

\subsection{\textbf{Groups and linear spaces}}\label{gpandls}

Let $F$ be a number field and $E$ a quadratic extension of $F$. Let $\fg$ be a central simple algebra over $F$ with a fixed embedding $E\ra\fg$ as $F$-algebras. Denote by $\fh:=\Cent_\fg(E)$ the centralizer of $E$ in $\fg$. Then by the the double centralizer theorem (see \cite[Theorem 3.1 in Chapter IV]{Milne} for example), $\fh$ is a central simple algebra over $E$. Write $G:=\fg^\times$ and $H:=\fh^\times$ for the group of invertible elements. They are considered as algebraic groups over $F$ with Lie algebra $\fg$ and $\fh$ respectively. 

Let $\alpha\in E$ such that $\alpha^2\in F$ and that $E=F(\alpha)$. Denote by $\Ad$ the adjoint action of $G$ on $\fg$. Define an involution $\theta$ on $\fg$ by $\theta(X):=\Ad(\alpha)(X)$, which is independent of the choice of $\alpha$. Then $H=G^\theta$, where $G^\theta$ denotes the $\theta$-invariant subgroup of $G$. Thus $G/H$ is a symmetric space. Define an anti-involution on $G$ by $\iota(g):=\theta(g^{-1})$. Denote by $S$ the $\iota$-invariant subvariety of $G$. Then there is a symmetrization map $s: G \ra S$ defined by $s(g):=g\iota (g)$. 

\begin{lem}\label{isosymandinv2}
The symmetrization map $s$ induces a bijection $(G/H)(F)\simeq S(F)$. 
\end{lem}

\begin{remark}
For the special case $(G,H)=(GL_{n,D}, \Res_{E/F}GL_{n,E})$, where $D$ is a quaternion algebra over $F$ containing $E$, this result is included in \cite[p. 282]{MR1487565}. 
\end{remark}

\begin{proof}[Proof of Lemma \ref{isosymandinv2}]
Since $H^1(F,H)=1$, we have $(G/H)(F)=G(F)/H(F)$. For $g\in G(F)$, let $s_0(g):=s(g)\alpha=\Ad(g)(\alpha)$. Let $S_0:=S\alpha=\{g\in G: g^2=\alpha^2\}$. It suffices to prove that the map $s_0: G(F) \ra S_0(F)$ is surjective. Let $g\in S_0(F)$. Its minimal polynomial in $\fg(F)$ is $\lambda^2-\alpha^2$, which is irreducible over $F$. Therefore, its reduced characteristic polynomial in $\fg(F)$ must be $(\lambda^2-\alpha^2)^m$, where $\dim_F(\fg(F))=(2m)^2$. We deduce that all elements in $S_0(F)$ are conjugate by $G(F)$ (see \cite[Theorem 9]{yu2013} for example). Since $\alpha\in S_0(F)$, we draw our conclusion. 
\end{proof}

One may consider the left and right translation of $H\times H$ on $G$ and the conjugation of $H$ on $S$. Denote by $\fs$ the tangent space of $S$ at the neutral element. We shall always view $\fs$ as the subspace of $\fg$. Then $\fs=\{X\in\fg: \theta(X)=-X\}$ and $H$ acts on $\fs$ by conjugation. 

\subsection{\textbf{Semi-simple elements}}

We say that an element $Y$ of $\fs$ is semi-simple if the orbit $\Ad(H)(Y)$ is Zariski closed in $\fs$. By a regular element $Y$ of $\fs$, we mean that the centralizer $H_Y$ of $Y$ in $H$ has minimal dimension. 

\begin{prop}\label{normmap}
The map $Y\mapsto Y^2$ from $\fs(F)$ to $\fh(F)$ induces an injection from the set of $H(F)$-conjugacy classes of semi-simple elements in $\fs(F)$ to the set of conjugacy classes of semi-simple elements in $\fh(F)$. 
\end{prop}

\begin{remark}
In the special case $(G,H)=(GL_{n,D}, \Res_{E/F}GL_{n,E})$, where $D$ is a quaternion algebra over $F$ containing $E$, this map plays the role of the norm map (see \cite[\S1 in Chapter 1]{MR1007299}). 
\end{remark}

\begin{proof}[Proof of Proposition \ref{normmap}]
Let $\chi_{\fg,F}(X)$ be the reduced characteristic polynomial of $X\in \fg(F)$ and $\chi_{\fh,E}(X^\ast)$ the reduced characteristic polynomial of $X^\ast\in\fh(F)$ (viewed as a central simple algebra over $E$). After the base change to an algebraic closure of $F$ containing $E$, the embedding $\fh\subseteq\fg$ is identical to the diagonal embedding $\fh':=\fg\fl_m\oplus\fg\fl_m\subseteq\fg':=\fg\fl_{2m}$ and $\fs\subseteq\fg$ becomes $\fs':=\bigg\{
\left( \begin{array}{cc}
0 & A \\
B & 0 \\
\end{array} \right): A,B\in \fg\fl_m\bigg\} \subseteq\fg'$, where $m$ denotes the degree of $\fh(F)$ (viewed as a central simple algebra over $E$). By \cite[Proposition 2.1]{MR1394521}, we see that if $Y\in\fs$ is semi-simple, then $Y^2\in\fh$ is semi-simple (in the usual sense).  
Since
$$ \det\left(\lambda I_{2m}-\mat(0,A,B,0)\right)=\det(\lambda^2 I_{m}-AB), $$
we see that for $Y\in\fs(F)\subseteq\fg(F)$, 
$$ \chi_{\fg,F}(Y)(\lambda)=\chi_{\fh,E}(Y^2)(\lambda^2), $$
which implies that $\chi_{\fh,E}(Y^2)$ is actually defined over $F$. 

It is known that the semi-simple conjugacy classes in $\fh(F)$ are uniquely determined by $\chi_{\fh,E}$ (see \cite[Theorem 9]{yu2013} for example). Thus it suffices to prove that the semi-simple $H(F)$-conjugacy classes  in $\fs(F)$ are uniquely determined by $\chi_{\fg,F}$. From \cite[Proposition 2.1]{MR1394521}, we know that the semi-simple $H$-conjugacy classes  in $\fs(F)$ are uniquely determined by $\chi_{\fg,F}$. Therefore, we reduce ourselves to proving that each semi-simple $H$-conjugacy class in $\fs(F)$ contains a unique $H(F)$-conjugacy class. 

For a semi-simple element $Y\in\fs(F)$, the $H(F)$-orbits in $\Ad(H)(Y)(F)$ are parametrized by 
$$ \ker[H^1(F,H_Y)\ra H^1(F,H)]=H^1(F,H_Y), $$
where $H_Y$ is the centralizer of $Y$ in $H$. 
Since $H_Y$ is the group of invertible elements of a finite dimensional associative algebra (the centralizer of $Y$ in $\fh$), by \cite[Exercice 2 in p. 160]{MR0354618}, we obtain
$$ H^1(F,H_Y)=1, $$
which completes our proof. 
\end{proof}

\subsection{\textbf{Invariants}}\label{inv2}

Denote by $\fc$ the affine space $\bfA^m$, where $m$ denotes the degree of $\fh(F)$ (viewed as a central simple algebra over $E$). Define a morphism $\pi:\fs\ra\fc$ which is contant on $H$-orbits by mapping $Y\in\fs$ to the coefficients of the reduced characteristic polynomial of $Y\in\fg$. In fact, we see that the coefficients in odd degrees vanish for $Y\in\fs$ from the proof of Proposition \ref{normmap}. On $F$-points, alternatively, $\pi$ is given by mapping $Y\in\fs(F)$ to the coefficients of the reduced characteristic polynomial of $Y^2\in\fh(F)$ (viewed as a central simple algebra over $E$). 

\begin{prop}\label{propcatquot2}
The pair $(\fc,\pi)$ defines a categorical quotient of $\fs$ by $H$ over $F$. 
\end{prop}

\begin{proof}
By the proof of \cite[Proposition 3.3]{2019arXiv190407102L}, after the base change to an algebraic closure $\ov{F}$ of $F$ containing $E$, the pair $(\fc_{\ov{F}}, \pi_{\ov{F}})$ defines a categorical quotient of $\fs_{\ov{F}}$ by $H_{\ov{F}}$. That is to say, we have an isomorphism of $\ov{F}$-algebras $\ov{F}[\fc]\simeq\ov{F}[\fs]^H$ dual to $\pi_{\ov{F}}$. But this isomorphism is obtained from the base change of a morphism of $F$-algebras $F[\fc]\ra F[\fs]^H$ dual to $\pi$. By Galois descent, the latter morphism is necessarily an isomorphism of $F$-algebras. Then the pair $(\fc,\pi)$ defines a categorical quotient of $\fs$ by $H$ over $F$. 
\end{proof}

\begin{remark}
The morphism $\pi$ is surjective as a morphism of algebraic varieties (see the proof of \cite[Proposition 3.3]{2019arXiv190407102L}) but not surjective on the level of $F$-points. 
\end{remark}

We define a relation of equivalence on $\fs(F)$ using the categorical quotient $(\fc,\pi)$, where two elements are in the same class if and only if they have the same image under $\pi$. We denote by $\CO$ the set of equivalent classes for this relation. From the proof of Proposition \ref{normmap}, we see that two semi-simple elements of $\fs(F)$ belong to the same class of $\CO$ if and only if they are conjugate by $H(F)$. Denote by $\CO_{rs}$ the subset of $\CO$ formed by classes of $Y\in\fs(F)$ such that $\chi_{\fh,E}(Y^2)$ is separable and $\chi_{\fh,E}(Y^2)(0)\neq 0$, where $\chi_{\fh,E}$ denotes the reduced characteristic polynomial of an element in $\fh(F)$ (viewed as a central simple algebra over $E$). By \cite[Proposition 3.2]{2019arXiv190407102L} and the base change to an algebraic closure of $F$ containing $E$, we see that each class in $\CO_{rs}$ is a regular semi-simple $H(F)$-orbit in $\fs(F)$. 

\subsection{\textbf{Explicit description of $H\hookrightarrow G$}}\label{explicitdescription}

First of all, we would like to describe the symmetric pair $(G,H)$ in a more explicit way. By the Noether-Skolem theorem (see \cite[Theorem 2.10 of Chapter IV]{Milne} for example), the embedding $E\ra\fg(F)$ is unique up to conjugation by an element of $G(F)$. From the Wedderburn-Artin theorem, we know that $G$ is isomorphic to $GL_{n,D}$, which denotes the reductive group over $F$ whose $F$-points are $GL_n(D)$, for some central division algebra $D$ over $F$. We recall that $n$ is called the capacity of $\fg(F)$ and we denote it by $\capa(\fg(F))$. Let $d$ be the degree of $D$, i.e., $\dim_F(D)=d^2$. Since there is an embedding $E\ra\fg(F)$ as $F$-algebras, we know that $nd$ is even. 

\begin{prop}
Up to conjugation by $G(F)$, the embedding $H\hookrightarrow G$ is reduced to one of the two cases below.  

\textbf{Case I}: if there is an embedding $E\ra D$ as $F$-algebras, then the embedding $H\hookrightarrow G$ is isomorphic to $\Res_{E/F} GL_{n,D'}\hookrightarrow GL_{n,D}$ up to conjugation by $G(F)$. Here $D':=\Cent_D(E)$ denotes the centralizer of $E$ in $D$ and is a central division algebra over $E$. 

\textbf{Case II}: if there is no embedding $E\ra D$ as $F$-algebras, then the embedding $H\hookrightarrow G$ is isomorphic to $\Res_{E/F} GL_{\frac{n}{2},D\otimes_F E}\hookrightarrow GL_{n,D}$ up to conjugation by $G(F)$. Here $D\otimes_F E$ is a central division algebra over $E$. 
\end{prop}

\begin{proof}
\textbf{Case I}: there is an embedding $E\ra D$ as $F$-algebras. This case is a direct consequence of the Noether-Skolem theorem. By the double centralizer theorem, we know that $D'$ is a central division algebra over $E$. 

\textbf{Case II}: there is no embedding $E\ra D$ as $F$-algebras. 
By \cite[Theorem 1.1.2]{MR3217641}, when $nd$ is even, there is an embedding $E\ra\fg(F)$ as $F$-algebras if and only if $n\cdot\capa(D\otimes_F E)$ is even, where $\capa(D\otimes_F E)$ denotes the capacity of the central simple algebra $D\otimes_F E$ over $E$ (see \cite[Proposition 2.15 in Chapter IV]{Milne} for example). Additionally, from \cite[Theorem 1.1.3]{MR3217641}, we show that $\capa(D\otimes_F E)\leq [E:F]=2$. 
In this case, there are two possibilities. 
\begin{enumerate}[(1)]
\item $d$ is even. By \cite[Theorem 1.1.2]{MR3217641}, $\capa(D\otimes_F E)$ is odd, so $\capa(D\otimes_F E)=1$. Since $n\cdot\capa(D\otimes_F E)$ is even, we know that $n$ is even. 
\end{enumerate}
\begin{enumerate}[(2)]
\item $d$ is odd. Since $nd$ is even, we see that $n$ is even. Besides, from \cite[Theorem 1.1.3]{MR3217641}, we also deduce that $\capa(D\otimes_F E)=1$. 
\end{enumerate}
In sum, we have shown that $n$ is even and that $D\otimes_F E$ is a central division algebra over $E$\footnote{As suggested by the referee, there is another more elementary way to see that $D\otimes_F E$ is a central division algebra over $E$ in \textbf{Case II}. In fact, if there were a non-zero non-invertible element $x\in D_E:=D\otimes_F E$, then $x D_E$ would be a nontrivial proper right $D$-submodule of $D_E$ stable by $E$. Hence, it would be of dimension $1$ over $D$, and there would be an embedding from $E$ into $\End_D(xD_E)\simeq D$ as $F$-algebras. Contradiction! }. The tensor of $\fg\fl_{\frac{n}{2},D}$ and a fixed embedding $\Res_{E/F} \fg\fl_{1,E}\ra\fg\fl_2$ gives the indicated way to embed $\fh$ to $\fg$. By the Noether-Skolem theorem, such an embedding is unique up to conjugation by $G(F)$. 
\end{proof}

Next, we describe the correspondence of some parabolic subgroups in $H$ and $G$ in both cases. 

\textbf{Case I}: $(G,H)=(GL_{n,D},\Res_{E/F} GL_{n,D'})$, where $D':=\Cent_D(E)$.  
We denote by $M_0\simeq (\Res_{E/F} \BG_{m,D'})^n$ the subgroup of diagonal elements in $H$, which is a minimal Levi $F$-subgroup of $H$, and by $M_{\wt{0}}\simeq (\BG_{m,D})^n$ the subgroup of diagonal elements in $G$, which is a minimal Levi $F$-subgroup of $G$. We call a parabolic subgroup of $G$ semi-standard if it contains $M_{\wt{0}}$. We also fix $P_0$ the subgroup of upper triangular elements in $H$, which is a minimal parabolic $F$-subgroup of $H$. We call a parabolic subgroup of $H$ standard if it contains $P_0$. There is a bijection $P\mapsto\wt{P}$ characterized by partitions of $n$ from the set of standard parabolic subgroups $P$ of $H$ to the set of semi-standard parabolic subgroups $\wt{P}$ containing $P_0$ of $G$. More precisely, the image of an element $P$ in the source of this bijection is given by the unique element $\wt{P}$ in the target such that $\wt{P}\cap H=P$. In this case, the target is exactly the set of standard parabolic subgroups (namely containing $\wt{P_0}$ the subgroup of upper triangular elements in $G$) of $G$. We shall always write $\wt{P}$ for the image of $P$ under this bijection. Notice that we can identify $A_P$ with $A_{\wt{P}}$. 

\textbf{Case II}: $(G,H)=(GL_{n,D},\Res_{E/F} GL_{\frac{n}{2},D\otimes_F E})$. 
We denote by $M_0\simeq (\Res_{E/F} \BG_{m,D\otimes_F E})^\frac{n}{2}$ the subgroup of diagonal elements in $H$, which is a minimal Levi $F$-subgroup of $H$, and by $M_{\wt{0}}\simeq (\BG_{m,D})^n$ the subgroup of diagonal elements in $G$, which is a minimal Levi $F$-subgroup of $G$. We call a parabolic subgroup of $G$ semi-standard if it contains $M_{\wt{0}}$. We also fix $P_0$ the subgroup of upper triangular elements in $H$, which is a minimal parabolic $F$-subgroup of $H$. We call a parabolic subgroup of $H$ standard if it contains $P_0$. There is a bijection $P\mapsto\wt{P}$ characterized by partitions of $\frac{n}{2}$ from the set of standard parabolic subgroups $P$ of $H$ to the set of semi-standard parabolic subgroups $\wt{P}$ containing $P_0$ of $G$. More precisely, the image of an element $P$ in the source of this bijection is given by the unique element $\wt{P}$ in the target such that $\wt{P}\cap H=P$. In this case, the target is a subset of the set of standard parabolic subgroups (namely containing the subgroup of upper triangular elements in $G$) of $G$. We shall always write $\wt{P}$ for the image of $P$ under this bijection. Notice that we can identify $A_P$ with $A_{\wt{P}}$. 

\begin{prop}
  Let $P$ be a standard parabolic subgroup of $H$. For all $Y\in(\fm_{\wt{P}}\cap\fs)(F)$ and $U\in(\fn_{\wt{P}}\cap\fs)(F)$, we have
  $$ \pi(Y)=\pi(Y+U). $$
\end{prop}

\begin{proof}
  It is obvious, since the reduced characteristic polynomial of $Y+U\in\fg$ is equal to that of $Y\in\fg$. 
\end{proof}

\begin{coro}\label{orbit2}
  Let $P$ be a standard parabolic subgroup of $H$ and $\fo\in\CO$. For all subsets $S_1\subseteq(\fm_{\wt{P}}\cap\fs)(F)$ and $S_2\subseteq(\fn_{\wt{P}}\cap\fs)(F)$, we have $\fo\cap(S_1\oplus S_2)=(\fo\cap S_1)\oplus S_2$.
\end{coro}

Let $P$ be a standard parabolic subgroup of $H$. We denote by $\Phi(A_0, \fm_{\wt{P}}\cap\fs)$ (resp. $\Phi(A_0, \fn_{\wt{P}}\cap\fs)$) the set of weights of $A_0$ in $\fm_{\wt{P}}\cap\fs$ (resp. $\fn_{\wt{P}}\cap\fs$). We also denote by $\Phi(A_0, \fm_P)$ (resp. $\Phi(A_0, \fn_P)$) the set of weights of $A_0$ in $\fm_P$ (resp. $\fn_P$). 

\begin{prop}\label{equalwt2.1}
For any standard parabolic subgroup $P$ of $H$, we have
$$ \Phi(A_0, \fm_{\wt{P}}\cap\fs)=\Phi(A_0, \fm_P) $$
and
$$ \Phi(A_0, \fn_{\wt{P}}\cap\fs)=\Phi(A_0, \fn_P). $$
Moreover, each weight of $A_0$ has the same multiplicity in $\fm_{\wt{P}}\cap\fs$ (resp. $\fn_{\wt{P}}\cap\fs$) and $\fm_P$ (resp. $\fn_P$). 
\end{prop}

\begin{proof}
It is obvious for both of \textbf{Case I} and \textbf{Case II} described above.  
\end{proof}

For $P$ a standard parabolic subgroup of $H$, let $\rho_{P,\fs}$ (resp. $\rho_P$) denote the half of the sum of weights (with multiplicities) of $A_0$ in $\fn_{\wt{P}}\cap\fs$ (resp. $\fn_P$). 

\begin{coro}\label{corequalwt2.1}
For any standard parabolic subgroup $P$ of $H$, we have
$$ \rho_{P,\fs}=\rho_P. $$
\end{coro}

At the end of this subsection, we point out two non-canonical $F$-linear isomorphisms between $\fh$ and $\fs$ which will be useful for some technical problems. We have chosen an element $\alpha\in E$ in Section \ref{gpandls}. 
Since $\Ad(\alpha)$ is a nontrivial involution on $D$ in \textbf{Case I} (resp. $\fg\fl_2(F)$ in \textbf{Case II}), there exists an element $\tau\in D^\times$ in \textbf{Case I} (resp. $\tau\in GL_2(D)$ in \textbf{Case II}) such that $\Ad(\alpha)(\tau)=-\tau$. Let $\tau$ be such an element. By the diagonal embedding, we can view $\BG_{m, D}$ in \textbf{Case I} (resp. $GL_{2, D}$ in \textbf{Case II}) as an algebraic subgroup of $G$. Then $\tau$ can be regarded as an element of $G(F)$. 

\begin{prop}\label{equalwt2.2}
There are two non-canonical isomorphisms induced by multiplication by $\tau$ between $\fh$ and $\fs$ as free left or right $D'$-modules in \textbf{Case I} (resp. $D\otimes_F E$-modules in \textbf{Case II}), i.e., $$ \fs=\fh\tau=\tau\fh. $$
Moreover, for any standard parabolic subgroup $P$ of $H$, we have
$$ \fm_{\wt{P}}\cap\fs=\fm_P\tau=\tau\fm_P $$
and
$$ \fn_{\wt{P}}\cap\fs=\fn_P\tau=\tau\fn_P. $$
\end{prop}

\begin{proof}
It is obvious for both of \textbf{Case I} and \textbf{Case II} described above.  
\end{proof}

\subsection{\textbf{Fourier transform}}
Fix a nontrivial unitary character $\Psi$ of $\BA/F$. Let $\langle\cdot,\cdot\rangle$ be the $H(\BA)$-invariant bilinear form on $\fs(\BA)$ defined by
\begin{equation}\label{bilinearform2}
 \forall Y_1,Y_2\in\fs(\BA), \langle Y_1,Y_2\rangle:=\Trd_{\fg,F}(Y_1 Y_2), 
\end{equation}
where $\Trd_{\fg,F}(Y_1 Y_2)$ denotes the reduced trace of $Y_1 Y_2\in\fg(\BA)$. It is non-degenerate, which can be seen after the base change to an algebraic closure of $F$. For $f\in\CS(\fs(\BA))$, its Fourier transform $\hat{f}$ is defined by
\begin{equation}\label{fouriertransform2}
 \forall \wh{Y}\in\fs(\BA), \hat{f}(\wh{Y}):=\int_{\fs(\BA)} f(Y)\Psi(\langle Y,\wh{Y}\rangle) dY. 
\end{equation}


\section{\textbf{Integrability of the modified kernel}}

Let $f\in\CS(\fs(\BA))$, $P$ be a standard parabolic subgroup of $H$ and $\fo\in\CO$. For $x\in M_P(F) N_P(\BA)\bs H(\BA)$, define
$$ k_{f,P,\fo}(x):=\sum_{Y\in\fm_{\wt{P}}(F)\cap\fo} \int_{(\fn_{\wt{P}}\cap\fs)(\BA)} f(\Ad(x^{-1})(Y+U)) dU. $$
For $T\in\fa_0$ and $x\in H(F)\bs H(\BA)$, define
\begin{equation}\label{deftrunc2.1}
 k_{f,\fo}^T(x):=\sum_{\{P: P_0\subseteq P\}} (-1)^{\dim(A_P/A_{H})} \sum_{\delta\in P(F)\bs H(F)} \wh{\tau}_P^H(H_{P}(\delta x)-T)\cdot k_{f,P,\fo}(\delta x). 
\end{equation}
By \cite[Lemma 5.1]{MR518111}, we know that the sum over $\delta\in P(F)\bs H(F)$ is finite.

\begin{lem}\label{combinatoriclemma2}
  There is a $T_+\in\fa_{P_0}^+$ such that for all standard parabolic subgroup $P$ of $H$, $T\in T_+ +\fa_{P_0}^+$ and $x\in H(\BA)$, we have
  $$ \sum_{\{P_1:P_0\subseteq P_1\subseteq P\}} \sum_{\delta_1\in P_1(F)\bs P(F)} F^{P_1}(\delta_1 x, T) \tau_{P_1}^P(H_{P_1}(\delta_1 x)-T)=1. $$
\end{lem}

\begin{proof}
  This is \cite[Lemma 6.4]{MR518111}.
\end{proof}

We shall fix such a $T_+$. 

\begin{thm}\label{convergence2}
  For all $T\in T_+ +\fa_{P_0}^+$,
  $$ \sum_{\fo\in\CO} \int_{H(F)\bs H(\BA)^1} |k_{f,\fo}^T(x)| dx < \infty. $$
\end{thm}

\begin{proof}
  Let $P_1\subseteq P_2$ be a pair of standard parabolic subgroups of $H$. As in \cite[\S6]{MR518111}, for $T_1\in \fa_{P_1}$, define the characteristic function
  $$ \sigma_{P_1}^{P_2}(T_1):=\sum_{\{Q:P_2\subseteq Q\}} (-1)^{\dim(A_{P_2}/A_Q)} \tau_{P_1}^Q(T_1) \wh{\tau}_Q^H(T_1). $$
  Recall that for $P\supseteq P_1$ a standard parabolic subgroup of $H$, we have (see \cite[p. 943]{MR518111})
  $$ \tau_{P_1}^P(T_1)\wh{\tau}_P^H(T_1)=\sum_{\{P_2:P\subseteq P_2\}}\sigma_{P_1}^{P_2}(T_1). $$
  For $x\in P_1(F)\bs H(\BA)$, we write
  $$ \chi_{P_1, P_2}^T(x):=F^{P_1}(x, T)\sigma_{P_1}^{P_2}(H_{P_1}(x)-T) $$
  and
  $$ k_{P_1, P_2, \fo}(x):=\sum_{\{P: P_1\subseteq P\subseteq P_2\}} (-1)^{\dim(A_P/A_{H})} k_{f,P,\fo}(x). $$

  By Lemma \ref{combinatoriclemma2} and the left invariance of $H_P$ and $k_{f,P,\fo}$ by $P(F)$, we obtain
  $$ k_{f, \fo}^T(x)=\sum_{\{P_1,P_2: P_0\subseteq P_1\subseteq P_2\}} \sum_{\delta\in P_1(F)\bs H(F)} \chi_{P_1, P_2}^T(\delta x) k_{P_1, P_2, \fo}(\delta x). $$
  Thus
  $$ \sum_{\fo\in\CO} \int_{H(F)\bs H(\BA)^1} |k_{f,\fo}^T(x)| dx\leq \sum_{\fo\in\CO} \sum_{\{P_1,P_2: P_0\subseteq P_1\subseteq P_2\}} \int_{P_1(F)\bs H(\BA)^1} \chi_{P_1, P_2}^T(x) |k_{P_1, P_2, \fo}(x)| dx. $$
  Then we only need to show that for any pair of standard parabolic subgroups $P_1\subseteq P_2$ of $H$,
  $$ \sum_{\fo\in\CO} \int_{P_1(F)\bs H(\BA)^1} \chi_{P_1, P_2}^T(x) |k_{P_1, P_2, \fo}(x)| dx <\infty. $$
  If $P_1=P_2\neq H$, by \cite[Lemma 6.1]{MR518111}, we have $\sigma_{P_1}^{P_2}=0$ and then $\chi_{P_1, P_2}^T=0$, so the integration vanishes. If $P_1=P_2=H$, since $F^{H}(\cdot, T)$ is a characteristic function with compact support in $H(F)\bs H(\BA)^1$, the integration is convergent. Hence, we reduce ourselves to proving the following proposition.
\end{proof}

\begin{prop}\label{prop43}
  Let $f\in\CS(\fs(\BA))$ and $P_1\subsetneq P_2$ be a pair of standard parabolic subgroups of $H$. Suppose that $\epsilon_0$ and $N$ are two arbitrary but fixed positive real numbers. Then there exists a constant $C$ such that
  $$ \sum_{\fo\in\CO} \int_{P_1(F)\bs H(\BA)^1} \chi_{P_1, P_2}^T(x) |k_{P_1, P_2, \fo}(x)| dx\leq Ce^{-N\parallel T\parallel} $$
  for all $T\in T_+ +\fa_{P_0}^+$ satisfying $\alpha(T)\geq\epsilon_0\parallel T\parallel$ for any $\alpha\in\Delta_{P_0}^H$.
\end{prop}

For $x\in H(F)\bs H(\BA)$, define
$$ k_{f,H}(x):=\sum_{\fo\in\CO} k_{f,H,\fo}(x)=\sum_{Y\in\fs(F)} f(\Ad(x^{-1})(Y)) $$
and
\begin{equation}\label{deftrunc2.2}
 k_{f}^T(x):=\sum_{\fo\in\CO}k_{f,\fo}^{T}(x). 
\end{equation}

\begin{coro}\label{comparewithnaive2}
  Let $f\in\CS(\fs(\BA))$. For two arbitrary but fixed positive real number $\epsilon_0$ and $N$, there exists a constant $C$ such that
  $$ \int_{H(F)\bs H(\BA)^1} |k_{f}^{T}(x)-F^{H}(x,T)k_{f,H}(x)| dx\leq Ce^{-N\parallel T\parallel} $$
  for all $T\in T_+ +\fa_{P_0}^+$ satisfying $\alpha(T)\geq\epsilon_0\parallel T\parallel$ for any $\alpha\in\Delta_{P_0}^H$.
\end{coro}

\begin{proof}[Proof of Proposition \ref{prop43}]
  Let $P$ be any standard parabolic subgroup of $H$ such that $P_1\subseteq P\subseteq P_2$. For any $Y\in\fm_{\wt{P}}(F)\cap\fo$, there is a unique standard parabolic subgroup $R$ of $H$ such that $P_1\subseteq R\subseteq P$ and $Y\in (\fm_{\wt{P}}(F)\cap\wt{\fr}(F)\cap\fo)-\left(\bigcup\limits_{P_1\subseteq Q\subsetneq R} \fm_{\wt{P}}(F)\cap\wt{\fq}(F)\cap\fo\right)$. We denote
  $$ \wt{\fm}_{\wt{P_1}}^{\wt{R}}:=\fm_{\wt{R}}-\left(\bigcup_{\{Q:P_1 \subseteq Q \subsetneq R\}} \fm_{\wt{R}}\cap\wt{\fq}\right) $$
  and
  $$ \fn_{\wt{R}}^{\wt{P}}:=\fn_{\wt{R}}\cap\fm_{\wt{P}}. $$
  From Corollary \ref{orbit2}, we get
  $$ (\fm_{\wt{P}}(F)\cap\wt{\fr}(F)\cap\fo)-\left(\bigcup\limits_{P_1\subseteq Q\subsetneq R} \fm_{\wt{P}}(F)\cap\wt{\fq}(F)\cap\fo\right)=(\wt{\fm}_{\wt{P_1}}^{\wt{R}}(F)\cap\fo)\oplus((\fn_{\wt{R}}^{\wt{P}}\cap\fs)(F)). $$
  Thus
  \[\begin{split}
   k_{f,P,\fo}(x)&=\sum_{Y\in\fm_{\wt{P}}(F)\cap\fo} \int_{(\fn_{\wt{P}}\cap\fs)(\BA)} f(\Ad(x^{-1})(Y+U)) dU \\
   &=\sum_{\{R:P_1\subseteq R\subseteq P\}} \sum_{\xi\in\wt{\fm}_{\wt{P_1}}^{\wt{R}}(F)\cap\fo} \sum_{Y\in(\fn_{\wt{R}}^{\wt{P}}\cap\fs)(F)} \int_{(\fn_{\wt{P}}\cap\fs)(\BA)} f(\Ad(x^{-1})(\xi+Y+U)) dU. 
  \end{split}\]

  We write $\ov{\wt{P}}$ for the semi-standard parabolic subgroup of $G$ opposite to $\wt{P}$ and
  $$ \ov{\fn}_{\wt{R}}^{\wt{P}}:=\fn_{\ov{\wt{R}}}\cap\fm_{\wt{P}}. $$
  Notice that the restriction of $\langle\cdot,\cdot\rangle$ (see (\ref{bilinearform2})) to $((\fn_{\wt{R}}^{\wt{P}}\cap\fs)(\BA))\times((\ov{\fn}_{\wt{R}}^{\wt{P}}\cap\fs)(\BA))$ is also non-degenerate. For any $\xi\in(\fm_{\wt{R}}\cap\fs)(\BA)$, applying the Poisson summation formula to the Bruhat-Schwartz function $\int_{(\fn_{\wt{P}}\cap\fs)(\BA)} f(\Ad(x^{-1})(\xi+\cdot+U)) dU$, we have
  $$ \sum_{Y\in(\fn_{\wt{R}}^{\wt{P}}\cap\fs)(F)} \int_{(\fn_{\wt{P}}\cap\fs)(\BA)} f(\Ad(x^{-1})(\xi+Y+U)) dU = \sum_{\wh{Y}\in(\ov{\fn}_{\wt{R}}^{\wt{P}}\cap\fs)(F)} \Phi_\xi^{x,R}(\wh{Y}), $$
  where the partial Fourier transform $\Phi_\xi^{x,R}$ of $\int_{(\fn_{\wt{P}}\cap\fs)(\BA)} f(\Ad(x^{-1})(\xi+\cdot+U)) dU$ is defined by
  $$ \forall \wh{Y}\in (\ov{\fn}_{\wt{R}}^{\wt{P}}\cap\fs)(\BA), \Phi_\xi^{x,R}(\wh{Y}):=\int_{(\fn_{\wt{R}}^{\wt{P}}\cap\fs)(\BA)} \left(\int_{(\fn_{\wt{P}}\cap\fs)(\BA)} f(\Ad(x^{-1})(\xi+Y+U)) dU\right) \Psi(\langle Y,\wh{Y}\rangle) dY. $$
  Since $\langle U,\wh{Y}\rangle=0$ for $U\in (\fn_{\wt{P}}\cap\fs)(\BA)$ and $\wh{Y}\in(\ov{\fn}_{\wt{R}}^{\wt{P}}\cap\fs)(\BA)$, as well as $\fn_{\wt{R}}\cap\fs=(\fn_{\wt{P}}\cap\fs)\oplus(\fn_{\wt{R}}^{\wt{P}}\cap\fs)$, we have
  $$ \forall \wh{Y}\in (\ov{\fn}_{\wt{R}}^{\wt{P}}\cap\fs)(\BA), \Phi_\xi^{x,R}(\wh{Y})=\int_{(\fn_{\wt{R}}\cap\fs)(\BA)} f(\Ad(x^{-1})(\xi+U))\Psi(\langle U,\wh{Y}\rangle) dU, $$
  which is actually independent of $P$.

  In sum,
  $$ k_{f,P,\fo}(x)=\sum_{\{R:P_1\subseteq R\subseteq P\}} \sum_{\xi\in\wt{\fm}_{\wt{P_1}}^{\wt{R}}(F)\cap\fo} \sum_{\wh{Y}\in(\ov{\fn}_{\wt{R}}^{\wt{P}}\cap\fs)(F)} \Phi_\xi^{x,R}(\wh{Y}). $$
  Then we have
  \[\begin{split}
    k_{P_1, P_2, \fo}(x)&=\sum_{\{P: P_1\subseteq P\subseteq P_2\}} (-1)^{\dim(A_P/A_{H})} k_{f,P,\fo}(x) \\
    &=\sum_{\{P: P_1\subseteq P\subseteq P_2\}} (-1)^{\dim(A_P/A_{H})} \left(\sum_{\{R:P_1\subseteq R\subseteq P\}} \sum_{\xi\in\wt{\fm}_{\wt{P_1}}^{\wt{R}}(F)\cap\fo} \sum_{\wh{Y}\in(\ov{\fn}_{\wt{R}}^{\wt{P}}\cap\fs)(F)} \Phi_\xi^{x,R}(\wh{Y})\right) \\
    &=\sum_{\{R:P_1\subseteq R\subseteq P_2\}} \sum_{\xi\in\wt{\fm}_{\wt{P_1}}^{\wt{R}}(F)\cap\fo} \left(\sum_{\{P:R\subseteq P\subseteq P_2\}} (-1)^{\dim(A_P/A_{H})} \sum_{\wh{Y}\in(\ov{\fn}_{\wt{R}}^{\wt{P}}\cap\fs)(F)} \Phi_\xi^{x,R}(\wh{Y})\right).
  \end{split}\]
  For $P_3$ a standard parabolic subgroup of $H$ containing $R$, denote
  $$ (\ov{\fn}_{\wt{R}}^{\wt{P_3}})':=\ov{\fn}_{\wt{R}}^{\wt{P_3}}-\left(\bigcup_{\{Q:R\subseteq Q\subsetneq P_3\}}\ov{\fn}_{\wt{R}}^{\wt{Q}}\right). $$
  We write
\[\begin{split}  
 &\sum_{\{P:R\subseteq P\subseteq P_2\}} (-1)^{\dim(A_P/A_{H})} \sum_{\wh{Y}\in(\ov{\fn}_{\wt{R}}^{\wt{P}}\cap\fs)(F)} \Phi_\xi^{x,R}(\wh{Y}) \\
             =&\sum_{\{P:R\subseteq P\subseteq P_2\}} (-1)^{\dim(A_P/A_{H})} \sum_{\{P_3:R\subseteq P_3\subseteq P\}} \sum_{\wh{Y}\in((\ov{\fn}_{\wt{R}}^{\wt{P_3}})'\cap\fs)(F)} \Phi_\xi^{x,R}(\wh{Y})  \\
             =&(-1)^{\dim(A_{P_2}/A_{H})} \sum_{\{P_3:R\subseteq P_3\subseteq P_2\}} \left(\sum_{\{P:P_3\subseteq P\subseteq P_2\}} (-1)^{\dim(A_{P}/A_{P_2})}\right) \sum_{\wh{Y}\in((\ov{\fn}_{\wt{R}}^{\wt{P_3}})'\cap\fs)(F)} \Phi_\xi^{x,R}(\wh{Y}) \\
             =&(-1)^{\dim(A_{P_2}/A_{H})} \sum_{\wh{Y}\in((\ov{\fn}_{\wt{R}}^{\wt{P_2}})'\cap\fs)(F)} \Phi_\xi^{x,R}(\wh{Y}), 
\end{split}\]
where we have used \cite[Proposition 1.1]{MR518111} in the last equality. 
  Then
  $$ k_{P_1, P_2, \fo}(x)=(-1)^{\dim(A_{P_2}/A_{H})} \sum_{\{R:P_1\subseteq R\subseteq P_2\}} \sum_{\xi\in\wt{\fm}_{\wt{P_1}}^{\wt{R}}(F)\cap\fo} \sum_{\wh{Y}\in((\ov{\fn}_{\wt{R}}^{\wt{P_2}})'\cap\fs)(F)} \Phi_\xi^{x,R}(\wh{Y}).  $$

  Now we get
  \[\begin{split}
    &\sum_{\fo\in\CO} \int_{P_1(F)\bs H(\BA)^1} \chi_{P_1, P_2}^T(x) |k_{P_1, P_2, \fo}(x)| dx \\
    \leq &\sum_{\fo\in\CO} \int_{P_1(F)\bs H(\BA)^1} \chi_{P_1, P_2}^T(x) \left( \sum_{\{R:P_1\subseteq R\subseteq P_2\}} \sum_{\xi\in\wt{\fm}_{\wt{P_1}}^{\wt{R}}(F)\cap\fo} \sum_{\wh{Y}\in((\ov{\fn}_{\wt{R}}^{\wt{P_2}})'\cap\fs)(F)} |\Phi_\xi^{x,R}(\wh{Y})| \right) dx \\
    = &\sum_{\{R:P_1\subseteq R\subseteq P_2\}} \int_{P_1(F)\bs H(\BA)^1} \chi_{P_1, P_2}^T(x) \sum_{\xi\in(\wt{\fm}_{\wt{P_1}}^{\wt{R}}\cap\fs)(F)} \sum_{\wh{Y}\in((\ov{\fn}_{\wt{R}}^{\wt{P_2}})'\cap\fs)(F)} |\Phi_\xi^{x,R}(\wh{Y})| dx. 
  \end{split}\]
  Thus it suffices to bound
  \begin{equation}\label{equation2.1}
    \int_{P_1(F)\bs H(\BA)^1} \chi_{P_1, P_2}^T(x) \sum_{\xi\in(\wt{\fm}_{\wt{P_1}}^{\wt{R}}\cap\fs)(F)} \sum_{\wh{Y}\in((\ov{\fn}_{\wt{R}}^{\wt{P_2}})'\cap\fs)(F)} |\Phi_\xi^{x,R}(\wh{Y})| dx
  \end{equation}
  for any fixed standard parabolic subgroup $R$ of $H$ such that $P_1\subseteq R\subseteq P_2$. 

By Iwasawa decomposition and our choice of measures, the integral over $x\in P_1(F)\bs H(\BA)^1$ can be decomposed as integrals over
$$ (n_1, m_1, a_1, k)\in N_{P_1}(F)\bs N_{P_1}(\BA)\times M_{P_1}(F)\bs M_{P_1}(\BA)^1\times A_{P_1}^{H,\infty}\times K, $$
and we have
  \[\begin{split}
    &\int_{P_1(F)\bs H(\BA)^1} \chi_{P_1, P_2}^T(x) \sum_{\xi\in(\wt{\fm}_{\wt{P_1}}^{\wt{R}}\cap\fs)(F)} \sum_{\wh{Y}\in((\ov{\fn}_{\wt{R}}^{\wt{P_2}})'\cap\fs)(F)} |\Phi_\xi^{x,R}(\wh{Y})| dx \\
    =&\int_K \int_{A_{P_1}^{H,\infty}} \int_{M_{P_1}(F)\bs M_{P_1}(\BA)^1} \int_{N_{P_1}(F)\bs N_{P_1}(\BA)} F^{P_1}(m_1,T) \sigma_{P_1}^{P_2}(H_{P_1}(a_1)-T) \\
    &\cdot \sum_{\xi\in(\wt{\fm}_{\wt{P_1}}^{\wt{R}}\cap\fs)(F)} \sum_{\wh{Y}\in((\ov{\fn}_{\wt{R}}^{\wt{P_2}})'\cap\fs)(F)} |\Phi_\xi^{n_1 m_1 a_1 k,R}(\wh{Y})| e^{-2\rho_{P_1}(H_{P_0}(a_1))} dn_1 dm_1 da_1 dk.
  \end{split}\]
  Because only those $m_1$ satisfying $F^{P_1}(m_1,T)\neq 0$ contribute to the integral, we can restrict the integration over those having representatives in $(N_{P_0}(\BA)M_{P_0}(\BA)^1 A_{P_0}^{P_1, \infty}(t_0,T) K)\cap M_{P_1}(\BA)^1$, where $A_{P_0}^{P_1, \infty}(t_0,T):=A_{P_0}^{\infty}(P_1,t_0,T)\cap M_{P_1}(\BA)^1$. Then
  \[\begin{split}
    &\int_{P_1(F)\bs H(\BA)^1} \chi_{P_1, P_2}^T(x) \sum_{\xi\in(\wt{\fm}_{\wt{P_1}}^{\wt{R}}\cap\fs)(F)} \sum_{\wh{Y}\in((\ov{\fn}_{\wt{R}}^{\wt{P_2}})'\cap\fs)(F)} |\Phi_\xi^{x,R}(\wh{Y})| dx \\
    \leq &c_1 \int_K \int_{[cpt\subseteq M_{P_0}(\BA)^1]} \int_{A_{P_0}^{P_1,\infty}(t_0,T)} \int_{A_{P_1}^{H,\infty}} \int_{[cpt\subseteq N_{P_0}^{P_2}(\BA)]} \int_{[cpt\subseteq N_{P_2}(\BA)]} \sigma_{P_1}^{P_2}(H_{P_1}(a_1)-T) \\ &\cdot \sum_{\xi\in(\wt{\fm}_{\wt{P_1}}^{\wt{R}}\cap\fs)(F)} \sum_{\wh{Y}\in((\ov{\fn}_{\wt{R}}^{\wt{P_2}})'\cap\fs)(F)} |\Phi_\xi^{n_2 n a_1 a m k,R}(\wh{Y})| e^{-2\rho_{P_0}(H_{P_0}(a_1 a))} dn_2 dn da_1 da dm dk,
  \end{split}\]
  where $c_1:=\vol(K\cap M_{P_1}(\BA)^1)$ is a constant independent of $T$. Here we use the notation $[cpt\subseteq *]$ for denoting a compact subset in $*$ independent of $T$. 

  We claim that for $n_2\in N_{P_2}(\BA)$,
  $$ \Phi_{\xi}^{n_2 x,R}(\wh{Y})=\Phi_{\xi}^{x,R}(\wh{Y}). $$
  In fact, let $U_2:=\Ad(n_2^{-1})(\xi)-\xi.$ Then
  \[\begin{split}
    \Phi_{\xi}^{n_2 x,R}(\wh{Y})&=\int_{(\fn_{\wt{R}}\cap\fs)(\BA)} f(\Ad (n_2 x)^{-1}(\xi+U))\Psi(\langle U,\wh{Y}\rangle) dU \\
    &=\int_{(\fn_{\wt{R}}\cap\fs)(\BA)} f(\Ad(x^{-1})(\xi+U_2+\Ad(n_2^{-1})(U)))\Psi(\langle U,\wh{Y}\rangle) dU.
  \end{split}\]
  As both of $U_2$ and $\Ad(n_2^{-1})(U)-U$ belong to $(\fn_{\wt{P_2}}\cap\fs)(\BA)$, we get
  $$ \langle U_2+\Ad(n_2^{-1})(U)-U,\wh{Y}\rangle=0. $$
  Hence
  $$ \Phi_{\xi}^{n_2 x,R}(\wh{Y})=\int_{(\fn_{\wt{R}}\cap\fs)(\BA)} f(\Ad(x^{-1})(\xi+U_2+\Ad(n_2^{-1})(U)))\Psi(\langle U_2+\Ad(n_2^{-1})(U),\wh{Y}\rangle) dU. $$
  Since the change of variables $U_2+\Ad(n_2^{-1})(U)\mapsto U$ does not change the Haar measure, we proved our claim.

  By this claim, we have
$$    \Phi_\xi^{n_2 n a_1 a m k,R}(\wh{Y})=\Phi_\xi^{n a_1 a m k,R}(\wh{Y}) =\Phi_\xi^{(a_1 a)(a_1 a)^{-1}n (a_1 a) m k,R}(\wh{Y}). $$
  Applying change of variables $\Ad(a_1 a)^{-1}(U)\mapsto U$ and the fact that
  $$ \langle U,\wh{Y}\rangle=\langle \Ad(a_1 a)^{-1}(U), \Ad(a_1 a)^{-1}(\wh{Y})\rangle, $$
  we deduce that
  $$ \Phi_\xi^{n_2 n a_1 a m k,R}(\wh{Y})=e^{2\rho_{R,\fs}(H_{P_0}(a_1 a))} \Phi_{\Ad(a_1 a)^{-1}(\xi)}^{(a_1 a)^{-1}n (a_1 a) m k,R}(\Ad(a_1 a)^{-1}(\wh{Y})). $$
Recall that $\rho_{R,\fs}=\rho_R$ by Corollary \ref{corequalwt2.1}. 
  From the reduction theory (see \cite[p. 944]{MR518111}), for $a_1$ satisfying $\sigma_{P_1}^{P_2}(H_{P_1}(a_1)-T)\neq 0$, we know that $\Ad(a_1 a)^{-1}(n)$ belongs to a compact subset independent of $T$. To sum up,
  \[\begin{split}
    &\int_{P_1(F)\bs H(\BA)^1} \chi_{P_1, P_2}^T(x) \sum_{\xi\in(\wt{\fm}_{\wt{P_1}}^{\wt{R}}\cap\fs)(F)} \sum_{\wh{Y}\in((\ov{\fn}_{\wt{R}}^{\wt{P_2}})'\cap\fs)(F)} |\Phi_\xi^{x,R}(\wh{Y})| dx \\
    \leq &c_2 \sup_{y\in \Gamma} \int_{A_{P_0}^{P_1,\infty}(t_0,T)} \int_{A_{P_1}^{H,\infty}} e^{(2\rho_R-2\rho_{P_0})(H_{P_0}(a_1 a))} \sigma_{P_1}^{P_2}(H_{P_1}(a_1)-T) \\
    & \cdot \sum_{\xi\in(\wt{\fm}_{\wt{P_1}}^{\wt{R}}\cap\fs)(F)} \sum_{\wh{Y}\in((\ov{\fn}_{\wt{R}}^{\wt{P_2}})'\cap\fs)(F)} |\Phi_{\Ad(a_1 a)^{-1}(\xi)}^{y,R}(\Ad(a_1 a)^{-1}(\wh{Y}))| da_1 da,
  \end{split}\]
  where $c_2$ is a constant independent of $T$, and $\Gamma$ is a compact subset independent of $T$.

  Let $\CO_F$ denote the ring of integers of $F$. We fix an $F$-basis for each weight space for the action of $A_0(F)$ on $\fs(F)$. Since the function $f\in\CS(\fs(\BA))$ is compactly supported at finite places, there exists an $\CO_F$-scheme structure on such weight spaces independent of $T$ such that the sums over $\xi\in(\wt{\fm}_{\wt{P_1}}^{\wt{R}}\cap\fs)(F)$ and $\wh{Y}\in((\ov{\fn}_{\wt{R}}^{\wt{P_2}})'\cap\fs)(F)$ can be restricted to $\wt{\fm}_{\wt{P_1}}^{\wt{R}}(F)\cap\fs(\CO_F)$ and $(\ov{\fn}_{\wt{R}}^{\wt{P_2}})'(F)\cap\fs(\CO_F)$ 
respectively. To see this, one may consult \cite[\S1.9 and p. 363]{MR1893921} for details, and one needs to replace $\fm_R$ and $\fn_R$ in {\it{loc. cit.}} by $\fm_{\wt{R}}\cap\fs$ and $\fn_{\wt{R}}\cap\fs$ respectively. 

Fix an $\BR$-basis $\{e_1, \cdots, e_\ell\}$ of the $\BR$-linear space $\fs(F\otimes_{\BQ} \BR)$, whose dimension is denoted by $\ell$, consisting of eigenvectors for the action of $A_{P_0}^\infty$. Let $\|\cdot\|$ be the standard Euclidean norm with respect to this basis. 
Consider a sufficiently large positive integer $k$ to be made precise later. Thanks to Proposition \ref{equalwt2.1}, as in \cite[(4.10) in p. 372]{MR1893921} (whose proof is included in \cite[p. 946-947]{MR518111}), there exists an even integer $m\geq 0$, a real number $k_\alpha\geq 0$ for each $\alpha\in\Delta_{P_0}^{P_2}$, and a real number $c_3>0$ satisfying the following conditions: 
\begin{enumerate}[\indent (1)]
\item if $R=P_2$, $m=0$; 
\end{enumerate}
\begin{enumerate}[\indent (2)]
\item for all $\alpha\in\Delta_{P_0}^{P_2}-\Delta_{P_0}^{R}, k_\alpha\geq k$; 
\end{enumerate}
\begin{enumerate}[\indent (3)]
\item for all $a_0\in A_{P_0}^\infty(P_2,t_0)$, 
\begin{equation}\label{chau4.10}
 \sum_{\wh{Y}\in (\ov{\fn}_{\wt{R}}^{\wt{P_2}})'(F)\cap\fs(\CO_F)} \|\Ad(a_0^{-1})(\wh{Y})\|^{-m}\leq c_3 \prod_{\alpha\in\Delta_{P_0}^{P_2}} e^{-k_\alpha \alpha(H_{P_0}(a_0))}. 
\end{equation}
\end{enumerate}
We fix such data. 

Any differential operator $\partial$ on $\fs(F\otimes_\BQ \BR)$ can be extended to $\fs(\BA)$ by $\partial(f_\infty\otimes\chi^\infty):=(\partial f_\infty)\otimes\chi^\infty$, where we use the notation in Section \ref{BSandHaar2}. Denote
  $$ \Phi_\xi^{x,R,\partial}(\wh{Y}):=\int_{(\fn_{\wt{R}}\cap\fs)(\BA)} (\partial f)(\Ad(x^{-1})(\xi+U)) \Psi(\langle U,\wh{Y}\rangle) dU. $$
For a multi-index $\overrightarrow{i}=(i_1,\cdots, i_\ell)\in\BZ_{\geq0}^\ell$, denote by $\partial^{\overrightarrow{i}}:=(\frac{\partial}{\partial e_1})^{i_1}\cdots(\frac{\partial}{\partial e_\ell})^{i_\ell}$ the corresponding differential operator on $\fs(F\otimes_\BQ \BR)$. 
  Using integration by parts, we deduce that there exists a differential operator $\partial^{(m)}$ on $\fs(F\otimes_\BQ \BR)$ satisfying the following two conditions. 
\begin{enumerate}[\indent (1)]
\item $\partial^{(m)}$ is a finite $\BZ$-linear combination of $\partial^{\overrightarrow{i}}$'s with the properties: 
\begin{itemize}
	\item the sum of components of $\overrightarrow{i}$ is $m$; 
	\item all components of $\overrightarrow{i}$ are even integers;  
	\item all non-zero components of $\overrightarrow{i}$ correspond to eigenvectors lying in $(\fn_{\wt{R}}^{\wt{P_2}}\cap\fs)(F\otimes_\BQ\BR)$. 
\end{itemize}
\end{enumerate}
\begin{enumerate}[\indent (2)]
\item For $\wh{Y}\neq 0$, we have 
  $$ |\Phi_{\xi}^{y,R}(\wh{Y})|=\|\wh{Y}\|^{-m} |\Phi_{\xi}^{y,R,\Ad(y^{-1})\partial^{(m)}}(\wh{Y})|. $$
\end{enumerate}
We fix such a $\partial^{(m)}$. Suppose that $\partial^{(m)}=\sum\limits_{\overrightarrow{i}\in\CI}r_{\overrightarrow{i}} \partial^{\overrightarrow{i}}$, where $\CI$ is a finite set of multi-indices and  $r_{\overrightarrow{i}}\in\BZ$. Then there exists a continuous function $c_4(y)>0$ of $y$ such that for $\wh{Y}\neq 0$, we have
	$$ |\Phi_{\Ad(a_1 a)^{-1}(\xi)}^{y,R}(\Ad(a_1 a)^{-1}(\wh{Y}))|\leq c_4(y)\|\Ad(a_1 a)^{-1}(\wh{Y})\|^{-m} \sum_{\overrightarrow{i}\in\CI} |\Phi_{\Ad(a_1 a)^{-1}(\xi)}^{y,R,\partial^{\overrightarrow{i}}}(\Ad(a_1 a)^{-1}(\wh{Y}))|. $$

  For $\mu\in\Phi(A_0, \fm_{\wt{R}}\cap\fs)$ (refer to Section \ref{explicitdescription} for the notation), denote by $(\fm_{\wt{R}}\cap\fs)_\mu$ the corresponding weight space. From \cite[\S41]{MR0165033}, there exists a function $\phi_\mu\in\CS((\fm_{\wt{R}}\cap\fs)_\mu(\BA))$ for each $\mu\in\Phi(A_0, \fm_{\wt{R}}\cap\fs)$ and a function $\phi_{\fn_{\wt{R}}\cap\fs}\in\CS((\fn_{\wt{R}}\cap\fs)(\BA))$ such that for all $\xi+U\in(\fm_{\wt{R}}\cap\fs)(\BA)\oplus(\fn_{\wt{R}}\cap\fs)(\BA)$ and $y\in\Gamma$,
  $$ \sum_{\overrightarrow{i}\in\CI} |(\partial^{\overrightarrow{i}} f)(\Ad(y^{-1})(\xi+U))|\leq \left(\prod_{\mu\in\Phi(A_0, \fm_{\wt{R}}\cap\fs)}\phi_\mu(\xi_\mu)\right) \phi_{\fn_{\wt{R}}\cap\fs}(U), $$
where $\xi_\mu$ denotes the projection to $(\fm_{\wt{R}}\cap\fs)_\mu(\BA)$ of $\xi$. 

 From \cite[p. 375]{MR1893921}, when $\sigma_{P_1}^{P_2}(H_{P_1}(a_1)-T)\neq 0$, we have $\alpha(H_{P_0}(a_1a))>t_0$ for all $\alpha\in\Delta_{P_0}^{P_2}$. Hence, 
  \[\begin{split}
    &\sum_{\xi\in(\wt{\fm}_{\wt{P_1}}^{\wt{R}}\cap\fs)(F)} \sum_{\wh{Y}\in((\ov{\fn}_{\wt{R}}^{\wt{P_2}})'\cap\fs)(F)} |\Phi_{\Ad(a_1 a)^{-1}(\xi)}^{y,R}(\Ad(a_1 a)^{-1}(\wh{Y}))| \\
    = &\sum_{\xi\in\wt{\fm}_{\wt{P_1}}^{\wt{R}}(F)\cap\fs(\CO_F)} \sum_{\wh{Y}\in(\ov{\fn}_{\wt{R}}^{\wt{P_2}})'(F)\cap\fs(\CO_F)} |\Phi_{\Ad(a_1 a)^{-1}(\xi)}^{y,R}(\Ad(a_1 a)^{-1}(\wh{Y}))| \\
    \leq &\sum_{\xi\in\wt{\fm}_{\wt{P_1}}^{\wt{R}}(F)\cap\fs(\CO_F)} \sum_{\wh{Y}\in(\ov{\fn}_{\wt{R}}^{\wt{P_2}})'(F)\cap\fs(\CO_F)} c_4(y) \|\Ad(a_1 a)^{-1}(\wh{Y})\|^{-m} \sum_{\overrightarrow{i}\in\CI} |\Phi_{\Ad(a_1 a)^{-1}(\xi)}^{y,R,\partial^{\overrightarrow{i}}}(\Ad(a_1 a)^{-1}(\wh{Y}))| \\
    \leq &c_5 \sum_{\xi\in\wt{\fm}_{\wt{P_1}}^{\wt{R}}(F)\cap\fs(\CO_F)} \left(\prod_{\mu\in\Phi(A_0, \fm_{\wt{R}}\cap\fs)}\phi_\mu(\mu(a_1 a)^{-1}\xi_\mu)\right) \cdot \sum_{\wh{Y}\in(\ov{\fn}_{\wt{R}}^{\wt{P_2}})'(F)\cap\fs(\CO_F)} \|\Ad(a_1 a)^{-1}(\wh{Y})\|^{-m} \\
    \leq &c_5 c_3 \sum_{\xi\in\wt{\fm}_{\wt{P_1}}^{\wt{R}}(F)\cap\fs(\CO_F)} \left(\prod_{\mu\in\Phi(A_0, \fm_{\wt{R}}\cap\fs)}\phi_\mu(\mu(a_1 a)^{-1}\xi_\mu)\right) \cdot \prod_{\alpha\in\Delta_{P_0}^{P_2}}e^{-k_\alpha \alpha(H_{P_0}(a_1 a))},
  \end{split}\]
  where $c_5:=\sup\limits_{y\in\Gamma}c_4(y) \int_{(\fn_{\wt{R}}\cap\fs)(\BA)}\phi_{\fn_{\wt{R}}\cap\fs}(U)dU$; in the last inequality, we have applied 
(\ref{chau4.10}) to $a_0=a_1 a$. 
We deduce that 
  \[\begin{split}
    &\int_{P_1(F)\bs H(\BA)^1} \chi_{P_1, P_2}^T(x) \sum_{\xi\in(\wt{\fm}_{\wt{P_1}}^{\wt{R}}\cap\fs)(F)} \sum_{\wh{Y}\in((\ov{\fn}_{\wt{R}}^{\wt{P_2}})'\cap\fs)(F)} |\Phi_\xi^{x,R}(\wh{Y})| dx \\
    \leq &c_2 c_5 c_3 \int_{A_{P_0}^{P_1,\infty}(t_0,T)} \int_{A_{P_1}^{H,\infty}} e^{(2\rho_R-2\rho_{P_0})(H_{P_0}(a_1 a))} \sigma_{P_1}^{P_2}(H_{P_1}(a_1)-T) \\
    & \cdot \sum_{\xi\in\wt{\fm}_{\wt{P_1}}^{\wt{R}}(F)\cap\fs(\CO_F)} \left(\prod_{\mu\in\Phi(A_0, \fm_{\wt{R}}\cap\fs)}\phi_\mu(\mu(a_1 a)^{-1}\xi_\mu)\right) \cdot \prod_{\alpha\in\Delta_{P_0}^{P_2}}e^{-k_\alpha \alpha(H_{P_0}(a_1 a))} da_1 da.
  \end{split}\]

Denote by $\Sigma_{P_0}^{\fm_{\wt{R}}\cap\fs}$ the positive weights of $\fm_{\wt{R}}\cap\fs$ under the action of $A_0$. Consider the subsets $S$ of $\Sigma_{P_0}^{\fm_{\wt{R}}\cap\fs}$ with the following property: for all $\alpha\in\Delta_{P_0}^R-\Delta_{P_0}^{P_1}$, there exists $\mu\in S$ such that its $\alpha$-coordinate is $>0$. Then
  \[\begin{split}
    &\sum_{\xi\in\wt{\fm}_{\wt{P_1}}^{\wt{R}}(F)\cap\fs(\CO_F)} \left(\prod_{\mu\in\Phi(A_0, \fm_{\wt{R}}\cap\fs)}\phi_\mu(\mu(a_1 a)^{-1}\xi_\mu)\right) \\
    \leq &\sum_{S}\left[\prod_{\mu\in S} \left(\sum_{\xi_{-}\in(\fm_{\wt{R}}\cap\fs)_{-\mu}(\CO_F)-\{0\}} \phi_{-\mu}(\mu(a_1a)\xi_{-})\right)\right] \left[\prod_{\mu\in\Sigma_{P_0}^{\fm_{\wt{R}}\cap\fs}} \left(\sum_{\xi_{+}\in(\fm_{\wt{R}}\cap\fs)_{\mu}(\CO_F)} \phi_{\mu}(\mu(a_1 a)^{-1}\xi_{+})\right)\right] \\
    &\cdot \left[\sum_{\xi_{0}\in(\fm_{\wt{R}}\cap\fs)_1
(\CO_F)} \phi_{0}(\xi_{0})\right].
  \end{split}\]
Denote by $\Sigma_{P_0}^{\fm_R}$ the positive weights of $\fm_R$ under the action of $A_0$. From Proposition \ref{equalwt2.1}, we know that $\Sigma_{P_0}^{\fm_{\wt{R}}\cap\fs}=\Sigma_{P_0}^{\fm_R}$ and that each weight has the same multiplicity in $\fm_{\wt{R}}\cap\fs$ and $\fm_R$. From now on, we are in exactly the same situation as in \cite[p. 373]{MR1893921} and able to borrow the rest of its proof to conclude.
\end{proof}


\section{\textbf{Polynomial distributions}}

Let $T\in T_+ +\fa_{P_0}^+$ and $\fo\in\CO$. For $f\in\CS(\fs(\BA))$, define
\begin{equation}\label{Jo2}
 J_\fo^{H,T}(f):=\int_{H(F)\bs H(\BA)^1} k_{f,\fo}^T(x) dx 
\end{equation}
and
\begin{equation}\label{J2}
 J^{H,T}(f):=\int_{H(F)\bs H(\BA)^1} k_{f}^T(x) dx, 
\end{equation}
where $k_{f,\fo}^T(x)$ and $k_{f}^T(x)$ are defined by (\ref{deftrunc2.1}) and (\ref{deftrunc2.2}) respectively. From Theorem \ref{convergence2}, we know that $J_\fo^{H,T}$ and $J^{H,T}$ are well-defined distributions on $\CS(\fs(\BA))$. We also have
$$ J^{H,T}(f)=\sum_{\fo\in\CO}J_\fo^{H,T}(f), $$
which is an analogue of the geometric side of Arthur's trace formula. In this section, we shall prove that the functions $T\mapsto J_\fo^{H,T}(f)$ and $T\mapsto J^{H,T}(f)$ can be extended to polynomials in $T\in\fa_{P_0}$ (see Corollary \ref{polynomial2} below), whose constant terms will be denoted by $J_\fo^H(f)$ and $J^H(f)$ respectively. 

Let us begin with a generalization of our results in last section. Let $Q$ be a standard parabolic subgroup of $H$. Recall the two cases studied in Section \ref{explicitdescription}. In \textbf{Case I}, we have
$$ M_Q\simeq \Res_{E/F}GL_{n_1,D'}\times\cdot\cdot\cdot\times \Res_{E/F}GL_{n_l,D'} $$
and 
$$ M_{\wt{Q}}\simeq GL_{n_1,D}\times\cdot\cdot\cdot\times GL_{n_l,D},  $$
where $\sum\limits_{i=1}^l n_i=n$. In \textbf{Case II}, we have
$$ M_Q\simeq \Res_{E/F}GL_{\frac{n_1}{2},D\otimes_F E}\times\cdot\cdot\cdot\times \Res_{E/F}GL_{\frac{n_l}{2},D\otimes_F E} $$
and 
$$ M_{\wt{Q}}\simeq GL_{n_1,D}\times\cdot\cdot\cdot\times GL_{n_l,D},  $$
where $n_i$ is even for all $1\leq i\leq l$ and $\sum\limits_{i=1}^l n_i=n$. In either case of the two, the tangent space of $M_{\wt{Q}}\cap S$ at the neutral element is $\fm_{\wt{Q}}\cap\fs$, on which $M_Q$ acts by conjugation. We remark that our results in last section can be generalized to the product setting here, whose proofs are similar and will be omitted. 
Define a relation of equivalence on $(\fm_{\wt{Q}}\cap\fs)(F)$ which is similar to that on $\fs(F)$ on each component. We denote by $\CO^{\fm_{\wt{Q}}\cap\fs}$ the set of equivalent classes for this relation. For $\fo\in\CO$, the intersection $\fo\cap\fm_{\wt{Q}}(F)$ is a finite (perhaps empty) union of classes $\fo_1,\cdot\cdot\cdot,\fo_t\in\CO^{\fm_{\wt{Q}}\cap\fs}$. Notice that there exists a bijection between the set of standard parabolic subgroups $P$ of $H$ contained in $Q$ and the set of standard parabolic subgroups $P^\ast$ of $M_Q$ (namely $P_0\cap M_Q\subseteq P^\ast$) given by $P\mapsto P\cap M_Q$, whose inverse is given by $P^\ast\mapsto P^\ast N_Q$. Let $f^\ast\in\CS((\fm_{\wt{Q}}\cap\fs)(\BA))$, $P^\ast$ be a standard parabolic subgroup of $M_Q$ and $1\leq j\leq t$. For $x\in M_{P^\ast}(F)N_{P^\ast}(\BA)\bs M_Q(\BA)$, define
\begin{equation}\label{LevikfPo2}
k_{f^\ast, P^\ast, \fo_j}^{M_Q}(x):=\sum_{Y\in\fm_{\wt{P^\ast}}(F)\cap\fo_j} \int_{(\fn_{\wt{P^\ast}}\cap\fs)(\BA)} f^\ast(\Ad(x^{-1})(Y+U)) dU, 
\end{equation}
where $\wt{P^\ast}$ denotes the standard parabolic subgroup of $G$ corresponding to $P^\ast$ (see Section \ref{explicitdescription}). 
For $T\in\fa_0$ and $x\in M_{Q}(F)\bs M_{Q}(\BA)$, define
$$ k_{f^\ast,\fo_j}^{M_Q,T}(x):=\sum_{\{P^\ast:P_0\cap M_Q\subseteq P^\ast\}} (-1)^{\dim(A_{P^\ast}/A_{M_Q})} \sum_{\delta\in P^\ast(F)\bs M_{Q}(F)} \wh{\tau}_{P^\ast}^{M_Q}(H_{P^\ast}(\delta x)-T)\cdot k^{M_Q}_{f^\ast,P^\ast,\fo_j}(\delta x). $$
For $T\in T_+ +\fa_{P_0}^+$, define
$$ J_{\fo_j}^{M_Q,T}(f^\ast):=\int_{M_{Q}(F)\bs M_Q(\BA)^1} k_{f^\ast,\fo_j}^{M_Q,T}(x) dx. $$
Then we obtain a well-defined distribution $J_{\fo_j}^{M_Q,T}$ on $\CS((\fm_{\wt{Q}}\cap\fs)(\BA))$. 
Now we define
\begin{equation}\label{LeviJo2}
 J_\fo^{M_Q,T}:=\sum_{j=1}^{t} J_{\fo_j}^{M_Q,T}
\end{equation}
and
\begin{equation}\label{LeviJ2}
 J^{M_Q,T}:=\sum_{\fo\in\CO} J_\fo^{M_Q,T}. 
\end{equation}
For $f\in\CS(\fs(\BA))$, define $f_Q\in\CS((\fm_{\wt{Q}}\cap\fs)(\BA))$ by
\begin{equation}\label{equation2.2}
 \forall Y\in(\fm_{\wt{Q}}\cap\fs)(\BA), f_Q(Y):=\int_K\int_{(\fn_{\wt{Q}}\cap\fs)(\BA)} f(\Ad(k^{-1})(Y+V)) dVdk. 
\end{equation}

Let $T_1,T_2\in\fa_{P_0}$. As in \cite[\S2]{MR625344}, we define $\Gamma_P(T_1,T_2)\in\BR$ inductively on $\dim(A_P/A_{H})$ by setting
\begin{equation}\label{defGammaP2}
 \wh{\tau}_P^H(T_1-T_2)=\sum_{\{Q:P\subseteq Q\}} (-1)^{\dim(A_Q/A_{H})} \wh{\tau}_P^Q(T_1) \Gamma_Q(T_1,T_2) 
\end{equation}
for any standard parabolic subgroup $P$ of $H$. This definition can be explicitly given by \cite[(2.1) in p. 13]{MR625344} and only depends on the projections of $T_1,T_2$ onto $\fa_P^{H}$.

\begin{lem}
  Let $T_2\in\fa_{P_0}$ and Q be a standard parabolic subgroup of $H$. The function $T_1\mapsto\Gamma_Q(T_1,T_2)$ is compactly supported on $\fa_Q^{H}$. Moreover, the function $T_2\mapsto\int_{\fa_Q^{H}} \Gamma_Q(T_1,T_2) dT_1$ is a homogeneous polynomial in $T_2$ of degree $\dim(A_Q/A_{H})$.
\end{lem}

\begin{proof}
  This is \cite[Lemmas 2.1 and 2.2]{MR625344}.
\end{proof}

\begin{thm}\label{pol2}
  Let $T'\in T_+ +\fa_{P_0}^+$, $\fo\in\CO$ and $f\in\CS(\fs(\BA))$. Then for all $T\in T_+ +\fa_{P_0}^+$,
  $$ J_\fo^{H,T}(f)=\sum_{\{Q:P_0\subseteq Q\}} J_\fo^{M_Q,T'}(f_Q) \int_{\fa_Q^{H}} \Gamma_Q(T_1,T-T') dT_1, $$
  where $J_\fo^{H,T}, J_\fo^{M_Q,T'}$ and $f_Q$ are defined by the formulae (\ref{Jo2}), (\ref{LeviJo2}) and (\ref{equation2.2}) respectively. 
\end{thm}

\begin{coro}\label{polynomial2}
  Let $\fo\in\CO$ and $f\in\CS(\fs(\BA))$. Then the functions $T\mapsto J_\fo^{H,T}(f)$ and $T\mapsto J^{H,T}(f)$ (defined by (\ref{J2})) are the restriction of polynomials in $T$ of degree $\leq n-1$. 
Thus we can extend them to all $T\in\fa_{P_0}$.
\end{coro}

For $\fo\in\CO$ and $f\in\CS(\fs(\BA))$, we denote by $J_\fo^{H}(f)$ and $J^{H}(f)$ the constant terms of $J_\fo^{H,T}(f)$ and $J^{H,T}(f)$ respectively. 

\begin{remark}
  We fix $M_0$ and $M_{\wt{0}}$ which are minimal Levi subgroups of $H$ and $G$ respectively. The distributions $J_\fo^{H}(f)$ and $J^{H}(f)$ are independent of the choice of the minimal parabolic subgroup $P_0\supseteq M_0$ of $H$. In fact, the argument of \cite[Proposition 4.6]{MR1893921} after some minor modifications applies here because elements in $\Omega^H$ have representatives in $H(F)\cap K$ in our cases.
\end{remark}

\begin{proof}[Proof of Theorem \ref{pol2}]
  Let $P$ be any standard parabolic subgroup of $H$, $\delta\in P(F)\bs H(F)$ and $x\in H(\BA)^1$. By substituting $T_1=H_P(\delta x)-T'$ and $T_2=T-T'$ in (\ref{defGammaP2}), we have
  $$ \wh{\tau}_P^H(H_P(\delta x)-T)=\sum_{\{Q:P\subseteq Q\}} (-1)^{\dim(A_Q/A_{H})} \wh{\tau}_P^Q(H_P(\delta x)-T') \Gamma_Q(H_P(\delta x)-T',T-T'). $$
  Then
  \[\begin{split}
    J_\fo^{H,T}(f)=&\int_{H(F)\bs H(\BA)^1} \left(\sum_{\{P: P_0\subseteq P\}} (-1)^{\dim(A_P/A_{H})} \sum_{\delta\in P(F)\bs H(F)} \wh{\tau}_P^H(H_{P}(\delta x)-T)\cdot k_{f,P,\fo}(\delta x)\right) dx \\
    =&\int_{H(F)\bs H(\BA)^1} \sum_{\{P: P_0\subseteq P\}} (-1)^{\dim(A_P/A_{H})} \sum_{\delta\in P(F)\bs H(F)} \\ &\left(\sum_{\{Q:P\subseteq Q\}} (-1)^{\dim(A_Q/A_{H})} \wh{\tau}_P^Q(H_P(\delta x)-T') \Gamma_Q(H_P(\delta x)-T',T-T')\right) k_{f,P,\fo}(\delta x) dx.
  \end{split}\]
  By exchanging the order of two sums over $P$ and $Q$ and decomposing the sum over $P(F)\bs H(F)$ into two sums over $P(F)\bs Q(F)$ and $Q(F)\bs H(F)$, we obtain
  \[\begin{split}
    J_\fo^{H,T}(f)=&\sum_{\{Q:P_0\subseteq Q\}} \int_{H(F)\bs H(\BA)^1} \sum_{\{P:P_0\subseteq P\subseteq Q\}} (-1)^{\dim(A_P/A_{Q})} \sum_{\delta'\in Q(F)\bs H(F)} \sum_{\delta\in P(F)\bs Q(F)} \\
    &\wh{\tau}_P^Q(H_P(\delta\delta' x)-T') \Gamma_Q(H_P(\delta\delta' x)-T',T-T') k_{f,P,\fo}(\delta\delta' x) dx.
  \end{split}\]
  Combining the integral over $H(F)\bs H(\BA)^1$ and the sum over $Q(F)\bs H(F)$ into an integral over $Q(F)\bs H(\BA)^1$, and noticing that
  $$ P(F)\bs Q(F)\simeq (P(F)\cap M_Q(F))\bs M_Q(F), $$
  we have
  \[\begin{split}
    J_\fo^{H,T}(f)=&\sum_{\{Q:P_0\subseteq Q\}} \int_{Q(F)\bs H(\BA)^1} \sum_{\{P:P_0\subseteq P\subseteq Q\}} (-1)^{\dim(A_P/A_{Q})} \sum_{\delta\in (P(F)\cap M_Q(F))\bs M_Q(F)} \\
    &\wh{\tau}_P^Q(H_P(\delta x)-T') \Gamma_Q(H_P(\delta x)-T',T-T') k_{f,P,\fo}(\delta x) dx.
  \end{split}\]

By Iwasawa decomposition and our choice of measures, the integral over $x\in Q(F)\bs H(\BA)^1$ can be decomposed as integrals over 
$$ (n,a,m,k)\in N_Q(F)\bs N_Q(\BA)\times A_Q^{H,\infty}\times M_Q(F)\bs M_Q(\BA)^1\times K, $$
and we get 
\begin{equation}\label{equintegrand}
  \begin{split}
    J_\fo^{H,T}(f)=&\sum_{\{Q:P_0\subseteq Q\}} \int_K\int_{M_Q(F)\bs M_Q(\BA)^1}\int_{A_Q^{H,\infty}}\int_{N_Q(F)\bs N_Q(\BA)} \sum_{\{P:P_0\subseteq P\subseteq Q\}} (-1)^{\dim(A_P/A_{Q})} \\
    &\sum_{\delta\in (P(F)\cap M_Q(F))\bs M_Q(F)} \wh{\tau}_P^Q(H_P(\delta namk)-T') \Gamma_Q(H_P(\delta namk)-T',T-T') \\
    &\cdot k_{f,P,\fo}(\delta namk) e^{-2\rho_Q(H_{P_0}(a))}dndadmdk.
  \end{split}
\end{equation}
  We notice that
  $$ \wh{\tau}_P^Q(H_P(\delta namk)-T')=\wh{\tau}_P^Q(H_P(\delta m)+H_P(a)-T')=\wh{\tau}_P^Q(H_P(\delta m)-T') $$
  and that
  $$ \Gamma_Q(H_P(\delta namk)-T',T-T')=\Gamma_Q(H_Q(\delta namk)-T',T-T')=\Gamma_Q(H_Q(a)-T',T-T'). $$
  Additionally, using $A_Q=A_{\wt{Q}}$ and change of variables, we see that
  \[\begin{split}
    k_{f,P,\fo}(\delta namk)&=\sum_{Y\in\fm_{\wt{P}}(F)\cap\fo} \int_{(\fn_{\wt{P}}\cap\fs)(\BA)} f(\Ad(\delta namk)^{-1}(Y+U)) dU \\
    &=\sum_{Y\in\fm_{\wt{P}}(F)\cap\fo} \int_{(\fn_{\wt{P}}\cap\fs)(\BA)} f(\Ad(\delta a^{-1}namk)^{-1}(Y+a^{-1}Ua)) dU \\
    &=\sum_{Y\in\fm_{\wt{P}}(F)\cap\fo} \int_{(\fn_{\wt{P}}\cap\fs)(\BA)} f(\Ad(\delta a^{-1}namk)^{-1}(Y+U)) e^{2\rho_{Q,\fs}(H_{P_0}(a))} dU \\
    &=e^{2\rho_{Q,\fs}(H_{P_0}(a))}k_{f,P,\fo}(\delta a^{-1}namk),
  \end{split}\]
where $\rho_{Q,\fs}$ is defined in Section \ref{explicitdescription}. 
  Since $\delta a^{-1}na\delta^{-1}\in N_Q(\BA)\subseteq N_P(\BA)$ and $k_{f,P,\fo}$ is left invariant by $N_P(\BA)$, we deduce that
  $$ k_{f,P,\fo}(\delta namk)=e^{2\rho_{Q,\fs}(H_{P_0}(a))}k_{f,P,\fo}(\delta mk). $$
  To sum up, the integrand for the term indexed by $Q$ in \eqref{equintegrand} is independent of $n\in N_Q(F)\bs N_Q(\BA)$. Recall that we choose the Haar measure such that $\vol(N_Q(F)\bs N_Q(\BA))=1$. By Corollary \ref{corequalwt2.1}, the factors $e^{-2\rho_Q(H_{P_0}(a))}$ and $e^{2\rho_{Q,\fs}(H_{P_0}(a))}$ cancel each other, and then
  \[\begin{split}
    J_\fo^{H,T}(f)=&\sum_{\{Q:P_0\subseteq Q\}} \left(\int_{A_Q^{H,\infty}} \Gamma_Q(H_Q(a)-T',T-T') da\right) \int_{M_Q(F)\bs M_Q(\BA)^1} \sum_{\{P:P_0\subseteq P\subseteq Q\}} \\
    &(-1)^{\dim(A_P/A_{Q})} \sum_{\delta\in (P(F)\cap M_Q(F))\bs M_Q(F)} \wh{\tau}_P^Q(H_P(\delta m)-T') \left(\int_K k_{f,P,\fo}(\delta mk) dk\right)dm.
  \end{split}\]

  From the definition of the Haar measure on $A_Q^{H,\infty}$, we have
  \[\begin{split}
    \int_{A_Q^{H,\infty}} \Gamma_Q(H_Q(a)-T',T-T') da&:=\int_{\fa_Q^{H}} \Gamma_Q(T_1-T',T-T') dT_1 \\
    &=\int_{\fa_Q^{H}} \Gamma_Q(T_1,T-T') dT_1.
  \end{split}\]
Since $\fn_{\wt{P}}=\fn_{\wt{P}}^{\wt{Q}}\oplus\fn_{\wt{Q}}$, by change of variables, we deduce that
  \[\begin{split}
    k_{f,P,\fo}(\delta mk)&=\sum_{Y\in\fm_{\wt{P}}(F)\cap\fo} \int_{(\fn_{\wt{P}}^{\wt{Q}}\cap\fs)(\BA)} dU \int_{(\fn_{\wt{Q}}\cap\fs)(\BA)} f(\Ad(\delta mk)^{-1}(Y+U+V)) dV \\
    &=\sum_{Y\in\fm_{\wt{P}}(F)\cap\fo} \int_{(\fn_{\wt{P}}^{\wt{Q}}\cap\fs)(\BA)} dU \int_{(\fn_{\wt{Q}}\cap\fs)(\BA)} f(\Ad(k^{-1})(\Ad(\delta m)^{-1}(Y+U)+V)) dV,
  \end{split}\]
  where we need to verify that the change of variables $V\mapsto\Ad(\delta m)(V)$ does not change the Haar measure. 
In fact, the action of $M_Q$ on $\fn_{\wt{Q}}\cap\fs$ is algebraic and the Jacobian, which is the adelic absolute value of the determinant, is trivial on $M_Q(\BA)^1$. 
Then we can write
  \[\begin{split}
    \int_K k_{f,P,\fo}(\delta mk) dk&=\sum_{Y\in\fm_{\wt{P}}(F)\cap\fo} \int_{(\fn_{\wt{P}}^{\wt{Q}}\cap\fs)(\BA)} f_Q(\Ad(\delta m)^{-1}(Y+U)) dU \\
    &=\sum_{j=1}^{t} k_{f_Q,P\cap M_Q,\fo_j}^{M_Q} (\delta m)
  \end{split}\]
  by (\ref{LevikfPo2}). Now we can conclude by noting that
  \[\begin{split}
    J_{\fo}^{M_Q,T'}(f_Q)=&\sum_{j=1}^{t} \int_{M_Q(F)\bs M_Q(\BA)^1} \sum_{\{P:P_0\subseteq P\subseteq Q\}} (-1)^{\dim(A_{P\cap M_Q}/A_{M_Q})} \sum_{\delta\in (P(F)\cap M_Q(F))\bs M_Q(F)} \\
    &\wh{\tau}_{P\cap M_Q}^{M_Q}(H_{P\cap M_Q}(\delta m)-T') k_{f_Q,P\cap M_Q,\fo_j}^{M_Q}(\delta m) dm \\
    =&\int_{M_Q(F)\bs M_Q(\BA)^1} \sum_{\{P:P_0\subseteq P\subseteq Q\}} (-1)^{\dim(A_P/A_{Q})} \sum_{\delta\in (P(F)\cap M_Q(F))\bs M_Q(F)} \\
    &\wh{\tau}_P^Q(H_P(\delta m)-T') \left(\sum_{j=1}^{t} k_{f_Q,P\cap M_Q,\fo_j}^{M_Q}(\delta m)\right) dm.
  \end{split}\]
\end{proof}


\section{\textbf{Noninvariance}}

Let $Q$ be a standard parabolic subgroup of $H$ and $y\in H(\BA)^1$. For $f\in\CS(\fs(\BA))$, define $f_{Q,y}\in\CS((\fm_{\wt{Q}}\cap\fs)(\BA))$ by
\begin{equation}
 \forall Y\in(\fm_{\wt{Q}}\cap\fs)(\BA), f_{Q,y}(Y):=\int_{K} \int_{(\fn_{\wt{Q}}\cap\fs)(\BA)} f(\Ad(k^{-1})(Y+V)) p_{Q}(-H_Q(ky)) dVdk, 
\end{equation}
where for $T\in\fa_{P_0}$, we write
$$ p_Q(T):=\int_{\fa_Q^H} \Gamma_Q(T_1,T) dT_1. $$
We can also extend our results in last section to the product setting by the same argument. Let $\fo\in\CO$ and $f^\ast\in\CS((\fm_{\wt{Q}}\cap\fs)(\BA))$. For $T\in T_+ +\fa_{P_0}^+$, define $J_\fo^{M_Q,T}(f^\ast)$ and $J^{M_Q,T}(f^\ast)$ by (\ref{LeviJo2}) and (\ref{LeviJ2}) respectively. Then the functions $T\mapsto J_\fo^{M_Q,T}(f^\ast)$ and $T\mapsto J^{M_Q,T}(f^\ast)$ are the restriction of polynomials in $T$ and we can extend them to all $T\in\fa_{P_0}$. Denote by $J_\fo^{M_Q}(f^\ast)$ the constant term of $J_\fo^{M_Q,T}(f^\ast)$. 

\begin{prop}\label{noninvariance2}
For $f\in\CS(\fs(\BA))$ and $y\in H(\BA)^1$, we denote $f^y(x):=f(\Ad(y)(x))$. For $\fo\in\CO$, we have
$$ J_\fo^{H}(f^y)=\sum_{\{Q:P_0\subseteq Q\}} J_\fo^{M_Q}(f_{Q,y}). $$
\end{prop}

\begin{proof}
Let $T\in T_+ +\fa_{P_0}^+$. By definition, 
$$  J_\fo^{H,T}(f^y)=\int_{H(F)\bs H(\BA)^1} \left(\sum_{\{P: P_0\subseteq P\}} (-1)^{\dim(A_P/A_{H})} \sum_{\delta\in P(F)\bs H(F)} \wh{\tau}_P^H(H_{P}(\delta x)-T) \cdot k_{f^y,P,\fo}(\delta x)\right) dx, $$
where
$$ k_{f^y,P,\fo}(\delta x)=\sum_{Y\in\fm_{\wt{P}}(F)\cap\fo}\int_{(\fn_{\wt{P}}\cap\fs)(\BA)} f(\Ad(y)\Ad(\delta x)^{-1}(Y+U)) dU=k_{f,P,\fo}(\delta xy^{-1}). $$
By a change of variables, we get
$$ J_\fo^{H,T}(f^y)=\int_{H(F)\bs H(\BA)^1} \left(\sum_{\{P: P_0\subseteq P\}} (-1)^{\dim(A_P/A_{H})} \sum_{\delta\in P(F)\bs H(F)} \wh{\tau}_P^H(H_{P}(\delta xy)-T) \cdot k_{f,P,\fo}(\delta x)\right) dx. $$

For $x\in H(\BA)$ and $P$ a standard parabolic subgroup of $H$, let $k_P(x)$ be an element in $K$ satisfying $xk_P(x)^{-1}\in P(\BA)$. Then
$$ \wh{\tau}_P^H(H_{P}(\delta xy)-T)=\wh{\tau}_P^H(H_{P}(\delta x)-T+H_P(k_P(\delta x)y)). $$
By substituting $T_1=H_{P}(\delta x)-T$ and $T_2=-H_P(k_P(\delta x)y)$ in (\ref{defGammaP2}), we get
  $$ \wh{\tau}_P^H(H_P(\delta xy)-T)=\sum_{\{Q:P\subseteq Q\}} (-1)^{\dim(A_Q/A_{H})} \wh{\tau}_P^Q(H_{P}(\delta x)-T) \Gamma_Q(H_{P}(\delta x)-T,-H_P(k_P(\delta x)y)). $$
Then
\[\begin{split}
    J_\fo^{H,T}(f^y)=&\int_{H(F)\bs H(\BA)^1} \sum_{\{P: P_0\subseteq P\}} (-1)^{\dim(A_P/A_{H})} \sum_{\delta\in P(F)\bs H(F)} \\ & \left(\sum_{\{Q:P\subseteq Q\}} (-1)^{\dim(A_Q/A_{H})} \wh{\tau}_P^Q(H_{P}(\delta x)-T) \Gamma_Q(H_{P}(\delta x)-T,-H_P(k_P(\delta x)y))\right) \cdot k_{f,P,\fo}(\delta x) dx. 
  \end{split}\]
As in the proof of Theorem \ref{pol2}, by exchanging the order of sums, we deduce that 
  \[\begin{split}
    J_\fo^{H,T}(f^y)=&\sum_{\{Q:P_0\subseteq Q\}} \int_{Q(F)\bs H(\BA)^1} \sum_{\{P:P_0\subseteq P\subseteq Q\}} (-1)^{\dim(A_P/A_{Q})} \sum_{\delta\in (P(F)\cap M_{Q}(F))\bs M_{Q}(F)} \\
    &\wh{\tau}_P^Q(H_P(\delta x)-T) \Gamma_Q(H_P(\delta x)-T,-H_P(k_P(\delta x)y)) k_{f,P,\fo}(\delta x) dx.
  \end{split}\]

By Iwasawa decomposition and our choice of measures, the integral over $x\in Q(F)\bs H(\BA)^1$ can be decomposed as integrals over 
$$ (n,a,m,k)\in N_Q(F)\bs N_Q(\BA)\times A_Q^{H,\infty}\times M_Q(F)\bs M_Q(\BA)^1\times K, $$
and we obtain 
\begin{equation}\label{equintegrand(2)}
  \begin{split}
    J_\fo^{H,T}(f^y)=&\sum_{\{Q:P_0\subseteq Q\}} \int_{K}\int_{M_{Q}(F)\bs M_{Q}(\BA)^1}\int_{A_Q^{H,\infty}}\int_{N_{Q}(F)\bs N_{Q}(\BA)} \sum_{\{P:P_0\subseteq P\subseteq Q\}} (-1)^{\dim(A_P/A_{Q})} \\
    &\sum_{\delta\in (P(F)\cap M_{Q}(F))\bs M_{Q}(F)} \wh{\tau}_P^Q(H_P(\delta namk)-T) \Gamma_Q(H_P(\delta namk)-T,-H_P(k_P(\delta namk)y)) \\
    &\cdot k_{f,P,\fo}(\delta namk) e^{-2\rho_{Q}(H_{P_0}(a))}dndadmdk.
  \end{split}
\end{equation}
As in the proof of Theorem \ref{pol2}, we see that
$$ \wh{\tau}_P^Q(H_P(\delta namk)-T)=\wh{\tau}_P^Q(H_P(\delta m)-T), $$
and that
$$ k_{f,P,\fo}(\delta namk)=e^{2\rho_{Q}(H_{P_0}(a))}k_{f,P,\fo}(\delta mk). $$
Additionally, 
\[\begin{split}
\Gamma_Q(H_P(\delta namk)-T,-H_P(k_P(\delta namk)y))&=\Gamma_Q(H_Q(\delta namk)-T,-H_Q(k_P(\delta namk)y)) \\
&=\Gamma_Q(H_Q(a)-T,-H_Q(k_Q(\delta namk)y)) \\
&=\Gamma_Q(H_Q(a)-T,-H_Q(ky)). 
\end{split}\]
In sum, the integrand for the term indexed by $Q$ in \eqref{equintegrand(2)} is independent of $n\in N_{Q}(F)\bs N_{Q}(\BA)$. Recall that we choose the Haar measure such that $\vol(N_{Q}(F)\bs N_{Q}(\BA))=1$. Then
  \[\begin{split}
    J_\fo^{H,T}(f^y)=&\sum_{\{Q:P_0\subseteq Q\}} \int_{K}\int_{M_{Q}(F)\bs M_{Q}(\BA)^1}\int_{A_Q^{H,\infty}}\sum_{\{P:P_0\subseteq P\subseteq Q\}} (-1)^{\dim(A_P/A_{Q})} \sum_{\delta\in (P(F)\cap M_{Q}(F))\bs M_{Q}(F)} \\
    &\wh{\tau}_P^Q(H_P(\delta m)-T) \Gamma_Q(H_Q(a)-T,-H_Q(ky)) k_{f,P,\fo}(\delta mk)dadmdk.
  \end{split}\]
First, let us consider the integral on $A_Q^{H,\infty}$, which is
  \[\begin{split}
    \int_{A_Q^{H,\infty}} \Gamma_Q(H_Q(a)-T,-H_Q(ky)) da :=&\int_{\fa_Q^{H}} \Gamma_Q(T_1-T,-H_Q(ky)) dT_1 \\
    =&\int_{\fa_Q^{H}} \Gamma_Q(T_1,-H_Q(ky)) dT_1 \\
    =&p_{Q}(-H_Q(ky)).
  \end{split}\]
Next, we compute the integral on $K$, which is
$$ \int_{K} k_{f,P,\fo}(\delta mk) p_{Q}(-H_Q(ky)) dk. $$
As in the proof of Theorem \ref{pol2}, we see that
$$ k_{f,P,\fo}(\delta mk)=\sum_{Y\in\fm_{\wt{P}}(F)\cap\fo} \int_{(\fn_{\wt{P}}^{\wt{Q}}\cap\fs)(\BA)} dU \int_{(\fn_{\wt{Q}}\cap\fs)(\BA)} f(\Ad(k^{-1})(\Ad(\delta m)^{-1}(Y+U)+V)) dV, $$
so we can write
  \[\begin{split}
    \int_{K} k_{f,P,\fo}(\delta mk) p_{Q}(-H_Q(ky)) dk=&\sum_{Y\in\fm_{\wt{P}}(F)\cap\fo} \int_{(\fn_{\wt{P}}^{\wt{Q}}\cap\fs)(\BA)} f_{Q,y}(\Ad(\delta m)^{-1}(Y+U)) dU \\
    =&\sum_{j=1}^{t} k_{f_{Q,y},P\cap M_Q,\fo_j}^{M_Q} (\delta m)
  \end{split}\]
by (\ref{LevikfPo2}). 
Therefore, we obtain
  \[\begin{split}
    J_\fo^{H,T}(f^y)=&\sum_{\{Q:P_0\subseteq Q\}} \int_{M_{Q}(F)\bs M_{Q}(\BA)^1} \sum_{\{P:P_0\subseteq P\subseteq Q\}} (-1)^{\dim(A_P/A_{Q})} \sum_{\delta\in (P(F)\cap M_{Q}(F))\bs M_{Q}(F)} \\ 
    &\wh{\tau}_P^Q(H_P(\delta m)-T)  \left(\sum_{j=1}^{t} k_{f_{Q,y},P\cap M_Q,\fo_j}^{M_Q} (\delta m)\right) dm.
  \end{split}\]
As in the proof of Theorem \ref{pol2}, we notice that
  \[\begin{split}
    J_{\fo}^{M_Q,T}(f_{Q,y})=&\int_{M_Q(F)\bs M_Q(\BA)^1} \sum_{\{P:P_0\subseteq P\subseteq Q\}} (-1)^{\dim(A_P/A_{Q})} \sum_{\delta\in (P(F)\cap M_Q(F))\bs M_Q(F)} \\
    &\wh{\tau}_P^Q(H_P(\delta m)-T) \left(\sum_{j=1}^{t} k_{f_{Q,y},P\cap M_Q,\fo_j}^{M_Q}(\delta m)\right) dm. 
  \end{split}\]
Thus we deduce that
$$ J_\fo^{H,T}(f^y)=\sum_{\{Q:P_0\subseteq Q\}} J_\fo^{M_Q,T}(f_{Q,y}). $$
We may conclude by taking the constant terms of both sides. 
\end{proof}


\section{\textbf{An infinitesimal trace formula for $\fs//H$}}

Recall that for $f\in\CS(\fs(\BA))$, we have defined its Fourier transform $\hat{f}\in\CS(\fs(\BA))$ by (\ref{fouriertransform2}) and denoted the constant term of $J_\fo^{H,T}(f)$ by $J_\fo^{H}(f)$. 

\begin{thm}\label{infitf2}
For $f\in\CS(\fs(\BA))$, we have the equality, 
$$ \sum_{\fo\in\CO}J_\fo^{H}(f)=\sum_{\fo\in\CO}J_\fo^{H}(\hat{f}). $$
\end{thm}

	We can prove this equality using the argument in \cite[Th\'{e}or\`{e}me 4.5]{MR1893921} with the Poisson summation formula on $\fs(\BA)$, Corollary \ref{comparewithnaive2} and Corollary \ref{polynomial2}. The following alternative proof is suggested by the referee and does not involve the asymptotic estimation of $J^{H,T}(f)$. 

\begin{lem}\label{lemref}
	Let $f\in\CS(\fs(\BA))$, $P$ be a standard parabolic subgroup of $H$ and $x\in H(\BA)$. We have the equality 
	$$ k_{f,P}(x)=k_{\hat{f},P}(x), $$
	where 
	$$ k_{f,P}(x):=\sum_{\fo\in\CO} k_{f,P,\fo}(x)=\sum_{Y\in(\fm_{\wt{P}}\cap\fs)(F)}\int_{(\fn_{\wt{P}}\cap\fs)(\BA)} f(\Ad(x^{-1})(Y+U)) dU. $$
\end{lem}

\begin{proof}
	Define $g\in\CS(\fs(\BA))$ by $g(Y):=f(\Ad(x^{-1})(Y))$ for all $Y\in\fs(\BA)$. Then we have 
	$$ k_{f,P}(x)=\sum_{Y\in(\fm_{\wt{P}}\cap\fs)(F)}h(Y), $$
	where $h(Y):=\int_{(\fn_{\wt{P}}\cap\fs)(\BA)} g(Y+U) dU$, and 
	$$ k_{\hat{f},P}(x)=\sum_{Y\in(\fm_{\wt{P}}\cap\fs)(F)}i(Y), $$
	where $i(Y):=\int_{(\fn_{\wt{P}}\cap\fs)(\BA)} \hat{g}(Y+U) dU$. Note that $h,i\in\CS((\fm_{\wt{P}}\cap\fs)(\BA))$. 

	Denote by $\hat{h}$ the Fourier transform of $h$ on $(\fm_{\wt{P}}\cap\fs)(\BA)$. We claim that 
	$$ i=\hat{h}. $$
	The proof below of this statement is close to that of its local analogue (cf. \cite[Lemma 1.4]{MR1702257}). Denote by $\ov{\wt{P}}$ the semi-standard parabolic subgroup of $G$ opposite to $\wt{P}$. For $X_j\in\fs(\BA)$ with $j\in\{1,2\}$, we write $X_j=\ov{U}_j+Y_j+U_j$ where $\ov{U}_j\in(\fn_{\ov{\wt{P}}}\cap\fs)(\BA)$, $Y_j\in(\fm_{\wt{P}}\cap\fs)(\BA)$ and $U_j\in(\fn_{\wt{P}}\cap\fs)(\BA)$. Then 
	$$ \langle X_1,X_2\rangle=\langle\ov{U}_1,U_2\rangle+\langle Y_1,Y_2\rangle+\langle U_1,\ov{U}_2\rangle. $$
Hence 
	\[\begin{split}
	 \hat{g}(Y_1+U_1)&=\int_{\fs(\BA)} g(X_2)\Psi(\langle X_2,Y_1+U_1\rangle) dX_2 \\
					  &=\int_{(\fn_{\ov{\wt{P}}}\cap\fs)(\BA)\oplus(\fm_{\wt{P}}\cap\fs)(\BA)\oplus(\fn_{\wt{P}}\cap\fs)(\BA)} g(\ov{U}_2+Y_2+U_2)\Psi(\langle Y_1,Y_2\rangle+\langle U_1,\ov{U}_2 \rangle) dU_2dY_2d\ov{U}_2 \\
					  &=\int_{(\fn_{\ov{\wt{P}}}\cap\fs)(\BA)}\Psi(\langle U_1,\ov{U}_2 \rangle) d\ov{U}_2 \int_{(\fm_{\wt{P}}\cap\fs)(\BA)\oplus(\fn_{\wt{P}}\cap\fs)(\BA)} g(\ov{U}_2+Y_2+U_2)\Psi(\langle Y_1,Y_2\rangle) dU_2dY_2. \\
	\end{split}\]
Since $d\ov{U}_j$ and $dU_j$ are dual to each other with respect to $\langle\ov{U}_j,U_j\rangle$, by the Fourier inversion formula, we have 
	$$ i(Y_1)=\int_{(\fn_{\wt{P}}\cap\fs)(\BA)} \hat{g}(Y_1+U_1) dU_1=\int_{(\fm_{\wt{P}}\cap\fs)(\BA)\oplus(\fn_{\wt{P}}\cap\fs)(\BA)} g(Y_2+U_2)\Psi(\langle Y_1,Y_2\rangle) dU_2dY_2. $$
But 
	\[\begin{split}
	 \hat{h}(Y_1)&=\int_{(\fm_{\wt{P}}\cap\fs)(\BA)} h(Y_2)\Psi(\langle Y_2,Y_1\rangle) dY_2 \\
				&=\int_{(\fm_{\wt{P}}\cap\fs)(\BA)} \Psi(\langle Y_1,Y_2\rangle) dY_2 \int_{(\fn_{\wt{P}}\cap\fs)(\BA)} g(Y_2+U_2) dU_2. \\
	\end{split}\]
	Therefore, we have proved our claim $i=\hat{h}$. 

	To finish the proof of the lemma, it suffices to use the above claim and apply the Poisson summation formula on $(\fm_{\wt{P}}\cap\fs)(\BA)$. 
\end{proof}

\begin{proof}[Proof of Theorem \ref{infitf2}]
	By Lemma \ref{lemref}, we see that 
	$$ k_f^T(x)=k_{\hat{f}}^T(x) $$	
	for all $T\in T_+ +\fa_{P_0}^+$. Thus we have 
	$$ J^{H,T}(f)=J^{H,T}(\hat{f}). $$
  From
  $$ J^{H,T}(f)=\sum_{\fo\in\CO}J_\fo^{H,T}(f) $$
  and
  $$ J^{H,T}(\hat{f})=\sum_{\fo\in\CO}J_\fo^{H,T}(\hat{f}), $$
  we obtain
$$ \sum_{\fo\in\CO}J_\fo^{H,T}(f)=\sum_{\fo\in\CO}J_\fo^{H,T}(\hat{f}). $$
We can draw the conclusion by taking the constant terms of both sides. 
\end{proof}


\section{\textbf{The second modified kernel}}

Let $f\in\CS(\fs(\BA))$, $P$ be a standard parabolic subgroup of $H$ and $\fo\in\CO_{rs}$ (see Section \ref{inv2}). For $x\in P(F)\bs H(\BA)$, define
$$ j_{f,P,\fo}(x):=\sum_{Y\in\fm_{\wt{P}}(F)\cap\fo} \sum_{n\in N_P(F)} f(\Ad(nx)^{-1}(Y)). $$
For $T\in\fa_0$ and $x\in H(F)\bs H(\BA)$, define
$$ j_{f,\fo}^T(x):=\sum_{\{P: P_0\subseteq P\}} (-1)^{\dim(A_P/A_{H})} \sum_{\delta\in P(F)\bs H(F)} \wh{\tau}_P^H(H_{P}(\delta x)-T)\cdot j_{f,P,\fo}(\delta x). $$
By \cite[Lemma 5.1]{MR518111}, we know that the sum over $\delta\in P(F)\bs H(F)$ is finite. 

\begin{lem}\label{regsslemma2}
  Let $P$ be a standard parabolic subgroup of $H$ and $\fo\in\CO_{rs}$. For $Y\in\fm_{\wt{P}}(F)\cap\fo$, the map
  $$ N_P\ra\fn_{\wt{P}}\cap\fs, n\mapsto \Ad(n^{-1})(Y)-Y $$
  is an $F$-isomorphism of algebraic varieties and preserves the Haar measures on $\BA$-points. 
\end{lem}

\begin{proof}
Recall that there are two cases considered in Section \ref{explicitdescription}. First let us focus on \textbf{Case I}. In this case, we can suppose
$$ P=
\left( \begin{array}{cccc}
\Res_{E/F}GL_{n_1,D'} & \Res_{E/F}Mat_{n_1\times n_2,D'} & \cdots & \Res_{E/F}Mat_{n_1\times n_l,D'} \\
                      & \Res_{E/F}GL_{n_2,D'}            & \cdots & \Res_{E/F}Mat_{n_2\times n_l,D'} \\
                      &                                & \ddots & \vdots                         \\
                      &                                &        & \Res_{E/F}GL_{n_l,D'}            \\
\end{array} \right). $$
Then
$$ \wt{P}=
\left( \begin{array}{cccc}
GL_{n_1,D} & Mat_{n_1\times n_2,D} & \cdots & Mat_{n_1\times n_l,D} \\
                      & GL_{n_2,D}            & \cdots &Mat_{n_2\times n_l,D} \\
                      &                                & \ddots & \vdots                         \\
                      &                                &        &GL_{n_l,D}            \\
\end{array} \right). $$
We have chosen an element $\tau\in D^\times$ in Section \ref{explicitdescription}. Recall also Proposition \ref{equalwt2.2}. 

Let
$$ Y=
\left( \begin{array}{ccc}
  Y_1 &            &   \\
        &  \ddots &   \\
        &            & Y_l   \\
\end{array} \right)\in\fm_{\wt{P}}(F)\cap\fo, $$
where $Y_i\in GL_{n_i}(D')\tau$ for $1\leq i\leq l$, and
$$ n=
\left( \begin{array}{cccc}
 1 & n_{12} & \cdots & n_{1l} \\
    & 1         & \cdots & n_{2l} \\
    &            & \ddots & \vdots \\
    &            &           & 1        \\
\end{array} \right)\in N_{P}, $$
where $n_{ij}\in\Res_{E/F}Mat_{n_i\times n_j,D'}$ for $1\leq i<j\leq l$. Then
$$ Yn-nY=
\left( \begin{array}{cccc}
 0 & Y_1 n_{12}-n_{12} Y_2 & \cdots & Y_1 n_{1l}-n_{1l} Y_l \\
    & 0                                 & \cdots & Y_2 n_{2l}-n_{2l} Y_l \\
    &                                   & \ddots & \vdots \\
    &                                   &           & 0        \\
\end{array} \right)\in\fn_{\wt{P}}\cap\fs. $$

Now we claim that the morphism of $F$-affine spaces
\[\begin{split}
\Res_{E/F}Mat_{n_i\times n_j,D'} &\ra (\Res_{E/F}Mat_{n_i\times n_j,D'})\tau \\
n_{ij} &\mapsto Y_i n_{ij}-n_{ij} Y_j \\
\end{split}\]
induces an $F$-linear isomorphism on $F$-points. In fact, since it gives an $F$-linear map between finite dimensional linear spaces of the same dimension, we only need to prove that this map is injective. If $Y_i n_{ij}-n_{ij} Y_j=0$, then $Y_i^2 n_{ij}=Y_i n_{ij} Y_j=n_{ij} Y_j^2$. We view this as an equation of matrices with entries in $D'$ or its base change to an algebraic closure of $E$. Since $Y$ is regular semi-simple, $Y^2$ is regular semi-simple in $\fh(F)$ (viewed as a central simple algebra over $E$), so $Y_i^2$ and $Y_j^2$ have no common eigenvalue. By the classical theory of Sylvester equation, we know that $n_{ij}=0$ and conclude. 

Using this claim, we see that the map
$$ N_P\ra\fn_{\wt{P}}\cap\fs, n\mapsto Yn-nY $$
is an $F$-isomorphism of algebraic varieties and preserves the Haar measure on $\BA$-points. Notice that 
$$ \Ad(n^{-1})(Y)-Y=n^{-1}(Yn-nY)=
\left( \begin{array}{cccc}
 0 & Z_{12} & \cdots & Z_{1l} \\
    & 0                                 & \cdots & Z_{2l} \\
    &                                   & \ddots & \vdots \\
    &                                   &           & 0        \\
\end{array} \right)\in\fn_{\wt{P}}\cap\fs, $$
where $Z_{ij}\in(\Res_{E/F}Mat_{n_i\times n_j,D'})\tau$ is of the form 
$$ Y_i n_{ij}-n_{ij} Y_j+\sum\limits_{k>j}(\text{a polynomial of } n_{i'j'}, i'>i, j'\leq j \text{ or } i'\geq i, j'<j)\cdot(Y_k n_{kj}-n_{kj} Y_j) $$
for $1\leq i<j\leq l$. Therefore, the map
$$ N_P\ra\fn_{\wt{P}}\cap\fs, n\mapsto n\mapsto \Ad(n^{-1})(Y)-Y $$
is also an $F$-isomorphism of algebraic varieties and preserves the Haar measure on $\BA$-points. 

Next let us turn to \textbf{Case II} whose proof is close to the first one. In this case, we may suppose
$$ P=
\left( \begin{array}{cccc}
\Res_{E/F}GL_{\frac{n_1}{2},D\otimes_F E} & \Res_{E/F}Mat_{\frac{n_1}{2}\times\frac{n_2}{2},D\otimes_F E} & \cdots & \Res_{E/F}Mat_{\frac{n_1}{2}\times\frac{n_l}{2},D\otimes_F E} \\
                      & \Res_{E/F}GL_{\frac{n_2}{2},D\otimes_F E}            & \cdots & \Res_{E/F}Mat_{\frac{n_2}{2}\times\frac{n_l}{2},D\otimes_F E} \\
                      &                                & \ddots & \vdots                         \\
                      &                                &        & \Res_{E/F}GL_{\frac{n_l}{2},D\otimes_F E}            \\
\end{array} \right). $$
Then
$$ \wt{P}=
\left( \begin{array}{cccc}
GL_{n_1,D} & Mat_{n_1\times n_2,D} & \cdots & Mat_{n_1\times n_l,D} \\
                      & GL_{n_2,D}            & \cdots &Mat_{n_2\times n_l,D} \\
                      &                                & \ddots & \vdots                         \\
                      &                                &        &GL_{n_l,D}            \\
\end{array} \right). $$
We have chosen an element $\tau\in GL_2(D)$ in Section \ref{explicitdescription}. Recall again Proposition \ref{equalwt2.2}. 

Let
$$ Y=
\left( \begin{array}{ccc}
  Y_1 &            &   \\
        &  \ddots &   \\
        &            & Y_l   \\
\end{array} \right)\in\fm_{\wt{P}}(F)\cap\fo, $$
where $Y_i\in GL_{\frac{n_i}{2}}(D\otimes_F E)\tau$ for $1\leq i\leq l$, and
$$ n=
\left( \begin{array}{cccc}
 1 & n_{12} & \cdots & n_{1l} \\
    & 1         & \cdots & n_{2l} \\
    &            & \ddots & \vdots \\
    &            &           & 1        \\
\end{array} \right)\in N_{P}, $$
where $n_{ij}\in\Res_{E/F}Mat_{\frac{n_i}{2}\times \frac{n_j}{2},D\otimes_F E}$ for $1\leq i<j\leq l$. Then
$$ Yn-nY=
\left( \begin{array}{cccc}
 0 & Y_1 n_{12}-n_{12} Y_2 & \cdots & Y_1 n_{1l}-n_{1l} Y_l \\
    & 0                                 & \cdots & Y_2 n_{2l}-n_{2l} Y_l \\
    &                                   & \ddots & \vdots \\
    &                                   &           & 0        \\
\end{array} \right)\in\fn_{\wt{P}}\cap\fs. $$

As in the proof of the first case, we show that the morphism of $F$-affine spaces
\[\begin{split}
\Res_{E/F}Mat_{\frac{n_i}{2}\times \frac{n_j}{2},D\otimes_F E} &\ra (\Res_{E/F}Mat_{\frac{n_i}{2}\times \frac{n_j}{2},D\otimes_F E})\tau \\
n_{ij} &\mapsto Y_i n_{ij}-n_{ij} Y_j \\
\end{split}\]
induces an $F$-linear isomorphism on $F$-points. This implies that the map
$$ N_P\ra\fn_{\wt{P}}\cap\fs, n\mapsto Yn-nY $$
is an $F$-isomorphism of algebraic varieties and preserves the Haar measure on $\BA$-points. By an argument similar to that in the first case, we deduce that an analogous assertion is still true for the map $n\mapsto \Ad(n^{-1})(Y)-Y$. 
\end{proof}

\begin{thm}\label{secondkernel2}
  For all $T\in T_+ +\fa_{P_0}^+$ and $\fo\in\CO_{rs}$, we have
  $$ \int_{H(F)\bs H(\BA)^1} |j_{f,\fo}^T(x)| dx < \infty $$
  and
  $$ J_\fo^{H,T}(f)=\int_{H(F)\bs H(\BA)^1} j_{f,\fo}^T(x) dx. $$
\end{thm}

\begin{proof}
  As in the proof of Theorem \ref{convergence2}, by the left invariance of $j_{f,P,\fo}$ by $P(F)$, we reduce the first statement to
  $$ \int_{P_1(F)\bs H(\BA)^1} \chi_{P_1, P_2}^T(x) |j_{P_1, P_2, \fo}(x)| dx < \infty, $$
  where $P_1\subsetneq P_2$ is a pair of standard parabolic subgroups of $H$ and for $x\in P_1(F)\bs H(\BA)$, we put
  $$ j_{P_1, P_2, \fo}(x):=\sum_{\{P:P_1\subseteq P\subseteq P_2\}} (-1)^{\dim(A_P/A_{H})} j_{f,P,\fo}(x). $$
  Additionally,
  $$ j_{f,P,\fo}(x)=\sum_{\{R:P_1\subseteq R\subseteq P\}} \sum_{\xi\in\wt{\fm}_{\wt{P_1}}^{\wt{R}}(F)\cap\fo} \sum_{Y\in(\fn_{\wt{R}}^{\wt{P}}\cap\fs)(F)} \sum_{n\in N_P(F)} f(\Ad(nx)^{-1}(\xi+Y)), $$
where we use the notations $\wt{\fm}_{\wt{P_1}}^{\wt{R}}$ and $\fn_{\wt{R}}^{\wt{P}}$ in the proof of Proposition \ref{prop43}. 

  From Lemma \ref{regsslemma2}, we have
  \[\begin{split}
    j_{f,P,\fo}(x)&=\sum_{\{R:P_1\subseteq R\subseteq P\}} \sum_{\xi\in\wt{\fm}_{\wt{P_1}}^{\wt{R}}(F)\cap\fo} \sum_{Y\in(\fn_{\wt{R}}^{\wt{P}}\cap\fs)(F)} \sum_{u\in(\fn_{\wt{P}}\cap\fs)(F)} f(\Ad(x^{-1})(\xi+Y+u)) \\
    &=\sum_{\{R:P_1\subseteq R\subseteq P\}} \sum_{\xi\in\wt{\fm}_{\wt{P_1}}^{\wt{R}}(F)\cap\fo} \sum_{Y\in(\fn_{\wt{R}}\cap\fs)(F)} f(\Ad(x^{-1})(\xi+Y)).
  \end{split}\]
  Hence
  \[\begin{split}
    j_{P_1, P_2, \fo}(x)&=\sum_{\{P:P_1\subseteq P\subseteq P_2\}} (-1)^{\dim(A_P/A_{H})} \left(\sum_{\{R:P_1\subseteq R\subseteq P\}} \sum_{\xi\in\wt{\fm}_{\wt{P_1}}^{\wt{R}}(F)\cap\fo} \sum_{Y\in(\fn_{\wt{R}}\cap\fs)(F)} f(\Ad(x^{-1})(\xi+Y))\right) \\
    &=\sum_{\{R:P_1\subseteq R\subseteq P_2\}} \sum_{\xi\in\wt{\fm}_{\wt{P_1}}^{\wt{R}}(F)\cap\fo} \left(\sum_{\{P:R\subseteq P\subseteq P_2\}} (-1)^{\dim(A_P/A_{H})}\right) \sum_{Y\in(\fn_{\wt{R}}\cap\fs)(F)} f(\Ad(x^{-1})(\xi+Y)).
  \end{split}\]
  Using \cite[Proposition 1.1]{MR518111}, we get
  $$ j_{P_1, P_2, \fo}(x)=(-1)^{\dim(A_{P_2}/A_{H})} \sum_{\xi\in\wt{\fm}_{\wt{P_1}}^{\wt{P_2}}(F)\cap\fo} \sum_{Y\in(\fn_{\wt{P_2}}\cap\fs)(F)} f(\Ad(x^{-1})(\xi+Y)). $$
  Applying Lemma \ref{regsslemma2} again, we obtain
  $$ j_{P_1, P_2, \fo}(x)=(-1)^{\dim(A_{P_2}/A_{H})} \sum_{\xi\in\wt{\fm}_{\wt{P_1}}^{\wt{P_2}}(F)\cap\fo} \sum_{n_2\in N_{P_2}(F)} f(\Ad(n_2 x)^{-1}(\xi)). $$

  Now we decompose the integral over $x\in P_1(F)\bs H(\BA)^1$ into double integrals $n_1\in N_{P_1}(F)\bs N_{P_1}(\BA)$ and $y\in M_{P_1}(F)N_{P_1}(\BA)\bs H(\BA)^1$ and use the fact that $\chi_{P_1, P_2}^T(x)$ is left invariant under $N_{P_1}(\BA)$. We have
  \[\begin{split}
    &\int_{P_1(F)\bs H(\BA)^1} \chi_{P_1, P_2}^T(x) |j_{P_1, P_2, \fo}(x)| dx \\
    =&\int_{M_{P_1}(F)N_{P_1}(\BA)\bs H(\BA)^1} \int_{N_{P_1}(F)\bs N_{P_1}(\BA)} \chi_{P_1, P_2}^T(n_1 y) \left|\sum_{\xi\in\wt{\fm}_{\wt{P_1}}^{\wt{P_2}}(F)\cap\fo} \sum_{n_2\in N_{P_2}(F)} f(\Ad(n_2 n_1 y)^{-1}(\xi))\right| dn_1 dy \\
    \leq&\int_{M_{P_1}(F)N_{P_1}(\BA)\bs H(\BA)^1} \chi_{P_1, P_2}^T(y) \sum_{\xi\in\wt{\fm}_{\wt{P_1}}^{\wt{P_2}}(F)\cap\fo} \left(\int_{N_{P_1}(F)\bs N_{P_1}(\BA)} \sum_{n_2\in N_{P_2}(F)} |f(\Ad(n_2 n_1 y)^{-1}(\xi))| dn_1\right) dy.
  \end{split}\]
  Since $P_1\subseteq P_2$ and $\vol(N_{P_2}(F)\bs N_{P_2}(\BA))=1$, we see that
  \[\begin{split}
    &\int_{N_{P_1}(F)\bs N_{P_1}(\BA)} \sum_{n_2\in N_{P_2}(F)} |f(\Ad(n_2 n_1 y)^{-1}(\xi))| dn_1 \\
    =&\int_{N_{P_1}(F)\bs N_{P_1}(\BA)} \int_{N_{P_2}(F)\bs N_{P_2}(\BA)} \sum_{n_2\in N_{P_2}(F)} |f(\Ad(n_2 n n_1 y)^{-1}(\xi))| dn dn_1 \\
    =&\int_{N_{P_1}(F)\bs N_{P_1}(\BA)} \int_{N_{P_2}(\BA)} |f(\Ad(n n_1 y)^{-1}(\xi))| dn dn_1 \\
    =&\int_{N_{P_1}(F)\bs N_{P_1}(\BA)} \int_{(\fn_{\wt{P_2}}\cap\fs)(\BA)} |f(\Ad(n_1 y)^{-1}(\xi+U))| dU dn_1,
  \end{split}\]
  where we have applied Lemma \ref{regsslemma2} in the last equality.
  Hence
  \[\begin{split}
    &\int_{P_1(F)\bs H(\BA)^1} \chi_{P_1, P_2}^T(x) |j_{P_1, P_2, \fo}(x)| dx \\
    \leq&\int_{M_{P_1}(F)N_{P_1}(\BA)\bs H(\BA)^1} \chi_{P_1, P_2}^T(y) \sum_{\xi\in\wt{\fm}_{\wt{P_1}}^{\wt{P_2}}(F)\cap\fo} \left(\int_{N_{P_1}(F)\bs N_{P_1}(\BA)} \int_{(\fn_{\wt{P_2}}\cap\fs)(\BA)} |f(\Ad(n_1 y)^{-1}(\xi+U))| dU dn_1\right) dy \\
    =&\int_{P_1(F)\bs H(\BA)^1} \chi_{P_1, P_2}^T(x) \sum_{\xi\in\wt{\fm}_{\wt{P_1}}^{\wt{P_2}}(F)\cap\fo} \int_{(\fn_{\wt{P_2}}\cap\fs)(\BA)} |f(\Ad(x^{-1})(\xi+U))| dUdx,
  \end{split}\]
  whose convergence comes from that of the formula \eqref{equation2.1} when $R=P_2$.

  Next we begin to prove the second statement. From the first statement, now we are authorized to write
  $$ \int_{H(F)\bs H(\BA)^1} j_{f,\fo}^T(x) dx=\sum_{\{P_1,P_2:P_0\subseteq P_1\subseteq P_2\}} \int_{P_1(F)\bs H(\BA)^1} \chi_{P_1, P_2}^T(x) j_{P_1,P_2,\fo}(x) dx, $$
  where
  \[\begin{split}
    j_{P_1,P_2,\fo}(x)&=\sum_{\{P:P_1\subseteq P\subseteq P_2\}} (-1)^{\dim(A_P/A_{H})} j_{f,P,\fo}(x) \\
    &=\sum_{\{P:P_1\subseteq P\subseteq P_2\}} (-1)^{\dim(A_P/A_{H})} \left(\sum_{Y\in\fm_{\wt{P}}(F)\cap\fo} \sum_{n\in N_P(F)} f(\Ad(nx)^{-1}(Y))\right).
  \end{split}\]
  Decompose the integral over $x\in P_1(F)\bs H(\BA)^1$ into double integrals over $n_1\in N_{P_1}(F)\bs N_{P_1}(\BA)$ and $y\in M_{P_1}(F)N_{P_1}(\BA)\bs H(\BA)^1$. Since $N_{P_1}(F)\bs N_{P_1}(\BA)$ is compact, from Lemma \ref{regsslemma2} and \cite[\S41]{MR0165033}, we know that
  $$ \sum_{Y\in\fm_{\wt{P}}(F)\cap\fo} \sum_{n\in N_P(F)} |f(\Ad(n n_1 y)^{-1}(Y))|=\sum_{Y\in\fm_{\wt{P}}(F)\cap\fo} \sum_{u\in(\fn_{\wt{P}}\cap\fs)(F)} |f(\Ad(n_1 y)^{-1}(Y+u))| $$
  is bounded independently of $n_1\in N_{P_1}(F)\bs N_{P_1}(\BA)$. Then using the fact that $\chi_{P_1, P_2}^T(x)$ is left invariant under $N_{P_1}(\BA)$, we have
  \[\begin{split}
    \int_{H(F)\bs H(\BA)^1} j_{f,\fo}^T(x) dx=&\sum_{\{P_1,P_2:P_0\subseteq P_1\subseteq P_2\}} \int_{M_{P_1}(F)N_{P_1}(\BA)\bs H(\BA)^1} \chi_{P_1, P_2}^T(y) \sum_{\{P:P_1\subseteq P\subseteq P_2\}} (-1)^{\dim(A_P/A_{H})} \\
    &\sum_{Y\in\fm_{\wt{P}}(F)\cap\fo} \left(\int_{N_{P_1}(F)\bs N_{P_1}(\BA)} \sum_{n\in N_P(F)} f(\Ad(nn_1 y)^{-1}(Y)) dn_1\right) dy.
  \end{split}\]
  Since $P_1\subseteq P$ and $\vol(N_{P}(F)\bs N_{P}(\BA))=1$, we see that
  \[\begin{split}
    &\int_{N_{P_1}(F)\bs N_{P_1}(\BA)} \sum_{n\in N_{P}(F)} f(\Ad(n n_1 y)^{-1}(Y)) dn_1 \\
    =&\int_{N_{P_1}(F)\bs N_{P_1}(\BA)} \int_{N_{P}(F)\bs N_{P}(\BA)} \sum_{n\in N_{P}(F)} f(\Ad(n n_2 n_1 y)^{-1}(Y)) dn_2 dn_1 \\
    =&\int_{N_{P_1}(F)\bs N_{P_1}(\BA)} \int_{N_{P}(\BA)} f(\Ad(n n_1 y)^{-1}(Y)) dn dn_1 \\
    =&\int_{N_{P_1}(F)\bs N_{P_1}(\BA)} \int_{(\fn_{\wt{P}}\cap\fs)(\BA)} f(\Ad(n_1 y)^{-1}(Y+U)) dU dn_1,
  \end{split}\]
  where we have applied Lemma \ref{regsslemma2} in the last equality.
  Therefore, we have
  \[\begin{split}
    \int_{H(F)\bs H(\BA)^1} j_{f,\fo}^T(x) dx=&\sum_{\{P_1,P_2:P_0\subseteq P_1\subseteq P_2\}} \int_{M_{P_1}(F)N_{P_1}(\BA)\bs H(\BA)^1} \chi_{P_1, P_2}^T(y) \sum_{\{P:P_1\subseteq P\subseteq P_2\}} (-1)^{\dim(A_P/A_{H})} \\
    &\sum_{Y\in\fm_{\wt{P}}(F)\cap\fo} \left(\int_{N_{P_1}(F)\bs N_{P_1}(\BA)} \int_{(\fn_{\wt{P}}\cap\fs)(\BA)} f(\Ad(n_1 y)^{-1}(Y+U)) dU dn_1\right) dy \\
    =&\sum_{\{P_1,P_2:P_0\subseteq P_1\subseteq P_2\}} \int_{P_1(F)\bs H(\BA)^1} \chi_{P_1, P_2}^T(x) \sum_{\{P:P_1\subseteq P\subseteq P_2\}} (-1)^{\dim(A_P/A_{H})} \\
    &\cdot \left(\sum_{Y\in\fm_{\wt{P}}(F)\cap\fo} \int_{(\fn_{\wt{P}}\cap\fs)(\BA)} f(\Ad(x^{-1})(Y+U)) dU\right) dx \\
    =&\sum_{\{P_1,P_2:P_0\subseteq P_1\subseteq P_2\}} \int_{P_1(F)\bs H'(\BA)^1} \chi_{P_1, P_2}^T(x) k_{P_1,P_2,\fo}(x) dx.
  \end{split}\]
  By the proof of Theorem \ref{convergence2}, we are able to write
  $$ J_\fo^{H,T}(f)=\int_{H(F)\bs H(\BA)^1} k_{f,\fo}^T(x) dx=\sum_{\{P_1,P_2:P_0\subseteq P_1\subseteq P_2\}} \int_{P_1(F)\bs H(\BA)^1} \chi_{P_1, P_2}^T(x) k_{P_1,P_2,\fo}(x) dx, $$
  which completes our proof.
\end{proof}


\section{\textbf{Weighted orbital integrals}}\label{wtorbint}

Let $\fo\in\CO_{rs}$ (see Section \ref{inv2}). There is an element $Y_1\in\fo$ and a standard parabolic subgroup $P_1$ of $H$ such that $Y_1\in\fm_{\wt{P_1}}(F)$ but $Y_1$ cannot be $M_{P_1}(F)$-conjugate to an element in $\wt{\fr}$ (or equivalently in $\fm_{\wt{R}}$ by Proposition \ref{normmap}) for any standard parabolic subgroup $R\subsetneq P_1$. We call such $Y_1$ an elliptic element in $(\fm_{\wt{P_1}}\cap\fs)(F)$. For $P_1$ and $P_2$ a pair of standard parabolic subgroups of $H$, denote by $\Omega^{H}(\fa_{P_1},\fa_{P_2})$ the (perhaps empty) set of distinct isomorphisms from $\fa_{P_1}$ to $\fa_{P_2}$ obtained by restriction of elements in $\Omega^{H}$. 

\begin{lem}\label{lem91}
Let $P$ be a standard parabolic subgroup of $H$. Let $Y\in (\fm_{\wt{P}}\cap\fs)(F)$ be a regular semi-simple element in $\fs$. Then $Y$ is an elliptic element in $(\fm_{\wt{P}}\cap\fs)(F)$ if and only if the maximal $F$-split torus in the torus $H_Y$ is $A_P$. 
\end{lem}

\begin{proof}
It is evident that $H_Y\subseteq H_{Y^2}$. From $Y\in (\fm_{\wt{P}}\cap\fs)(F)$, one knows that $A_P\subseteq H_Y$. Since $Y$ is regular semi-simple in $\fs$, by \cite[Propositions 2.1 and 2.2]{MR1394521}, one deduces that $H_Y$ is a torus and that $Y^2\in\fm_P(F)$ is regular semi-simple in $\fh(F)$ (viewed as a central simple algebra over $E$). Thus $H_{Y^2}\subseteq M_P$. 

On the one hand, suppose that $Y$ is an elliptic element in $(\fm_{\wt{P}}\cap\fs)(F)$. If the maximal $F$-split torus in the torus $H_Y$ is not $A_P$, then there exists an $F$-split torus $A_\ast$ such that $A_P\subsetneq A_\ast\subseteq H_Y$. For $A_\ast\subseteq H_Y\subseteq H_{Y^2}\subseteq M_P$, the centralizer $\Cent_{M_P}(A_\ast)$ of $A_\ast$ in $M_P$ is a Levi subgroup of $M_P$. There exists $m\in M_P(F)$ such that $\Ad(m)(\Cent_{M_P}(A_\ast))=M_{R_\ast}$ for some standard parabolic subgroup $R_\ast$ of $M_P$. Then $\Ad(m)(A_\ast)\subseteq A_{R_\ast}$ and the centralizer of $\Ad(m)(A_\ast)$ in $M_P$ is $M_{R_\ast}$. Let $R$ be the unique standard parabolic subgroup of $H$ such that $R\subseteq P$ and that $R\cap M_P=R_\ast$. Then $A_R=A_{R_\ast}$ and $M_R=M_{R_\ast}$. Since $\tau$ (see Section \ref{explicitdescription}) commutes with $A_0$, by Proposition \ref{equalwt2.2}, the subspace of $\fm_{\wt{P}}\cap\fs$ formed by fixed points under conjugation by $\Ad(m)(A_\ast)$ is $\fm_{\wt{R}}\cap\fs$.  
From $A_\ast\subseteq H_Y$, one obtains $\Ad(m)(A_\ast)\subseteq H_{\Ad(m)(Y)}$. Since $\Ad(m)(Y)\in(\fm_{\wt{P}}\cap\fs)(F)$, we deduce that $\Ad(m)(Y)\in(\fm_{\wt{R}}\cap\fs)(F)$. Because $A_P\subsetneq A_\ast$ and $\Ad(m)(A_\ast)\subseteq A_R$, we have $A_P\subsetneq A_R$ and thus $R\subsetneq P$. That is to say, $Y$ is not an elliptic element in $(\fm_{\wt{P}}\cap\fs)(F)$. It is a contradiction. This proves one direction. 

On the other hand, suppose that the maximal $F$-split torus in the torus $H_Y$ is $A_P$. If $Y$ is not an elliptic element in $(\fm_{\wt{P}}\cap\fs)(F)$, there exists $m\in M_P(F)$ such that $\Ad(m)(Y)\in (\fm_{\wt{R}}\cap\fs)(F)$ for some standard parabolic subgroup $R\subsetneq P$. Then $A_R\subseteq H_{\Ad(m)(Y)}$, i.e., $\Ad(m^{-1})(A_R)\subseteq H_Y$. For $R\subsetneq P$, one sees that $A_P\subsetneq \Ad(m^{-1})(A_R)$. That is to say, $\Ad(m^{-1})(A_R)$ is a strictly larger split torus than $A_P$ in $H_Y$. It contradicts our hypothesis. This proves the other direction. 
\end{proof}

\begin{thm}\label{woi2}
  Let $\fo\in\CO_{rs}$, $P_1$ be a standard parabolic subgroup of $H$ and $Y_1\in\fo$ be an elliptic element in $(\fm_{\wt{P_1}}\cap\fs)(F)$. For $f\in\CS(\fs(\BA))$, we have
$$ J_\fo^{H}(f)=\vol(A_{P_1}^\infty H_{Y_1}(F)\bs H_{Y_1}(\BA))\cdot \int_{H_{Y_1}(\BA)\bs H(\BA)} f(\Ad(x^{-1})(Y_1)) v_{P_1}(x) dx, $$
where $v_{P_1}(x)$ is left-invariant under $H_{Y_1}(\BA)$ and is equal to the volume of the projection onto $\fa_{P_1}^H$ of the convex hull of $\{-H_Q(x)\}$, where $Q$ runs over all semi-standard parabolic subgroups of $H$ with $M_Q=M_{P_1}$. 
\end{thm}

\begin{remark}\label{rmkwoi2}
The weights that we obtain for regular semi-simple orbits are the same as Arthur's in \cite[p. 951]{MR518111}. These weights are also the same as those (see \cite[p. 131]{MR3026269}) appearing in the twisted trace formula for $H\rtimes\sigma$, where $\sigma$ acts on $H$ by $\Ad(\tau)$ (see Section \ref{explicitdescription} for the choice of $\tau$). Notice that the action of $\sigma$ stabilizes $P_0$ and $M_0$. All standard parabolic subgroups $P$ of $H$ are $\sigma$-stable and $\sigma$ fixes $\fa_P=\fa_{\wt{P}}$. 
\end{remark}

\begin{proof}[Proof of Theorem \ref{woi2}]
  Let $P$ be any standard parabolic subgroup of $H$ and $Y\in \fm_{\wt{P}}(F)\cap\fo$. There exists a standard parabolic subgroup $P_2\subseteq P$ and $Y_2$ an elliptic element in $(\fm_{\wt{P_2}}\cap\fs)(F)$ such that $Y_2$ is $M_{P}(F)$-conjugate to $Y$. By Lemma \ref{lem91}, the maximal $F$-split torus in $H_{Y_2}$ is $A_{P_2}$. Any element in $H(F)$ which conjugates $Y_1$ and $Y_2$ will conjugate $A_{P_1}$ and $A_{P_2}$. It follows that there exists $s\in\Omega^{H}(\fa_{P_1},\fa_{P_2})$ and $m\in M_{P}(F)$ such that
  $$ Y=\Ad(m\omega_{s}) (Y_1). $$
  Suppose that $P_3\subseteq P$ is another standard parabolic subgroup, $s'\in\Omega^{H}(\fa_{P_1},\fa_{P_3})$ and $m'\in M_{P}(F)$ such that
  $$ Y=\Ad(m'\omega_{s'}) (Y_1). $$
  Then there is $\zeta\in H_Y(F)$ such that
  $$ m'\omega_{s'}=\zeta m\omega_{s}. $$
  From $H_Y\subseteq M_{P}$, we see that
  $$ \omega_{s'}=\xi\omega_{s} $$
  for some $\xi\in M_{P}(F)$. Denote by $\Omega^{H}(\fa_{P_1};P)$ the set of $s\in \bigcup\limits_{\fa_{P_2}} \Omega^{H}(\fa_{P_1},\fa_{P_2})$ satisfying $\fa_{P}\subseteq s\fa_{P_1}$ and $s^{-1}\alpha>0$ for each $\alpha\in\Delta_{P_2}^{P}$. In sum, for any given $P$ a standard parabolic subgroup of $H$ and $Y\in \fm_{\wt{P}}(F)\cap\fo$, there is a unique $s\in \Omega^{H}(\fa_{P_1};P)$ such that $Y=\Ad(m\omega_{s}) (Y_1)$ for some $m\in M_{P}(F)$.

  For $x\in P(F)\bs H(\BA)$, we have
  \[\begin{split}
    j_{f,P,\fo}(x)&=\sum_{Y\in\fm_{\wt{P}}(F)\cap\fo} \sum_{n\in N_P(F)} f(\Ad(nx)^{-1}(Y)) \\
    &=\sum_{s\in\Omega^{H}(\fa_{P_1};P)} \sum_{m\in M_{P,\Ad(\omega_{s})(Y_1)}(F)\big\backslash M_{P}(F)} \sum_{n\in N_P(F)} f(\Ad(m nx)^{-1}\Ad(\omega_s) (Y_1)) \\
    &=\sum_{s\in\Omega^{H}(\fa_{P_1};P)} \sum_{m\in M_{P,\Ad(\omega_{s})(Y_1)}(F)\big\backslash P(F)} f(\Ad(m x)^{-1}\Ad(\omega_s) (Y_1)),
  \end{split}\]
  where $M_{P,\Ad(\omega_{s})(Y_1)}$ denotes the centralizer of $\Ad(\omega_{s})(Y_1)$ in $M_P$. Then for $T\in\fa_{0}$ and $x\in H(F)\bs H(\BA)$, we deduce that
  \[\begin{split}
    j_{f,\fo}^T(x)=&\sum_{\{P: P_0\subseteq P\}} (-1)^{\dim(A_P/A_{H})} \sum_{\delta\in P(F)\bs H(F)} \wh{\tau}_P^H(H_{P}(\delta x)-T)\cdot j_{f,P,\fo}(\delta x) \\
    =&\sum_{\{P: P_0\subseteq P\}} (-1)^{\dim(A_P/A_{H})} \sum_{\delta\in P(F)\bs H(F)} \wh{\tau}_P^H(H_{P}(\delta x)-T) \\
    &\cdot \left(\sum_{s\in\Omega^{H}(\fa_{P_1};P)} \sum_{m\in M_{P,\Ad(\omega_{s})(Y_1)}(F)\big\backslash P(F)} f(\Ad(m \delta x)^{-1}\Ad(\omega_s) (Y_1))\right) \\
    =&\sum_{\{P: P_0\subseteq P\}} (-1)^{\dim(A_P/A_{H})} \sum_{s\in\Omega^{H}(\fa_{P_1};P)} \sum_{\delta\in M_{P,\Ad(\omega_{s})(Y_1)}(F)\big\backslash H(F)} \wh{\tau}_P^H(H_{P}(\delta x)-T) \\
    &\cdot f(\Ad(\delta x)^{-1}\Ad(\omega_s) (Y_1)).
  \end{split}\]
  Notice that the centralizer of $\Ad(\omega_s)(Y_1)$ in $H$ is actually contained in $M_{P}$, so
  \[\begin{split}
    j_{f,\fo}^T(x)=&\sum_{\{P: P_0\subseteq P\}} (-1)^{\dim(A_P/A_{H})} \sum_{s\in\Omega^{H}(\fa_{P_1};P)} \sum_{\delta\in H_{\omega_{s}Y_1\omega_{s}^{-1}}(F)\big\backslash H(F)} \wh{\tau}_P^H(H_{P}(\delta x)-T) \cdot f(\Ad(\delta x)^{-1}\Ad(\omega_s) (Y_1)) \\
    =&\sum_{\{P: P_0\subseteq P\}} (-1)^{\dim(A_P/A_{H})} \sum_{s\in\Omega^{H}(\fa_{P_1};P)} \sum_{\delta\in H_{Y_1}(F)\backslash H(F)} \wh{\tau}_P^H(H_{P}(\omega_s \delta x)-T) \cdot f(\Ad(\delta x)^{-1}(Y_1)).
  \end{split}\]
  For $y\in H(\BA)$, we write
  $$ \chi_T(y):=\sum_{\{P: P_0\subseteq P\}} (-1)^{\dim(A_P/A_{H})} \sum_{s\in\Omega^{H}(\fa_{P_1};P)} \wh{\tau}_P^H(H_{P}(\omega_s y)-T). $$
  Then
  $$ j_{f,\fo}^T(x)=\sum_{\delta\in H_{Y_1}(F)\backslash H(F)} f(\Ad(\delta x)^{-1}(Y_1))\cdot \chi_T(\delta x). $$

  For $T\in T_+ +\fa_{P_0}^+$, applying Theorem \ref{secondkernel2} and the fact that $j_{f,\fo}^T(x)$ is left invariant by $A_{H}^\infty$, we have
  \[\begin{split}
    J_\fo^{H,T}(f)&=\int_{H(F)\bs H(\BA)^1} j_{f,\fo}^T(x) dx \\
    &=\int_{A_{H}^\infty H(F)\bs H(\BA)} \left(\sum_{\delta\in H_{Y_1}(F)\backslash H(F)} f(\Ad(\delta x)^{-1}(Y_1))\cdot \chi_T(\delta x)\right) dx. 
  \end{split}\]
Then we obtain
\begin{equation}\label{lastequ2}
  J_\fo^{H,T}(f)=\vol(A_{P_1}^\infty H_{Y_1}(F)\bs H_{Y_1}(\BA))\cdot \int_{H_{Y_1}(\BA)\bs H(\BA)} f(\Ad(x^{-1})(Y_1)) v_{P_1}(x,T) dx, 
\end{equation}
  where
  $$ v_{P_1}(x,T):=\int_{A_{H}^\infty\bs A_{P_1}^\infty} \chi_T(ax) da. $$
  Here we have used the fact that $v_{P_1}(x,T)$ is well-defined and left-invariant under $H_{Y_1}(\BA)\subseteq M_P(\BA)$. Moreover, $v_{P_1}(x,T)$ is equal to the volume of the projection onto $\fa_{P_1}^H$ of the convex hull of $\{T_Q-H_Q(x)\}$, where $T_Q$ denotes the projection of $sT$ in $\fa_Q$ for any $s\in\Omega^H$ satisfying $sP_0\subseteq Q$, and $Q$ takes over all semi-standard parabolic subgroups of $H$ with $M_Q=M_{P_1}$. These properties follow from \cite[p. 951]{MR518111}. We have also assumed the finiteness of $\vol(A_{P_1}^\infty H_{Y_1}(F)\bs H_{Y_1}(\BA))$, which results from Lemma \ref{lem91}. 

In the end, we may conclude by taking contant terms of both sides of (\ref{lastequ2}). 
\end{proof}


\section{\textbf{The weighted fundamental lemma}}

In this section, we turn to the local setting and change the notation by letting $F$ be a non-archimedean local field of characteristic $0$. Denote by $|\cdot|_F$ the normalized absolute value on $F$ and by $\val_F(\cdot)$ the normalized valuation on $F$. 

\subsection{\textbf{$(H,M)$-families associated to local weighted orbital integrals}}

Suppose that $H$ is a reductive group defined over $F$. Fix a maximal compact subgroup $K$ of $H(F)$ which is admissible relative to $M_0$ in the sense of \cite[p. 9]{MR625344}. 
In this paper, when $H(F)=GL_n(D)$, where $D$ is a central division algebra over a finite field extension $E$ of $F$, we choose $K:=GL_n(\CO_D)$, where $\CO_D$ is the ring of integers of $D$. 
For a parabolic subgroup $P$ of $H$ and $x\in H(F)$, we have $H(F)=P(F)K$ by Iwasawa decomposition and define $H_P(x)$ as in Section \ref{hcmap} by replacing $|\cdot|_\BA$ with $|\cdot|_F$. 
Suppose that $M$ is a Levi subgroup of $H$ containing $M_0$. Let $\CP(M)$ be the set of parabolic subgroups of $H$ with Levi component $M$. According to \cite[p. 40-41]{MR625344}, 
$$ v_P(\lambda,x):=e^{-\lambda(H_P(x))}, \forall \lambda\in i\fa_M^\ast, P\in\CP(M), $$
is an $(H,M)$-family in the sense of \cite[p. 36]{MR625344}. Let $P\in\CP(M)$ and $Q$ be a parabolic subgroup of $H$ containing $P$. Define
$$ \theta_P^Q(\lambda):=\vol(\fa_P^Q/\BZ(\Delta_P^Q)^\vee)^{-1}\prod_{\alpha^\vee\in(\Delta_P^Q)^\vee} \lambda(\alpha^\vee), $$
where $\BZ(\Delta_P^Q)^\vee$ denotes the lattice in $\fa_P^Q$ generated by $(\Delta_P^Q)^\vee$. Then we obtain a function
$$ v_M^Q(x):=\lim_{\lambda\ra 0} \sum_{\{P\in\CP(M): P\subseteq Q\}} v_P(\lambda,x) \theta_P^Q(\lambda)^{-1}, \forall x\in H(F). $$

\subsection{\textbf{Matching of orbits}}

Assume that $F$ has odd residue characteristic and that $E$ is an unramified quadratic extension over $F$. Denote by $\CO_F$ (resp. $\CO_E$) the ring of integers of $F$ (resp. $E$). Since $E/F$ is an unramified quadratic extension, there exists $\alpha\in\CO_E^\times$ such that $\alpha^2\in F$ and that $E=F(\alpha)$. Let $\alpha$ be such an element. Because $F$ has odd residue characteristic, we deduce that $\CO_E=\CO_F\oplus\CO_F\alpha$. 

Let $G:=GL_{2n}$ and $H:=\Res_{E/F}GL_{n,E}$ be the centralizer of $E^\times$ in $G$. Both of them are regarded as group schemes over $\CO_F$. More precisely, the $\CO_F$-basis $\{1, \alpha\}$ of $\CO_E$ induces an inclusion $\CO_E^\times\subseteq GL_2(\CO_F)$ or more generally $GL_n(\CO_E)\subseteq GL_{2n}(\CO_F)$. Define $S:=\{g\in G: \Ad(\alpha)(g^{-1})=g\}$, which is a scheme over $\CO_F$. Let $\fs$ be the tangent space at the neutral element of $S$, which is viewed as a subspace of $\fg$. Let $\tau$ be the $\CO_F$-linear automorphism on $\CO_E$ given by the conjugation $\tau(\alpha)=-\alpha$. We can show that $\fg\fl_2(\CO_F)=\CO_E\oplus\CO_E\tau$, where $\CO_E$ (resp. $\CO_E\tau$) is the eigenspace of eigenvalue $+1$ (resp. $-1$) under the action $\Ad(\alpha)$. Then $\fs=\fh\tau$, where $\tau$ is viewed as an element of $\fg(\CO_F)$ by the diagonal embedding. In the rest of this paper, we shall simply set $\fs:=\fh$ equipped with the twisted conjugation of $H$, i.e., $x\cdot Y=xY\ov{x}^{-1}$, where $\ov{x}:=\Ad(\tau)(x)$ is the nontrivial Galois conjugate of $x\in H$. 

Let $H':=GL_{n}\times GL_{n}$ be the subgroup of $G$ by diagonal embedding. It is regarded as a group scheme over $\CO_F$. Define $S':=\{g\in G: \Ad(\epsilon)(g^{-1})=g\}$ with $\epsilon:=\mat(1_n,,,-1_n)$, which is a scheme over $\CO_F$. Let $\fs'$ be the tangent space at the neutral element of $S'$, which is viewed as a subspace of $\fg$. Then $\fs'=\bigg\{
\left( \begin{array}{cc}
0 & A \\
B & 0 \\
\end{array} \right): A,B\in \fg\fl_n\bigg\} \simeq \fg\fl_n\oplus \fg\fl_n$ on which $H'$ acts by conjugation, i.e., $(x_{1},x_{2})\cdot(A,B)=(x_{1}Ax_{2}^{-1},x_{2}Bx_{1}^{-1})$. 

Recall \cite[Lemma 1.1 of Chapter 1]{MR1007299} that the norm map $Y\mapsto Y\ov{Y}$ induces an injection from the set of twisted conjugacy classes in $GL_n(E)$ to the set of conjugacy classes in $GL_n(F)$, whose image is denoted by $N(GL_n(E))$; in particular, we write $NE^\times$ for $N(GL_1(E))$. As before, we have the notion of regular semi-simple elements in $\fs$ (resp. $\fs'$) with respect to the action of $H$ (resp. $H'$) whose set is denoted by $\fs_{rs}$ (resp. $\fs'_{rs}$). Moreover, we have the following explicit description. 

\begin{prop}\label{matorb}
1) An element $Y$ of $\fs(F)$ is regular semi-simple if and only if $Y\ov{Y}$ belongs to $GL_n(E)$ and is regular semi-simple. The map $Y\mapsto Y\ov{Y}$ from $\fs(F)$ to $GL_n(E)$ induces an injection from the set of twisted $H(F)$-conjugacy classes of regular semi-simple elements in $\fs(F)$ into the set of regular semi-simple conjugacy classes in $GL_n(F)$.

2) An element $X$ of $\fs'(F)$ is regular semi-simple if and only if it is $H'(F)$-conjugate to an element of the form
$$ X(A):=
\left( \begin{array}{cc}
0 & 1_n \\
A & 0 \\
\end{array} \right) $$
with $A\in GL_{n}(F)$ being regular semi-simple. The map $\mat(0,A,B,0)\mapsto AB$ from $\fs'(F)$ to $GL_n(F)$ induces a bijection between the set of $H'(F)$-conjugacy classes of regular semi-simple elements in $\fs'(F)$ and the set of regular semi-simple conjugacy classes in $GL_{n}(F)$. 
\end{prop}

\begin{proof}
1) is contained in \cite[Lemma 2.1]{MR1487565}, while 2) is proved in \cite[Proposition 2.1 and Lemma 2.1]{MR1394521}. 
\end{proof}

To sum up, the composition of the map in 1) and the inverse of the map in 2) above induces an injection from the set of $H(F)$-orbits in $\fs_{rs}(F)$ into the set of $H'(F)$-orbits in $\fs'_{rs}(F)$. We shall say that $Y\in\fs_{rs}(F)$ and $X\in\fs'_{rs}(F)$ have matching orbits if their orbits are matched under this injection. Alternatively, this can be canonically characterized by an identification of categorical quotients $\fs//H\simeq\fs'//H'$ (see Proposition \ref{propcatquot2} and \cite[Proposition 3.3]{2019arXiv190407102L}). With our identification $\fs=\fh$, we see that $Y\in\fs_{rs}(F)$ and $X=\mat(0,A,B,0)\in\fs'_{rs}(F)$ have matching orbits if and only if the characteristic polynomial of $Y\ov{Y}\in GL_n(E)$ equals that of $AB\in GL_n(F)$. 

\subsection{\textbf{Matching of Levi subgroups involved}}

We recall some terminology in \cite[\S3.4 and \S5.2]{2019arXiv190407102L}. The subgroup of diagonal matrices in $G$ is a common minimal Levi subgroup of $G$ and $H'$. We also fix a minimal semi-standard parabolic subgroup of $H'$ to be the group of products of upper triangular matrices. We say that a semi-standard parabolic subgroup of $G$ is ``relatively standard'' if its intersection with $H'$ is a standard parabolic subgroup of $H'$. Let $\omega:=\mat(0,1_n,1_n,0)$. We say that a semi-standard parabolic subgroup $P$ of $G$ is ``$\omega$-stable'' if $\omega\in P$. 
Recall that if the Lie algebra of a relatively standard parabolic subgroup $P$ of $G$ has non-empty intersection with $\fs'_{rs}$, then $P$ must be $\omega$-stable (see \cite[Proposition 5.4]{2019arXiv190407102L}). 

We shall say that a semi-standard Levi subgroup $M'$ of $G$ is ``$\omega$-stable'' if $\omega\in M'$. Notice that $M'$ is $\omega$-stable if and only if $M'=M_{P'}$ for some $\omega$-stable parabolic subgroup $P'$ of $G$. 
Here $\omega$-stable Levi subgroups of $G$ play the role of semi-standard Levi subsets of $(GL_n\times GL_n)\rtimes \sigma'$ in the sense of \cite[\S I.1]{MR1339717}, where $\sigma'$ exchanges two copies of $GL_n$. For any linear subspace $\fv$ of $\fg$, we denote by $\fv^\times$ the intersection of $\fv$ and $G$ in $\fg$. Notice that there is a bijection between the set of semi-standard Levi subgroups of $GL_{n}$ and the set of semi-standard Levi subgroups of $H$ (resp. the set of $\omega$-stable Levi subgroups of $G$) induced by $M_n\mapsto M=\Res_{E/F} M_{n,E}$ (resp. $M_n\mapsto M'=\mat(\fm_n,\fm_n,\fm_n,\fm_n)^\times$); here $\fm_n$ denotes the Lie algebra of $M_n$. We shall use the notations $M_n, M, M'$ to denote the corresponding semi-standard or $\omega$-stable Levi subgroups of different groups under these bijections after fixing one of the three. We also have bijections among semi-standard or $\omega$-stable parabolic subgroups (denoted by $Q_n, Q, Q'$) of different groups containing these Levi subgroups. 

Let $M'$ be an $\omega$-stable Levi subgroup of $G$. We shall say that $Y\in\fm(F)\cap\fs_{rs}(F)$ and $X\in\fm'(F)\cap\fs'_{rs}(F)$ have $M'$-matching orbits if in each pair of blocks of $\fm$ and $\fm'$, their components have matching orbits. 

\subsection{\textbf{Transfer factor}}

Let $\eta$ be the quadratic character of $F^\times/{NE^\times}$ attached to the quadratic field extension $E/F$. For $X=\mat(0,A,B,0)\in\fs'_{rs}(F)$, define a transfer factor (see \cite[Definition 5.8]{MR3414387})
$$ \kappa(X):=\eta(\det(A)), $$
which satisfies $\kappa(\Ad(x^{-1})(X))=\eta(\det(x))\kappa(X)$ for any $x\in H'(F)$. 

\subsection{\textbf{Transfer of weighted orbital integrals}}

Fix the Haar measures on $H(F)$ and $H'(F)$ such that $\vol(H(\CO_F))=\vol(H'(\CO_F))=1$. For a locally compact and totally disconnected space $X$, denote by $\CC_c^\infty(X)$ the $\BC$-linear space of locally constant and compactly supported functions on $X$. 

\begin{defn}\label{deflocwoi}
1) Let $M$ be a semi-standard Levi subgroup of $H$ and $Q$ a parabolic subgroup of $H$ containing $M$. For $Y\in\fm(F)\cap\fs_{rs}(F)$ and $f\in\CC_c^\infty(\fs(F))$, we define the weighted orbital integral of $f$ at $Y$ by
$$ J_M^Q(Y, f):=\int_{H_Y(F)\bs H(F)} f(\Ad(x^{-1})(Y))v_M^Q(x) dx. $$

2) Let $M'$ be an $\omega$-stable Levi subgroup of $G$ and $Q'$ a parabolic subgroup of $G$ containing $M'$ (thus $Q'$ is $\omega$-stable). For $X\in\fm'(F)\cap\fs'_{rs}(F)$ and $f'\in\CC_c^\infty(\fs'(F))$, we define the weighted $\eta$-orbital integral of $f'$ at $X$ by
$$ J_{M'}^{\eta,Q'}(X, f'):=\int_{H'_X(F)\bs H'(F)} f'(\Ad(x^{-1})(X))\eta(\det(x))v_{M'}^{Q'}(x) dx. $$
\end{defn}

\begin{remark}\label{twwt}
1) $v_M^Q$ is a local analogue of the weight that we got in Theorem \ref{woi2}. By Remark \ref{rmkwoi2}, it is the same as $v_{(\Res_{E/F} M_{n,E})\rtimes\sigma}^{(\Res_{E/F} Q_{n,E})\rtimes\sigma}$ in \cite[\S I.3]{MR1339717}, where $\sigma$ is the nontrivial Galois conjugation. 

2) $v_{M'}^{Q'}$ is a local analogue of the weight that we got in \cite[Theorem 9.2]{2019arXiv190407102L}. By \cite[Remark 9.3]{2019arXiv190407102L}, it is the same as $v_{(M_n\times M_n)\rtimes\sigma'}^{(Q_n\times Q_n)\rtimes\sigma'}$ in \cite[\S I.3]{MR1339717}, where $\sigma'$ exchanges two copies. 
\end{remark}

If $Y\in\fs_{rs}(F)$ and $X\in\fs'_{rs}(F)$ have matching orbits, their centralizers $H_Y$ and $H'_X$ are isomorphic because they are tori and inner forms of each other (see \cite[Remark 5.6]{MR3414387}). We shall fix compatible Haar measures on them. 

\begin{defn}\label{defstrass}
For $f\in\CC_c^\infty(\fs(F))$ and $f'\in\CC_c^\infty(\fs'(F))$, we say that $f$ and $f'$ are strongly associated if for all $\omega$-stable Levi subgroup $M'$ of $G$ and all parabolic subgroup $Q'$ of $G$ containing $M'$ (thus $Q'$ is $\omega$-stable), we have
\begin{enumerate}[\indent (1)]
\item if $Y\in\fm(F)\cap\fs_{rs}(F)$ and $X\in\fm'(F)\cap\fs'_{rs}(F)$ have $M'$-matching orbits, then
$$ \kappa(X) J_{M'}^{\eta,Q'}(X, f')=J_M^Q(Y, f); $$
\end{enumerate}
\begin{enumerate}[\indent (2)]
\item if $X=\mat(0,A,B,0)\in\fm'(F)\cap\fs'_{rs}(F)$ satisfies $\xi(AB)\notin NE^\times$ for some $\xi\in X(M_{Q_n})_F$, then
$$ J_{M'}^{\eta,Q'}(X, f')=0. $$
\end{enumerate}
\end{defn}

We remark that this definition is inspired by \cite[Definition III.3.2]{MR1339717} on the base change for $GL_n$. The following result (cf. \cite[Remark III.3.2.(i)]{MR1339717}) shows that to check the vanishing statement (2) in the above definition, it suffices to check it for all $\omega$-stable Levi subgroup $M'$ of $G$ such that $X$ is an elliptic element in $\fm'(F)\cap\fs'_{rs}(F)$ (i.e. $A_{M'}$ is the maximal $F$-split torus in $H'_X$). 

\begin{prop}\label{equivvancond}
Let $f'\in\CC_c^\infty(\fs'(F))$. The following two conditions are equivalent: 
\begin{enumerate}[\indent 1)]
\item for all $\omega$-stable Levi subgroup $M'$ of $G$ and all parabolic subgroup $Q'$ of $G$ containing $M'$, if $X=\mat(0,A,B,0)\in\fm'(F)\cap\fs'_{rs}(F)$ satisfies $\xi(AB)\notin NE^\times$ for some $\xi\in X(M_{Q_n})_F$, then
$$ J_{M'}^{\eta,Q'}(X, f')=0; $$
\end{enumerate}
\begin{enumerate}[\indent 2)]
\item for all $\omega$-stable Levi subgroup $M'$ of $G$ and all parabolic subgroup $Q'$ of $G$ containing $M'$, if $X=\mat(0,A,B,0)$ is an elliptic element in $\fm'(F)\cap\fs'_{rs}(F)$ and satisfies $\xi(AB)\notin NE^\times$ for some $\xi\in X(M_{Q_n})_F$, then
$$ J_{M'}^{\eta,Q'}(X, f')=0. $$
\end{enumerate}
\end{prop}

\begin{proof}
The direction 1)$\Rightarrow$2) is trivial. Now we assume 2) and prove 1). 

Let $X=\mat(0,A,B,0)\in\fm'(F)\cap\fs'_{rs}(F)$ satisfy $\xi(AB)\notin NE^\times$ for some $\xi\in X(M_{Q_n})_F$. There is an $\omega$-stable Levi subgroup $M'_\ast$ of $G$ contained in $M'$ and an element $y\in M'(F)\cap H'(F)$ such that $X_\ast:=\Ad(y)(X)$ is an elliptic element in $\fm'_\ast(F)\cap\fs'_{rs}(F)$. We have
\begin{equation}\label{infigj2des1}
 J_{M'}^{\eta,Q'}(X_\ast, f')=\eta(\det(y))J_{M'}^{\eta,Q'}(X, f'). 
\end{equation}
Suppose that $X_\ast=\mat(0,A_\ast,B_\ast,0)$. Then $\xi(A_\ast B_\ast)\notin NE^\times$ for the above $\xi\in X(M_{Q_n})_F$. 

By the descent formula for $(G,M')$-families (see \cite[Lemma I.1.2]{MR1339717}), we have
$$ v_{M'}^{Q'}=\sum_{L'\in\msl^{Q'}(M'_\ast)} d_{M'_\ast}^{Q'}(M',L') v_{M'_\ast}^{Q'_{L'}}, $$
where $\msl^{Q'}(M'_\ast)$ denotes the set of Levi subgroups of $G$ contained in $Q'$ and containing $M'_\ast$ (thus $L'$ is $\omega$-stable), $Q'_{L'}$ is some parabolic subgroup of $G$ with Levi factor $L'$ (thus $Q'_{L'}$ is $\omega$-stable), and $d_{M'_\ast}^{Q'}(M',L')\in\BR_{\geq 0}$ is defined in \cite[p. 356]{MR928262}. Thus
\begin{equation}\label{infigj2des2}
 J_{M'}^{\eta,Q'}(X_\ast, f')=\sum_{L'\in\msl^{Q'}(M'_\ast)} d_{M'_\ast}^{Q'}(M',L') J_{M'_\ast}^{\eta,Q'_{L'}}(X_\ast, f'). 
\end{equation}
For all $L'\in\msl^{Q'}(M'_\ast)$, let $\xi_L\in X(L_n)_F$ be the image of $\xi$ under the restriction $X(M_{Q_n})_F\hookrightarrow X(L_n)_F$. Then $\xi_{L}(A_\ast B_\ast)\notin NE^\times$. By our assumption 2), we have
$$ J_{M'_\ast}^{\eta,Q'_{L'}}(X_\ast, f')=0. $$
Then by (\ref{infigj2des1}) and (\ref{infigj2des2}), we obtain
$$ J_{M'}^{\eta,Q'}(X, f')=\eta(\det(y))^{-1} J_{M'}^{\eta,Q'}(X_\ast, f')=0, $$
which shows 1). 
\end{proof}

The proposition below (cf. \cite[Lemma III.3.3]{MR1339717}) shows that strongly associated functions are smooth transfers of each other in the sense of \cite[Definition 5.10.(ii)]{MR3414387}. 

\begin{prop}\label{infigj2strassimpliessmtransfer}
If $f'\in\CC_c^\infty(\fs'(F))$ satisfies the conditions in Proposition \ref{equivvancond}, then for $X\in\fs'_{rs}(F)$ with no matching orbit in $\fs_{rs}(F)$, we have
$$ J_{G}^{\eta,G}(X,f')=0. $$
\end{prop}

To prove this proposition, we recall two basic facts. 

\begin{lem}\label{infigj2basicfact1}
Suppose that $\sum\limits_{j=1}^{l} n_j=n$. Let $A=(A_1,...,A_l)\in GL_{n_1}(F)\times\cdot\cdot\cdot\times GL_{n_l}(F)$ be a regular semi-simple element in $GL_n(F)$. Then $A\in N(GL_n(E))$ if and only if $A_j\in N(GL_{n_j}(E))$ for all $1\leq j\leq l$. 
\end{lem}

\begin{proof}
This is known, but we include its proof here for completeness (cf. \cite[Lemma 8.8]{MR564478}). For $A\in N(GL_n(E))$, there exists $B\in GL_n(E)$ such that $A=B\ov{B}$. Since $A\in GL_n(F)$, we have $B\ov{B}=\ov{B}B$, which implies that $AB=BA$. But $A$ is regular semi-simple in $GL_n(E)$. 
Thus $B\in GL_{n_1}(E)\times\cdot\cdot\cdot\times GL_{n_l}(E)$. We write $B=(B_1,...,B_l)$ with $B_j\in GL_{n_j}(E)$ for all $1\leq j\leq l$. Then we obtain $A_j=B_j\ov{B_j}\in N(GL_{n_j}(E))$ for all $1\leq j\leq l$. This shows one direction. The other direction is trivial. 
\end{proof}

\begin{lem}\label{infigj2basicfact2}
Let $A\in GL_n(F)$ be an elliptic regular element. Then $A\in N(GL_n(E))$ if and only if $\det(A)\in NE^\times$. 
\end{lem}

\begin{proof}
This is a special case of \cite[Lemma 1.4 in Chapter 1]{MR1007299}. 
\end{proof}

\begin{proof}[Proof of Proposition \ref{infigj2strassimpliessmtransfer}]
Up to conjugation by $H(F)$, it suffices to consider $X=\mat(0,1_n,A,0)$ with $A$ an elliptic regular element in $M_n(F)$ for some semi-standard Levi subgroup $M_n$ of $GL_n$. Then by Lemmas \ref{infigj2basicfact1} and \ref{infigj2basicfact2}, $X$ has no matching orbit in $\fs_{rs}(F)$ if and only if $\xi(A)\notin NE^\times$ for some $\xi\in X(M_n)_F$. Let $Q'$ be a parabolic subgroup of $G$ with Levi factor $M'$. We have
$$ J_{G}^{\eta,G}(X,f')=J_{M'}^{\eta, Q'}(X,f'). $$
But $J_{M'}^{\eta, Q'}(X,f')$ vanishes for $f'\in\CC_c^\infty(\fs'(F))$ satisfying the conditions in Proposition \ref{equivvancond}. Then we finish the proof. 
\end{proof}

\subsection{\textbf{The weighted fundamental lemma}}

Let $f_0\in\CC_c^\infty(\fs(F))$ (resp. $f'_0\in\CC_c^\infty(\fs'(F))$) be the characteristic function of $\fs(\CO_F)=\fg\fl_n(\CO_E)$ (resp. of $\fs'(\CO_F)=\bigg\{
\left( \begin{array}{cc}
0 & A \\
B & 0 \\
\end{array} \right): A,B\in \fg\fl_n(\CO_F)\bigg\}$). 

\begin{thm}\label{wfl}
The functions $f_0$ and $f'_0$ are strongly associated. 
\end{thm}

The rest of the section is devoted to the proof of this theorem with the help of (split and unramified) base changes for $GL_n$. Suppose that $M'$ is an $\omega$-stable Levi subgroup of $G$ and that $Q'$ is a parabolic subgroup of $G$ containing $M'$ (thus $Q'$ is $\omega$-stable). For $x=(x_{i,j})\in\fg\fl_n(E)$, let $|x|:=\max_{i,j}|x_{i,j}\ov{x_{i,j}}|_F^{1/2}$. 

\subsubsection{Split base change}

Let $A\in GL_n(F)$ be regular semi-simple and denote $v:=\val_F(\det(A))$. 
We shall define $\Phi_v\in\CC_c^\infty(GL_n(F)\times GL_n(F))$ and $\Psi_v\in\CC_c^\infty(GL_n(F))$ as in the proof of \cite[Lemma 5.18]{MR3414387}. Let $\Phi_v$ be the characteristic function of the subset of $(x_1,x_2)\in GL_n(F)\times GL_n(F)$ satisfying $|x_1|\leq1$, $|x_2|\leq1$ and $\val_F(\det (x_1 x_2))=v$. 
Let $\Psi_v$ be the function on $GL_n(F)$ defined by
$$ \Psi_v(g):=\int_{GL_n(F)} \Phi_v(x_2,x_2^{-1}g)\eta(\det(x_2)) dx_2. $$
We also define $\Theta_v(x_1,x_2):=\Phi_v(x_1,x_2)\eta(\det(x_1))\in\CC_c^\infty(GL_n(F)\times GL_n(F))$. 

We shall consider the involution $\sigma'$ of $GL_n\times GL_n$ which exchanges two copies. Denote by $(GL_n\times GL_n)_{(1_n,A),\sigma'}$ the twisted (by $\sigma'$) centralizer in $GL_n\times GL_n$ of $(1_n,A)$ and by $GL_n(F)_A$ the centralizer in $GL_n(F)$ of $A$. 
Recall the (split) base change homomorphism (see \cite[\S5 of Chapter 1]{MR1007299} for example)
$$ bc_{F\times F/F}: \CH(GL_n(F)\times GL_n(F), GL_n(\CO_F)\times GL_n(\CO_F))\ra\CH(GL_n(F), GL_n(\CO_F)) $$
defined by the convolution product, where $\CH(GL_n(F)\times GL_n(F), GL_n(\CO_F)\times GL_n(\CO_F))$ and $\CH(GL_n(F), \newline GL_n(\CO_F))$ denote the corresponding spherical Hecke algebras. Notice that $\Phi_v, \Theta_v\in\CH(GL_n(F)\times GL_n(F), GL_n(\CO_F)\times GL_n(\CO_F))$ and that $\Psi_v\in \CH(GL_n(F), GL_n(\CO_F))$. 

\begin{lem}\label{bclemcas1}
We have
$$ \Psi_v=bc_{F\times F/F}(\Theta_v). $$
\end{lem}

\begin{proof}
This is evident by definition since $\Psi_v$ is obtained from $\Theta_v$ by the convolution product. 
\end{proof}

Suppose additionally that $A$ belongs to the Levi subgroup $M_n(F)$. 
The twisted (by $\sigma'$) weighted orbital integral of $\Theta_v\in\CC_c^\infty(GL_n(F)\times GL_n(F))$ at $(1_n,A)$ is defined by
$$ J_{(M_n\times M_n)\rtimes\sigma'}^{(Q_n\times Q_n)\rtimes\sigma'}((1_n,A),\Theta_v):=\int_{(GL_n\times GL_n)_{(1,A),\sigma'}(F)\bs GL_n(F)\times GL_n(F)} \Theta_v(x^{-1}(1_n,A)\sigma'(x)) v_{(M_n\times M_n)\rtimes\sigma'}^{(Q_n\times Q_n)\rtimes\sigma'}(x) dx. $$
The weighted orbital integral of $\Psi_v\in\CC_c^\infty(GL_n(F))$ at $A$ is defined by
$$ J_{M_n}^{Q_n}(A,\Psi_v):=\int_{GL_n(F)_A\bs GL_n(F)} \Psi_v(\Ad(x^{-1})(A)) v_{M_n}^{Q_n}(x) dx. $$

\begin{coro}\label{corbclemcas1}
For $A\in M_n(F)$ which is regular semi-simple in $GL_n(F)$, we have
$$ J_{(M_n\times M_n)\rtimes\sigma'}^{(Q_n\times Q_n)\rtimes\sigma'}((1_n,A),\Theta_v)=J_{M_n}^{Q_n}(A,\Psi_v). $$
\end{coro}

\begin{proof}
It results from Lemma \ref{bclemcas1} and \cite[Theorem IV.5.2]{MR1339717} for the (split) base change $F\times F/F$. 
\end{proof}

Let $X=\mat(0,1_n,A,0)\in\fm'(F)\cap\fs'_{rs}(F)$. Then $\kappa(X)=1$. 
By Remark \ref{twwt}.1), since $\eta(\det(x_1 x_2))=\eta(\det(x_1^{-1}x_2))$ for $(x_1,x_2)\in GL_n(F)\times GL_n(F)$, we see that
\begin{equation}\label{equationwoi1}
\kappa(X)J_{M'}^{\eta,Q'}(X, f'_0)=J_{(M_n\times M_n)\rtimes\sigma'}^{(Q_n\times Q_n)\rtimes\sigma'}((1_n,A),\Theta_v). 
\end{equation}

\subsubsection{Unramified base change}

Let $B\in GL_n(E)$ be such that $B\ov{B}$ is a regular semi-simple element in $GL_n(F)$ and denote $w:=\val_F(\det(B\ov{B}))$. We shall define $\Xi_w\in\CC_c^\infty(GL_n(E))$ as in the proof of \cite[Lemma 5.18]{MR3414387}. Let $\Xi_w$ be the characteristic function of the subset of $x\in GL_n(E)$ satisfying $|x|\leq1$ and $\val_F(\det(x\ov{x}))=w$. 

We shall consider the nontrivial Galois conjugation $\sigma$ on $\Res_{E/F} GL_{n,E}$. Denote by $GL_n(E)_{B,\sigma}$ the twisted (by $\sigma$) centralizer in $GL_n(E)$ of $B$. Recall the (unramified) base change homomorphism (see \cite[\S4.2 of Chapter 1]{MR1007299} for example)
$$ bc_{E/F}: \CH(GL_n(E), GL_n(\CO_E))\ra\CH(GL_n(F), GL_n(\CO_F)) $$
described via the Satake transform by $f(z)\mapsto f(z^2)$, where $\CH(GL_n(E), GL_n(\CO_E))$
denotes the corresponding spherical Hecke algebra. Note that $\Xi_w\in\CH(GL_n(E), GL_n(\CO_E))$. 

\begin{lem}\label{bclemcas2}
 We have
$$ \Psi_w=bc_{E/F}(\Xi_w). $$
\end{lem}

\begin{proof}
This is essentially included in \cite[Corollary 3.7]{MR1382478}. 
Via the Satake isomorphism, it suffices to prove that $bc_{E/F}(\Xi_w)$ and $\Psi_w$ have the same orbital integrals at any regular element in the diagonal torus $A_n(F)$ of $GL_n(F)$. From \cite[Theorem 4.5 in Chapter 1]{MR1007299}, we reduce ourselves to comparing the twisted (by $\sigma$) orbital integral of $\Xi_w$ at $\beta\in A_n(E)$ such that $\beta\ov{\beta}$ belongs to $A_n(F)$ and is regular with the orbital integral of $\Psi_w$ at regular elements in $A_n(F)$. The former is computed in \cite[the first case in p. 139]{MR1382478}, while the latter is computed in \cite[the first case in p. 137]{MR1382478}. 
\end{proof}

Suppose additionally that $B$ belongs to the Levi subgroup $M_n(E)$. The twisted (by $\sigma$) weighted orbital integral of $\Xi_w\in\CC_c^\infty(GL_n(E))$ at $B$ is defined by
$$ J_{(\Res_{E/F} M_{n,E})\rtimes\sigma}^{(\Res_{E/F} Q_{n,E})\rtimes\sigma}(B,\Xi_w):=\int_{GL_n(E)_{B,\sigma}\bs GL_n(E)} \Xi_w(x^{-1}B\sigma(x)) v_{(\Res_{E/F} M_{n,E})\rtimes\sigma}^{(\Res_{E/F} Q_{n,E})\rtimes\sigma}(x) dx. $$

\begin{coro}\label{corbclemcas2}
For $B\in M_n(E)$ such that $A=B\ov{B}$ belongs to $M_n(F)$ and is regular semi-simple in $GL_n(F)$, 
we have 
$$ J_{(\Res_{E/F} M_{n,E})\rtimes\sigma}^{(\Res_{E/F} Q_{n,E})\rtimes\sigma}(B,\Xi_w)=J_{M_n}^{Q_n}(A,\Psi_w). $$
\end{coro}

\begin{proof}
It results from Lemma \ref{bclemcas2} and \cite[Theorem IV.5.2]{MR1339717} for the (unramified) base change $E/F$. 
\end{proof}

Let $Y=B\in\fm(F)\cap\fs_{rs}(F)$. By Remark \ref{twwt}.2), we have
\begin{equation}\label{equationwoi2}
J_M^Q(Y, f_0)=J_{(\Res_{E/F} M_{n,E})\rtimes\sigma}^{(\Res_{E/F} Q_{n,E})\rtimes\sigma}(B,\Xi_w). 
\end{equation}

\subsubsection{A reduction formula}

We fix Haar measures on $M_{Q'}(F)\cap H'(F)$ and $N_{Q'}(F)\cap H'(F)$ such that $\vol(M_{Q'}(F)\cap H'(\CO_F))=\vol(N_{Q'}(F)\cap H'(\CO_F))=1$. Then for $f^{H'}\in\CC_c^\infty(H'(F))$, we have (see \cite[\S4.1]{MR546593})
$$ \int_{H'(F)} f^{H'}(x) dx=\int_{M_{Q'}(F)\cap H'(F)} \int_{N_{Q'}(F)\cap H'(F)} \int_{H'(\CO_F)} f^{H'}(mnk) dkdndm. $$
We choose the Haar measure on $\fn_{Q'}(F)\cap\fh'(F)$ compatible with that on $N_{Q'}(F)\cap H'(F)$ under the exponential map. We choose the same Haar measure on four copies of $\fn_{Q_n}(F)$ in $\fn_{Q'}(F)=\mat(\fn_{Q_n}(F),\fn_{Q_n}(F),\fn_{Q_n}(F),{\fn_{Q_n}(F)})$, and equip $\fn_{Q'}(F)\cap\fs'(F)$ with the product measure. Then $\vol(\fn_{Q'}(F)\cap\fs'(\CO_F))=1$. 

Let $X\in\fm'(F)\cap\fs'_{rs}(F)$. We may define a distribution $J_{M'}^{\eta, M_{Q'}}(X, \cdot)$ on $\CC_c^\infty(\fm_{Q'}(F)\cap\fs'(F))$ as in Definition \ref{deflocwoi}.2). It appears as a product of distributions of the form of $J_{M'}^{\eta, G'}(X, \cdot)$ in lower ranks. 
As in \cite[\S3.2]{MR3414387}, we define the Weyl discriminant factor by
$$ |D^{\fm_{Q'}\cap\fs'}(X)|_F:=|\det(\ad(X)|_{\fm_{Q'}/\fm_{Q', X}})|_F^{1/2}>0, $$
where $\fm_{Q', X}$ denotes the centralizer of $X$ in $\fm_{Q'}$. 
For $f'\in\CC_c^\infty(\fs'(F))$ which is invariant under $\Ad(H'(\CO_F))$, we define its constant term $f'_{Q'}\in\CC_c^\infty(\fm_{Q'}(F)\cap\fs'(F))$ by
$$ f'_{Q'}(Z):=\int_{\fn_{Q'}(F)\cap\fs'(F)} f(Z+U) dU, \forall Z\in\fm_{Q'}(F)\cap\fs'(F). $$
Let $f_0^{M_{Q'}}\in\CC_c^\infty(\fm_{Q'}(F)\cap\fs'(F))$ be the characteristic function of $\fm_{Q'}(F)\cap\fs'(\CO_F)$. Then $(f'_0)_{Q'}=f_0^{M_{Q'}}$. 

\begin{prop}\label{redfor}
Let $X\in\fm'(F)\cap\fs'_{rs}(F)$. For all $f'\in\CC_c^\infty(\fs'(F))$ which is invariant under $\Ad(H'(\CO_F))$, we have
$$ J_{M'}^{\eta, Q'}(X,f')=|D^{\fs'}(X)|_F^{-1/2}|D^{\fm_{Q'}\cap\fs'}(X)|_F^{1/2} J_{M'}^{\eta, M_{Q'}}(X, f'_{Q'}). $$
\end{prop}

\begin{proof}
We apply the change of variables $x=mnk$ to $x\in H'(F)$, where $m\in M_{Q'}(F)\cap H'(F)$, $n\in N_{Q'}(F)\cap H'(F)$ and $k\in H'(\CO_F)$. Notice that $v_{M'}^{Q'}(x)=v_{M'}^{M_{Q'}}(m)$. Since $E/F$ is an unramified extension, the restriction of $\eta(\det(\cdot))$ to $H'(\CO_F)$ is trivial. Recall that $\vol(H'(\CO_F))=1$ and that $H'_X\subseteq M'\cap H'$ for $X\in\fm'(F)\cap\fs'_{rs}(F)$. We deduce that
$$ J_{M'}^{\eta, Q'}(X,f')=\int_{H'_X(F)\bs M_{Q'}(F)\cap H'(F)} \int_{N_{Q'}(F)\cap H'(F)} f'(\Ad(mn)^{-1}(X))\eta(\det(m)) v_{M'}^{M_{Q'}}(m) dndm. $$

By \cite[Lemma 8.1]{2019arXiv190407102L}, for $Z:=\Ad(m^{-1})(X)\in\fm_{Q'}(F)\cap\fs'_{rs}(F)$, the map
$$ N_{Q'}(F)\cap H'(F)\ra\fn_{Q'}(F)\cap\fs'(F), n\mapsto\Ad(n^{-1})(Z)-Z $$
is an isomorphism of $F$-analytic varieties. From the proof of \cite[Proposition 6.3.(ii)]{MR3414387}, its Jacobian is
$$ c(X):=|D^{\fs'}(X)|_F^{1/2}|D^{\fm_{Q'}\cap\fs'}(X)|_F^{-1/2}>0. $$
Then
\[\begin{split}
 J_{M'}^{\eta, Q'}(X,f')=&c(X)^{-1} \int_{H'_X(F)\bs M_{Q'}(F)\cap H'(F)} \int_{\fn_{Q'}(F)\cap\fs'(F)} f'(\Ad(m^{-1})(X)+U) \eta(\det(m)) v_{M'}^{M_{Q'}}(m) dUdm \\
=&c(X)^{-1} \int_{H'_X(F)\bs M_{Q'}(F)\cap H'(F)} f'_Q(\Ad(m^{-1})(X)) \eta(\det(m)) v_{M'}^{M_{Q'}}(m) dm \\
=&c(X)^{-1} J_{M'}^{\eta, M_{Q'}}(X, f'_{Q'}). 
\end{split}\]
\end{proof}

\subsubsection{End of the proof}

\begin{lem}\label{eoplem}
If $v$ is odd, then 
$$ \Psi_v=0. $$
\end{lem}

\begin{proof}
This is essentially included the proof of \cite[Proposition 3.7 and Corollary 3.7]{MR1382478}. In fact, our assertion is equivalent to \cite[the first line in p. 138]{MR1382478} since $E/F$ is unramified. But we shall also give a direct proof as follows. 

Let $
g\in GL_n(F)$. By the change of variables $x_2=gx^{-1}$, we obtain
$$ \Psi_v(g)=\int_{GL_n(F)} \Phi_v(x_2,x_2^{-1}g)\eta(\det(x_2))dx_2=\eta(\det(g))\int_{GL_n(F)} \Phi_v(gx^{-1},x)\eta(\det(x))dx. $$
For all $x_1,x_2\in GL_n(F)$, we notice that
$$ \Phi_v(x_1,x_2)=\Phi_v(x_2^t,x_1^t), $$
where the transpose of $x\in GL_n(F)$ is denoted by $x^{t}$. 
Therefore, we have
$$ \int_{GL_n(F)} \Phi_v(gx^{-1},x)\eta(\det(x))dx=\int_{GL_n(F)} \Phi_v(x^t,(x^t)^{-1} g^t)\eta(\det(x))dx. $$
By the change of variables $x^t\mapsto x$, we see that the last integral is equal to $\Psi_v(g^t)$. Thus
$$ \Psi_v(g)=\eta(\det(g))\Psi_v(g^t). $$
Because $\Psi_v\in \CH(GL_n(F), GL_n(\CO_F))$, by Cartan decomposition, we have
$$ \Psi_v(g^t)=\Psi_v(g). $$
Then
\begin{equation}\label{eqeoplem}
 \Psi_v(g)=\eta(\det(g))\Psi_v(g). 
\end{equation}

Suppose that $v$ is odd. 
We see from the definition that $\Psi_v(g)=0$ unless $\val_F(\det(g))=v$, in which case we have $\det(g)\notin NE^\times$ since $E/F$ is unramified. Thus $\eta(\det(g))=-1$ in this case, which implies that $\Psi_v(g)=0$ by (\ref{eqeoplem}).  
\end{proof}

\begin{proof}[Proof of Theorem \ref{wfl}]
For (1) in Definition \ref{defstrass}, it suffices to consider $X=\mat(0,1_n,A,0)$ and $Y=B$, where $B\in M_n(E)$ is such that $A=B\ov{B}$ belongs to $M_n(F)$ and is regular semi-simple in $GL_n(F)$. By Corollaries \ref{corbclemcas1} and \ref{corbclemcas2}, we obtain
\begin{equation}\label{equationwoi3}
J_{(M_n\times M_n)\rtimes\sigma'}^{(Q_n\times Q_n)\rtimes\sigma'}((1_n,A),\Theta_w)=J_{(\Res_{E/F} M_{n,E})\rtimes\sigma}^{(\Res_{E/F} Q_{n,E})\rtimes\sigma}(B,\Xi_w). 
\end{equation}
Combining the formulas \eqref{equationwoi1}, \eqref{equationwoi2} and \eqref{equationwoi3}, we obtain
$$ \kappa(X)J_{M'}^{\eta,Q'}(X, f'_0)=J_M^Q(Y, f_0). $$

For (2) in Definition \ref{defstrass}, it suffices to consider $X=\mat(0,1_n,A,0)$ with $A\in M_n(F)$ being regular semi-simple in $GL_n(F)$ such that $\xi(A)\notin NE^\times$ for some $\xi\in X(M_{Q_n})_F$. We still have Corollary \ref{corbclemcas1}. For the case $Q'=G$, we conclude by Lemma \ref{eoplem} since $E/F$ is unramified. We now consider a general $Q'$. Applying the reduction formula (Proposition \ref{redfor}) to $f'_0$, we may write 
\begin{equation}\label{equapplyredfortofun}
 J_{M'}^{\eta,Q'}(X, f'_0)=|D^{\fs'}(X)|_F^{-1/2}|D^{\fm_{Q'}\cap\fs'}(X)|_F^{1/2} J_{M'}^{\eta, M_{Q'}}(X, f_0^{M_{Q'}}). 
\end{equation}
Suppose that
$$ M_{Q'}\simeq GL_{2n_1}\times\cdot\cdot\cdot\times GL_{2n_l} $$
and that
$$ M'\simeq M'_1\times\cdot\cdot\cdot\times M'_l, $$
where $\sum\limits_{i=1}^{l} n_i=n$ and $M'_i$ is an $\omega$-stable Levi subgroup of $GL_{2n_i}$ for $1\leq i\leq l$. We have
$$ f_0^{M_{Q'}}=f'_{0,1}\otimes\cdot\cdot\cdot\otimes f'_{0,l} $$
and
$$ X=(X_1,\cdot\cdot\cdot,X_l), $$
where $f'_{0,i}$ (resp. $X_i$) is an analogue of $f'_0$ (resp. $X$) when $n$ is replaced by $n_i$ for $1\leq i\leq l$. Then
$$ J_{M'}^{\eta, M_{Q'}}(X, f_0^{M_{Q'}})=\prod_{i=1}^{l} J_{M'_i}^{\eta, GL_{2n_i}}(X_i, f'_{0,i}). $$
Our condition on $A$ and the special case $Q'=G$ above tell us that at least one factor $J_{M'_i}^{\eta, GL_{2n_i}}(X_i, f'_{0,i})$ in the above product vanishes. Thus $J_{M'}^{\eta,Q'}(X, f'_0)=0$ by (\ref{equapplyredfortofun}). 
\end{proof}

\bibliography{References}
\bibliographystyle{plain}

\medskip

\begin{flushleft}
Universit\'{e} de Paris, Sorbonne Universit\'{e}, CNRS, Institut de Math\'{e}matiques de Jussieu-Paris Rive Gauche, F-75013 Paris, France \\
\medskip
E-mail: huajie.li@imj-prg.fr \\
\end{flushleft}

\end{document}